\documentclass[12pt]{article}

\usepackage{tipa}
\usepackage{amsfonts}
\usepackage{amsmath}
\usepackage{amssymb}
\usepackage{mathtools}
\usepackage{authblk}

\usepackage{enumitem}
\usepackage{natbib}
\usepackage{graphicx}
\usepackage{xr}
\usepackage{cancel}
\usepackage{url}
\usepackage{xcolor}
\setcitestyle{round}

\newcommand{\vect}[1]{\boldsymbol{#1}}
\newcommand{\norm}[1]{\left\lVert#1\right\rVert}
\newcommand{\abs}[1]{\left\lvert#1\right\rvert}
\newcommand{\rank}[1]{\ensuremath{r\left(#1\right)}}
\newcommand{\Med}{\operatorname{\mathbb{M}ed}}
\newcommand{\bP}{\mathbb{P}}
\newcommand{\E}{\mathbb{E}}
\newtheorem{definition}{Definition}
\newtheorem{example}{Example}
\newtheorem{proposition}{Proposition}
\newtheorem{lemma}{Lemma}
\newtheorem{corollary}{Corollary}
\newtheorem{remark}{Remark}
\newenvironment{proof}[1][Proof]{\noindent\textbf{#1.} }{\hfill \rule{0.5em}{0.5em}}

\begin{document}

\title{Central subspace data depth}
\author[1]{Giacomo Francisci}
\author[1]{Claudio Agostinelli}

\affil[1]{Department of Mathematics, University of Trento, Trento, Italy \texttt{giacomo.francisci@unitn.it}, \texttt{claudio.agostinelli@unitn.it}}

\date{\today}

\maketitle

\begin{abstract}
  Statistical data depth plays an important role in the analysis of multivariate data sets. The main outcome is a center-outward ordering of the observations that can be used both to highlight features of the underlying distribution of the data and as input to further statistical analysis. An important property of data depth is related to symmetric distributions as the point with the highest depth value, the center, coincides with the point of symmetry. However, there are applications in which it is more natural to consider symmetry with respect to a subspace of a certain dimension rather than to a point, i.e.\ a subspace of dimension zero. We provide a general framework to construct statistical data depths which attain maximum value in a subspace, providing a center-outward ordering from that subspace. We refer to these data depths as central subspace data depths. Moreover, if the distribution is symmetric with respect to a subspace, then the depth is maximized at that subspace. We introduce general notions of symmetry about a subspace for distributions, study the properties of central subspace data depths and provide asymptotic convergence for the corresponding sample versions. Additionally, we discuss connections with projection pursuit and dimension reduction. An application based on custom data fraud detection shows the importance of the proposed approach and strengthens its potential. \\

\end{abstract}

\section{Introduction}
\label{sec:introduction}

Statistical data depths \citep{liu1990,zuoserfling2000} are bounded functions $d(\vect{x}, F)$ that are non-increasing along any ray departing from the center (i.e.\ the point with the highest depth) yielding a center-outward ordering of points in $\mathbb{R}^m$ with respect to a distribution function $F$. This allows the definition of central regions as sets of points with depth higher than a threshold that can be calibrated according to their probability content and thus provide quantiles for multivariate distributions \citep{zuoserfling2000c}. Using data depths several descriptive statistics can be defined, e.g.\ measures of outlyingness \citep{liu1993,liu1999} and robust location measures that coincide with the point of symmetry for symmetric distributions \citep{zuoserfling2000b}. In the last years, this concept has been extended to functional spaces \citep{chakraborty2014,nagy2016,nieto2016,gijbels2017}. Applications include but are not limited to classification and outlier detection \citep{ghosh2005,chen2008,li2012,hubert2015,hubert2017}. In some applications, it is more natural to consider symmetry with respect to a subspace of a certain dimension rather than to a point, i.e.\ a subspace of dimension zero. As an example, we consider European Union (EU) foreign trade data. EU member states report weights and prices of products imported to the member state from a state outside the EU. Of particular interest for European authorities are misdeclarations. Specifically, a low unit price may indicate undervaluation of the product to avoid import duties. Figure \ref{figure_pod33} (left) shows an example of such a data set, labeled as POD 33, see Section \ref{sec:data_analysis} for more details. We clearly observe a linear structure, with the data lying along straight lines. Therefore, in this context, it is more natural to consider a central subspace of dimension one (i.e., a straight line) rather than a single point.

We propose a notion of data depth that measures centrality with respect to a given central subspace of dimension $0 \leq p \leq m-1$, where the choice $p=0$ corresponds to the usual depth for points. We call this central subspace data depth. The orthogonal subspace of dimension $q=m-p$ is obtained via the minimization of a dispersion measure based on data depth \citep{romanazzi2009}. Early works on measures of scatter include those of \citet{wilks1960,bickel1976,bickel1979,eaton1982} and \citet{oja1983}, who provides a formal definition of scatter measure. \citet{giovagnoli1995} studies dispersion orderings and \citet{zuoserfling2000d} compares the scatter measures of \citet{bickel1976,bickel1979,eaton1982,oja1983} with the scatter measure given by the volume (i.e.\ Lebesgue measure) of depth regions \citep{liu1999}. Recently, \citet{wang2019} generalizes this scatter measure by allowing a general measure on $\mathbb{R}^m$. \citet{paindaveine2018} introduce halfspace depth for symmetric and positive definite matrices and provide, along the lines of \citet{liu1990} and \citet{zuoserfling2000}, desirable properties for a depth for scatter.

To highlight the advantage of our generalization we compare, in Figure \ref{figure_pod33}, the data depth ($p=0$, left panel) with the central subspace data depth ($p=1$, right panel), here $m=2$. All points are colored based on a gray scale according to their depth values. On the left panel, the data point of maximum depth is highlighted in yellow, while the data points corresponding to quantiles of order $0.5$ or lower are in green. Similarly, in the right panel, the straight line of maximum depth is in yellow. Since $m=2$ and $p=1$, the remaining dimension is $q=1$ and we can distinguish between lower and upper quantiles. Points lying in a line with quantiles of order smaller than $0.25$ or larger than $0.75$ are in green. Since the interest is in possible undervaluation of the prices, we also highlight, in blue the points lying on a line with quantiles of order between $0.95$ and $0.975$ and, in red, the points lying in a line with quantiles of order higher than $0.975$. Given the focus on the median price of the product and deviations from this value, the central subspace data depth (right panel) clearly provides a better ordering than the data depth (left panel). In particular, observations on the bottom-left and top-right close to the straight line of maximum depth possess a high depth value even if they are not in the center of the data cloud.
\begin{figure}
\centering{  
  \includegraphics[width=0.48\textwidth]{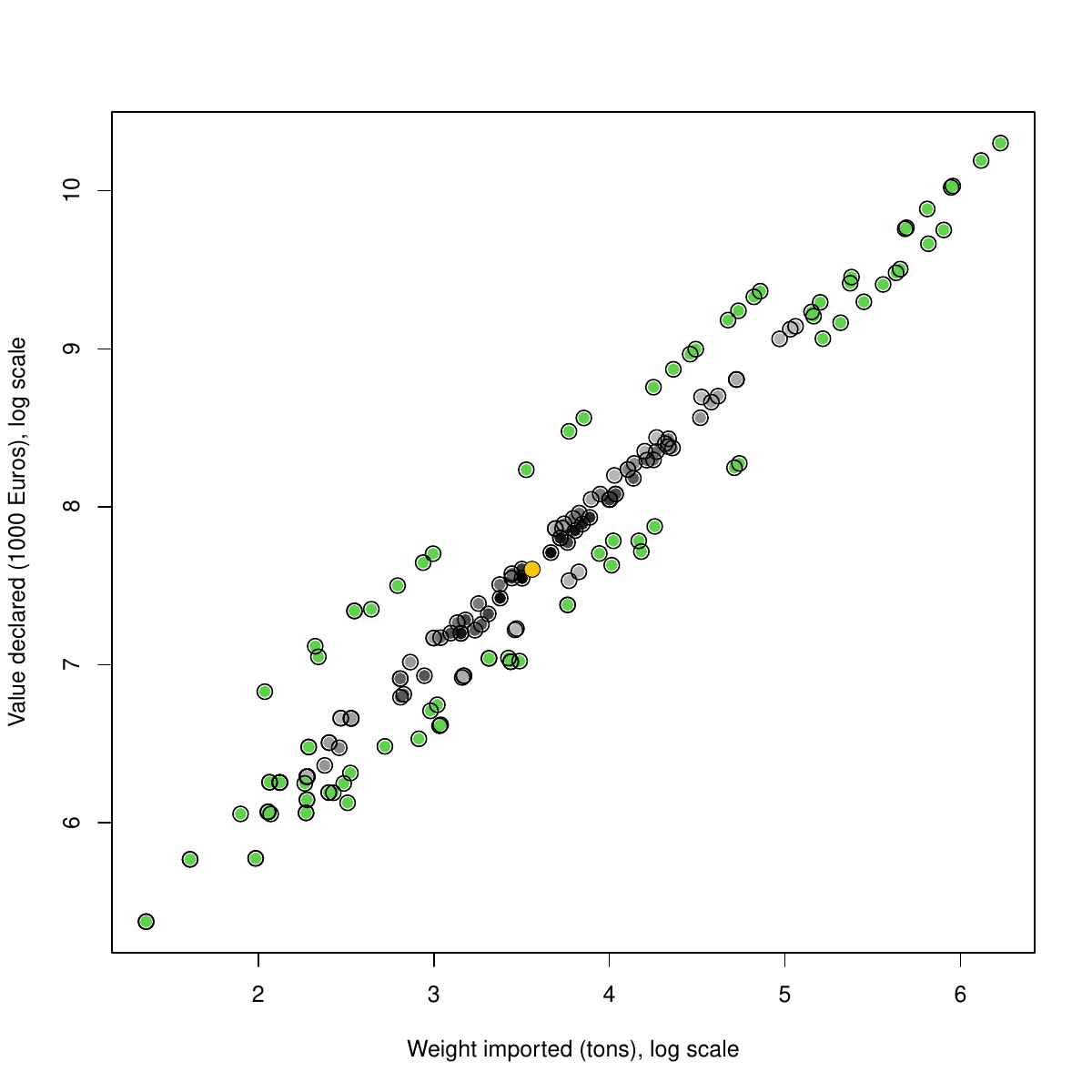}
  \includegraphics[width=0.48\textwidth]{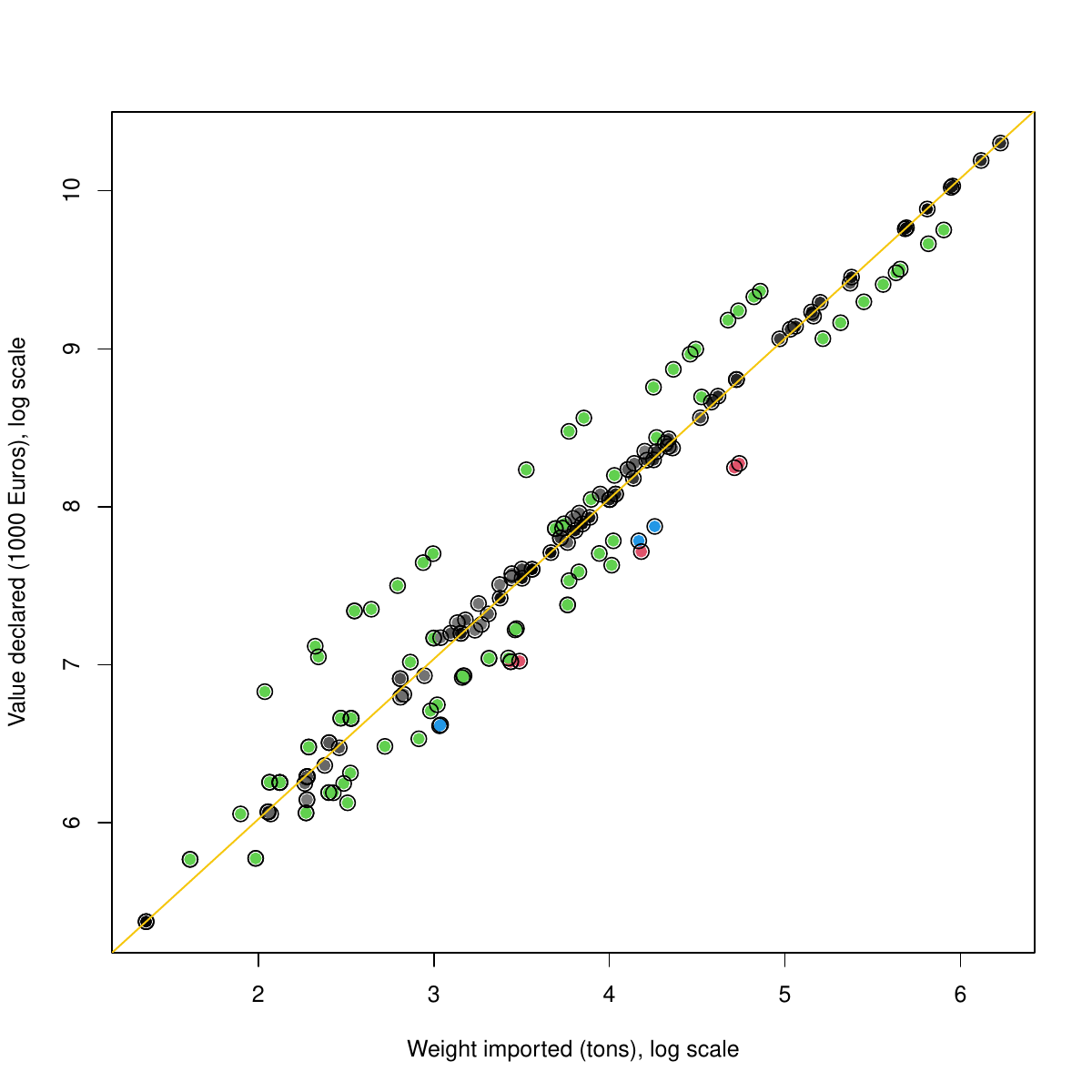}
}
\caption{POD 33 data set. Weights and prices are in log scale. Usual data depth  (left panel), central subspace data depth (right panel). Maximum depth is highlighted in yellow, value of the depth in gray scale for the central region with $0.5$ probability content. Points in green are in the outer region. In the right panel we further mark points with quantiles between $0.95$ and $0.975$ (blue) and those in a region with quantile order greater than $0.975$ (red). The analysis is performed using halfspace depth.}
\label{figure_pod33}
\end{figure}

Our proposed method is related to orthogonal regression in that it identifies an orthogonal direction in which the data are least dispersed. It also connects to dimension reduction and projection pursuit as the dispersion measure provides a measure of data spread in different directions. Projection pursuit is the systematic search for interesting linear projections of multivariate data \citep{friedman1974,jones1987}. These projections are obtained by optimizing a projection index. Classical examples of projection indices are entropy and kurtosis, and aim for departure from Gaussianity and spherical symmetry. A related method is Principal Component Analysis (PCA), which searches for directions capturing the largest variation in the data. Projection pursuit can also be used in regression and density estimation (see \citet{huber1985} and the references therein). Other approaches seek projections that reveal clusterings, for instance, linear discriminant analysis (LDA). In this case, departure from unimodality is of interest \citep{krause2005}. In this paper, we establishing equivalence of the minimization (and maximization) of the dispersion measure based on data depth and PCA for elliptically symmetric distributions. Whereas PCA relies on the existence of the covariance matrix, our approach can be fully non-parametric and suitable for distributions with arbitrary shape.

The paper is organized as follows. Section \ref{sec:symmetry} introduces the notation and extends the concepts of symmetry for multivariate distributions with respect to a point to symmetry with respect to a subspace. Some results are further investigated in Section \ref{sm:section_symmetry} of the Supplemental Material. Section \ref{sec:immersion} discusses the concepts of deeply immersion and central subspaces. Section \ref{sec:depthsubspaces} generalizes statistical data depths to central subspace data depths. Section \ref{sec:selection_of_q} discusses the choice of the optimal subspace dimension, including the choice of $p$ and $q$, and illustrates the method by applying it to some simulated data sets. Section \ref{sec:properties_of_sigma} contains several useful properties of the dispersion measure. Section \ref{sec:dimension_reduction} discusses dimension reduction techniques using the dispersion measure. Section \ref{sec:data_analysis} contains the analysis of real data sets with the aim of identifying potential custom fraud declarations and Section \ref{sec:concluding_remarks} provides some concluding remarks. In the Supplemental Material, Section \ref{sm:sec:depthproperties} recalls the usual properties of statistical data depth. Section \ref{sm:sec:dispersion} provides the theoretical background for the proposed generalization of statistical data depths. In particular, the properties of the dispersion measure introduced in Section \ref{sec:immersion} are studied including conditions for finiteness, continuity, existence and uniqueness of subspace minimizers of the dispersion measure as well as finite sample properties. Section \ref{sm:sec:further_examples} contains the analysis of some additional real data sets.

\section{Symmetry}
\label{sec:symmetry}

Let $\vect{X}$ be a random variable in $\mathbb{R}^{m}$ with distribution function $F$. For $0 \le p \le m-1$ and $q=m-p$, we denote by $S_p$ and $S_q$ orthogonal linear subspaces of $\mathbb{R}^m$ with dimension $p$ and $q$, respectively. The set of all such pairs $(S_p, S_q)$ is denoted by $\mathcal{S}_{p,q}$. In particular, $S_{p,\vect{\mu}} = S_p + \vect{\mu}$ and $S_{q,\vect{\mu}} = S_q + \vect{\mu}$ are orthogonal affine subspaces for all $\vect{\mu} \in \mathbb{R}^m$. This second representation is useful in the following for distributions $F$ having a center of symmetry $\vect{\mu} \neq \vect{0}$. An orthonormal basis of a subspace $S_{k}$ of dimension $k$ is represented by a $k \times m$ matrix $B_{k}$ with orthonormal row vectors. In particular, the $m \times m$ matrix $B_{k}^{\top} B_{k}$ is the orthogonal projection from $\mathbb{R}^m$ onto $S_k$. For simplicity, we use in the following the matrix $B_{k}$ instead of the projection $B_{k}^{\top} B_{k}$. Various notions of symmetry have been proposed for multivariate random variables $\vect{X}$ in $\mathbb{R}^m$ \citep{serfling2006}. The most general notion is halfspace symmetry, which is defined as follows.
\begin{definition}[Halfspace symmetry]
$\vect{X}$ is \textbf{halfspace symmetric} about $\vect{\mu} \in \mathbb{R}^m$ if
\begin{equation*}
  \bP(\vect{X} \in H_{\vect{\mu}}) \ge \frac{1}{2}
\end{equation*}
for any closed halfspace $H_{\vect{\mu}}$ with $\vect{\mu}$ on the boundary.
\end{definition}
Other useful notions of symmetry with respect to a point (in decreasing order of generality) are: angular, central, elliptical, and spherical symmetry. For the definition of these notions of symmetry we refer to Subsection \ref{sm:subsec:symmetries} of the Supplemental Material. We extend these concepts of symmetry to a more general setting by introducing symmetry with respect to a subspace.
\begin{definition}[Symmetry with respect to a subspace]
\label{definition_symmetry_in_a_subspace}
$\vect{X}$ is symmetric with respect to the subspace $S_p$ if the random variable $\vect{Y} = B_q \vect{X}$ is symmetric in $\mathbb{R}^q$.
\end{definition}
The definition entails that the random variable $B_q^\top \vect{Y}$ exhibits symmetry within $S_q$. This symmetry is independent of the specific orthonormal bases chosen for $S_p$ and $S_q$, as the property is invariant under rigid-body transformations. Furthermore, symmetry is preserved under the projection $B_{q}$. Specifically, if $\vect{X}$ is symmetric in $\mathbb{R}^m$ at $\vect{\mu}$, then $\vect{Y} = B_q \vect{X}$ is also symmetric in $\mathbb{R}^q$ at $B_q \vect{\mu}$ or all aforementioned notions of symmetry. Detailed proofs can be found in Propositions \ref{sm:prop:sym:sub:spherical}--\ref{sm:prop:sym:sub:halfspace} of the Supplemental Material. 

\section{Deeply immersion, central subspace and data depth}
\label{sec:immersion}

In this section we discuss the concepts of deeply immersion and central subspace, which refer to the subspace where the distribution is least dispersed and its complement. This framework allows us to define a data depth that is maximized at the central subspace, and it is therefore referred to as the central subspace data depth (see Section \ref{sec:depthsubspaces} below). It is well known that statistical data depths $d(\cdot , \cdot)$ extend the concept of location (median) in a univariate space to the multivariate setting (multivariate median, i.e.\ deepest point). However, data depths can also serve as a scalar dispersion measure for multivariate random variables. An example of such a dispersion measure is given below.

\begin{definition}[\citet{romanazzi2009}]
\label{def:dispersion}
The dispersion measure $\sigma(F)$ is the (Lebesgue) integral of the data depth $d(\vect{x}, F)$, that is
\begin{equation*}
\sigma(F) = \int_{\mathbb{R}^m} d(\vect{x}, F) \ d\vect{x}.
\end{equation*}
\end{definition}
Notice that $\sigma(F)$ does not require the evaluation of any location parameter. See Section \ref{sm:sec:dispersion} in the Supplemental Material for a more complete treatment of dispersion measures based on statistical data depths. Dispersion measures can be used to identify subspaces where a random variable shows low variability. This leads to the following definition.

\begin{definition}[Deeply immersion] \label{def:DeeplyImmersion}
Let $\vect{X}$ be a random variable in $\mathbb{R}^m$ with distribution function $F$. For $(S_p, S_q) \in \mathcal{S}_{p,q}$ consider the random variable $\vect{Y} = B_q \vect{X} \sim F_{B_q}$ and
\begin{equation*}
(\hat{S}_p, \hat{S}_q) = \arg\min_{(S_p, S_q) \in \mathcal{S}_{p,q}} \sigma(F_{B_q}).
\end{equation*}
We say that $\vect{X}$ is deeply immersed in the subspace $\hat{S}_q$ with respect to the dispersion measure $\sigma(\cdot)$.
\end{definition}
The subspaces $(\hat{S}_p, \hat{S}_q)$ are independent of the specific bases $B_p$, $B_q$ used to perform projections. It should be noted, however, that the subspace $\hat{S}_q$ which minimizes $\sigma(F_{B_q})$ is not necessarily unique. Let $\hat{\mathcal{S}}_{p,q}$ denote the set of all such minimizers of $\sigma(F_{B_q})$ and define $\tilde{\mathcal{S}}_p = \cap_{(S_p,S_q) \in \hat{\mathcal{S}}_{p,q}} S_p$. Deeply immersion identifies a subspace $S_q$, where the dispersion of $\vect{X}$ is minimized. The corresponding orthogonal subspace, $S_p$, is referred to as the central subspace.

\begin{definition}[Central subspace]
A subspace $S_p$ is a \textbf{central subspace} of dimension $p$ if there exists $S_q \subset \mathbb{R}^m$ such that $(S_p,S_q) \in \hat{\mathcal{S}}_{p,q}$. The subspace $\tilde{\mathcal{S}}_p$ is the \textbf{core of the central subspaces}. 
\end{definition}

\begin{example}[Spherical symmetry]
\label{exe:spherically}
If $\vect{X}$ is spherically symmetric about $\vect{\mu} \in \mathbb{R}^m$, then any subspace $S_{1}$ of $\mathbb{R}^m$ is a central subspace of dimension $1$ and $\vect{0}$ is the core of the central subspaces. The subspaces $S_{1,\vect{\mu}}$ are also  central subspaces and their core is $\vect{\mu}$. For a convenient interpretation of these subspaces we prefer this second parametrization.
\end{example}

\begin{example}[Elliptical symmetry]
\label{exe:elliptically}
If $\vect{X}$ is elliptically symmetric about $\vect{\mu} \in \mathbb{R}^m$ and the largest eigenvalue of its covariance matrix has multiplicity one, then, its first principal component is the central subspace $\hat{S}_{1,\vect{\mu}}$ of dimension one centered at $\vect{\mu}$, and it corresponds to the core of the central subspaces of dimension one.
\end{example}
Proposition \ref{proposition:pca} in Section \ref{sec:dimension_reduction} formalizes Examples \ref{exe:spherically}-\ref{exe:elliptically} and generalizes them to subspaces of arbitrary dimension. In the next section, we introduce a class of statistical data depths for which the center is an appropriate central subspace, referred to as central subspace data depths.

\subsection{Central subspace data depths}
\label{sec:depthsubspaces}

Consider a subspace $S_{q}$ of dimension $q$ with basis $B_q$. We define the class $\mathcal{S}_{B_q}$ of $p$-dimensional subspaces in $\mathbb{R}^m$ orthogonal to $S_q$ as $\mathcal{S}_{B_q}=\{ S_{B_q}(\vect{y}) \, : \, \vect{y} \in \mathbb{R}^q \}$, where $S_{B_q}(\vect{y})=B_q^{-1} \vect{y}=\{ \vect{x} \in \mathbb{R}^m \, : \, \vect{y} = B_q \vect{x} \}$. Note that $\mathcal{S}_{B_q}$ is independent of the specific basis $B_q$ chosen for $S_{q}$. Indeed, any alternative basis can be expressed as $U B_q$, for some orthogonal $q \times q$ matrix $U$, and $S_{UB_q}(U\vect{y})=S_{B_q}(\vect{y})$ for $\vect{y} \in \mathbb{R}^q$. Using the previous notation we have $S_{p}=S_{B_q}(\vect{0})$ and $S_{p,\vect{\mu}}=S_{B_q}(B_q\vect{\mu})$. As it is easy to see from Examples \ref{exe:spherically}-\ref{exe:elliptically}, the set of central subspaces is not invariant under affine transformations. Consequently, it is appropriate to require the following properties for a central subspace data depth $d_S(\cdot , \cdot)$.

\begin{enumerate}[label=(\textbf{P\arabic*})]
\item \label{PropDepthAffineInvariance} Invariance. \\
(i) Location and scale invariance: $d_S( S_{B_q}(\vect{y}), F) = d_S(a S_{B_q}(\vect{y}) + \vect{b}, F_{a \vect{X}+\vect{b}})$, for any $a \neq 0$ and $\vect{b} \in \mathbb{R}^m$; \\
(ii) Rotation and reflection invariance: $d_S(S_{B_q}(\vect{y}), F) = d_S(U S_{B_q}(\vect{y}), F_{U \vect{X}})$, for any orthogonal matrix $U$.
\item \label{PropDepthMaximalityAtCenter} Maximality at center. If $F$ is symmetric in the subspace $S_q$, that is, there exists $\vect{\nu} \in \mathbb{R}^q$ such that $\vect{Y}=B_q \vect{X}$ is symmetric about $\vect{\nu}$, then, for all $\vect{y} \in \mathbb{R}^q$,
\begin{equation*}
d_S(S_{B_q}(\vect{y}),F) \leq d_{S}(S_{B_q}(\vect{\nu}),F).
\end{equation*}
\item \label{PropDepthMonotonicity} Monotonicity. If $\vect{\nu} \in \mathbb{R}^q$ satisfies $d_S(S_{B_q}(\vect{y}),F) \leq d_{S}(S_{B_q}(\vect{\nu}),F)$, for all $\vect{y} \in \mathbb{R}^q$, then
\begin{equation*}
d_S(S_{B_q}(\vect{y}),F) \leq d_{S}(S_{B_q}(\vect{\nu})+\alpha (S_{B_q}(\vect{y})-S_{B_q}(\vect{\nu})),F), \quad \alpha \in [0,1].
\end{equation*}
\item \label{PropDepthZero} Approaching zero.
\begin{equation*}
d_S(S_{B_q}(\vect{y}),F) \rightarrow 0, \text{ as } \norm{\vect{y}} \rightarrow \infty.
\end{equation*}
\end{enumerate}
In the special case where $p=0$ and $q=m$, properties \ref{PropDepthMaximalityAtCenter}--\ref{PropDepthZero} correspond exactly to the classical statistical data depth properties \ref{sm:PropDepthMaximalityAtCenter}--\ref{sm:PropDepthZero} in Section \ref{sm:sec:depthproperties} of the Supplemental Material. Furthermore, Property \ref{sm:PropDepthAffineInvariance} is satisfied specifically for translation, rotation, reflection, and scaling, rather than for all general affine transformations.

\begin{definition}[Central subspace data depths] \label{definition_depth_for_subspaces}
Let $\vect{X}$ be a random variable in $\mathbb{R}^m$ with distribution $F$ and $d(\cdot , \cdot)$ a statistical data depth in $\mathbb{R}^q$. The depth of $S_{B_q}(\vect{y}) \in \mathcal{S}_{B_q}$ with respect to $F$ is
\begin{equation*}
d_S(S_{B_q}(\vect{y}),F) = d(\vect{y},F_{B_q}).
\end{equation*}
If $B_q$ is the minimizer of $\sigma(F_{B_q})$ in Definition \ref{def:DeeplyImmersion}, we call $d_S(\cdot , \cdot)$ a central subspace data depth.
\end{definition}
It follows immediately from Definition \ref{definition_depth_for_subspaces} and the properties of statistical data depths (see \ref{sm:PropDepthAffineInvariance}--\ref{sm:PropDepthZero} in Section \ref{sm:sec:depthproperties} of the Supplemental Material) that, for any $p$ and $q$, $d_S(\cdot , F)$ satisfies \ref{PropDepthAffineInvariance}--\ref{PropDepthZero}. In particular, if $\vect{\nu} = \arg\max_{\vect{y} \in \mathbb{R}^q} d(\vect{y},F_{B_q})$ is the point in $\mathbb{R}^q$ with maximum depth, then $S_{B_q}(\vect{\nu})$ is the subspace with maximum depth with respect to $d_S(\cdot , F)$. Furthermore, if $F$ is symmetric about $\vect{\mu}$ then $\vect{\nu} = B_q \vect{\mu} = \arg\max_{\vect{y} \in \mathbb{R}^q} d(\vect{y},F_{B_q})$ as $F_{B_q}$ is symmetric about $\vect{\nu}$.

Several choice are possible for the depth function $d(\cdot , \cdot)$. Among the most popular are the halfspace and simplicial depths \citep{zuoserfling2000,mosler2022}. We recall their definition hereafter. The halfspace depth \citep{tukey1975} of a point $\vect{x} \in \mathbb{R}^m$ with respect to the distribution $F$ is given by
\begin{equation*}
d_{H}(\vect{x},F)=\inf_{\vect{u} \in S^{m-1}} \bP(\vect{X} \in H_{\vect{x},\vect{u}}),
\end{equation*}
where $S^{m-1}$ is the unit sphere in $\mathbb{R}^m$ and $H_{\vect{x},\vect{u}} = \{ \vect{y} \in \mathbb{R}^m : \, \langle \vect{u}, \vect{y} \rangle \leq \langle \vect{u}, \vect{x} \rangle \}$ is the closed halfspace with boundary point $\vect{x} \in \mathbb{R}^m$ and outer normal $\vect{u} \in S^{m-1}$. The simplicial depth \citep{liu1990} of $\vect{x}$ with respect to $F$ is
\begin{equation*}
d_\Delta(\vect{x},F) = \bP(\vect{x} \in \Delta[\vect{X}_1,\dots,\vect{X}_{m+1}]),
\end{equation*}
where $\vect{X}_i$ are independent and identically distributed (i.i.d.) with distribution $F$ and $\Delta[\vect{x}_1,\dots,\vect{x}_{m+1}]$ is the closed simplex with vertices $\vect{x}_1,\dots,\vect{x}_{m+1}$.

In the next subsection, we discuss in detail the practical issue of choosing the optimal subspace dimension $p$ and $q=m-p$. Several properties of the central subspace data depth and the dispersion measure $\sigma(\cdot)$ are discussed in Section \ref{sec:properties_of_sigma}.

\subsection{Selection of the optimal subspace dimension} \label{sec:selection_of_q}

In many phenomena, data tend to exhibit some form of symmetry. To extract as much information as possible, one should look for optimal subspace dimension such that, in the orthogonal subspace, the projected random variable exhibits spherical symmetry. This indicates that, within our approach, no further information can be extracted. Let $\mathcal{X}=\{\vect{X}_{1}, \vect{X}_{2}, \dots, \vect{X}_{n}\}$ be a sample of i.i.d.\ random variables in $\mathbb{R}^{m}$ with distribution function $F$. In order to find the optimal dimensions $p^{\ast}$ and $q^{\ast}=m-p^{\ast}$ we perform recursive uniformity test to verify the null hypothesis of spherical symmetry. We begin by taking $p=1$ and draw with repetition $k$ subsamples of equal size $s<n$ from $\mathcal{X}$ independently and uniformly at random. For each subsample $j \in \{1,\dots,k\}$ we compute an optimal direction $B_{p}^{j}=B_{1}^{j}$ as in Definition \ref{def:DeeplyImmersion}. If the distribution $F$ is spherically symmetric, there is no preferred direction and we expect $B_{1}^{j}$ to be uniformly distributed on the unit sphere $S^{m-1}$. To check this, we perform a test of uniformity for the $k$ directions obtained. Several test for uniformity on the unit sphere are available (see for instance \citet{garciaportugues2018,garciaportugues2023,banerjee2024}). We use the multivariate Raylegh test. Rayleigh test statistic is given by $R_{k} = k m \lVert \bar{B} \rVert^{2}$, where $\bar{B}$ is the average of $B_{1}^{j}$, and under the null hypothesis it is asymptotically distributed as $\chi_{m}^{2}$. If the null hypothesis is accepted, we stop and set $p^{\ast}=0$. If the null hypothesis is rejected, we compute the optimal matrices $B_{p}=B_{1}$ and $B_{q}=B_{m-1}$ and obtain a new sample $\tilde{\mathcal{X}}=B_{m-1} \mathcal{X}$ in $\mathbb{R}^{m-1}$. We now perform the uniformity test on the projected sample $\tilde{\mathcal{X}}$. If the null hypothesis is accepted, we stop and set $p^{\ast}=1$, whereas if it is rejected we set $p=2$ and find the optimal matrices $B_{p}=B_{2}$ and $B_{q}=B_{m-2}$ using the full sample $\mathcal{X}$. We repeat the test on the projected sample $\tilde{\mathcal{X}}=B_{m-2} \mathcal{X}$ and continue until the null hypothesis is accepted. Thus, the optimal dimension is the first nonnegative integer $p^{\ast} \in \{0, \dots, m-1 \}$ for which the uniformity test is accepted for the projected sample $\tilde{\mathcal{X}}$ in $R^{q^{\ast}}$, where $q^{\ast}=m-p^{\ast}$. If all the tests are rejected then we set $p^\ast = m-1$ and $q^\ast=1$. The optimal matrices $B_{p^{\ast}}$ and $B_{q^{\ast}}$ are then given by Definition \ref{def:DeeplyImmersion}. To illustrate the procedure, we examine three simulation scenarios corresponding to different distribution functions $F$. Specifically, we set $n=100$, $k=500$, and $s=20$ and consider (i) a multivariate normal in $\mathbb{R}^{3}$ with independent components and standard deviations equal to $1$, $1$, and $5$, respectively; (ii) a multivariate distribution in $\mathbb{R}^{3}$ with independent components, where the first and third components are standard normal's and the second component is uniformly distributed on the interval $[0,0.1]$; and (iii) a multivariate normal in $\mathbb{R}^{5}$ with independent components and standard deviations equal to $1$, $1$, $1$, $5$, and $5$, respectively. The analysis is performed using halfspace depth.

We begin with scenario (i). The hypothesis of spherical symmetry on the full sample $\mathcal{X}$ is rejected by Raylegh test with a p-value smaller than $0.0001$. The optimal matrices are given by
\begin{equation} \label{optimal_matrices_i}
  B_{p^{\ast}}=B_{1} \approx \begin{pmatrix}
    -0.029 \\
    0.021 \\
    -0.999
\end{pmatrix}
\text{ and } B_{q^{\ast}}=B_{2} \approx \begin{pmatrix}
    0.021 & -0.999 \\
    1.000 & 0.021 \\
    0.021 & 0.029
\end{pmatrix},
\end{equation}
which are very close to the true subspaces bases. The hypothesis of spherical symmetry on the projected sample $\tilde{\mathcal{X}} = B_{2} \mathcal{X}$ is accepted with a p-value of about $0.907$. Thus, the optimal dimensions are $p^{\ast}=1$ and $q^{\ast}=2$ and the optimal matrices are given by \eqref{optimal_matrices_i}. The sample $\mathcal{X}$ is plotted in Figure \ref{simulation_i}, where points are colored based on a gray scale according to their (central subspace) depth values. The region of maximum depth is in yellow.
\begin{figure} \label{simulation_i}
\centering{    
  \includegraphics[width=0.32\textwidth]{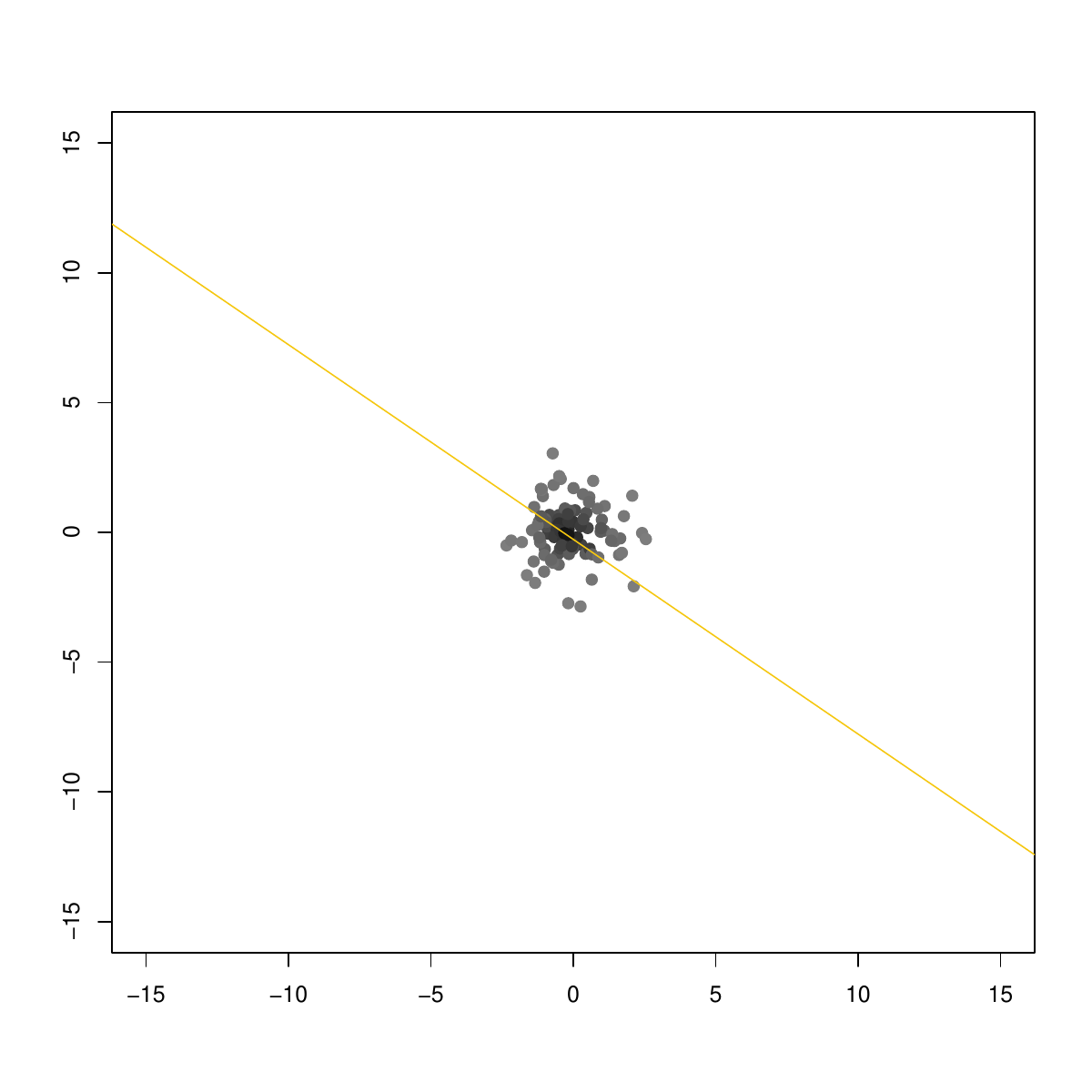}
  \includegraphics[width=0.32\textwidth]{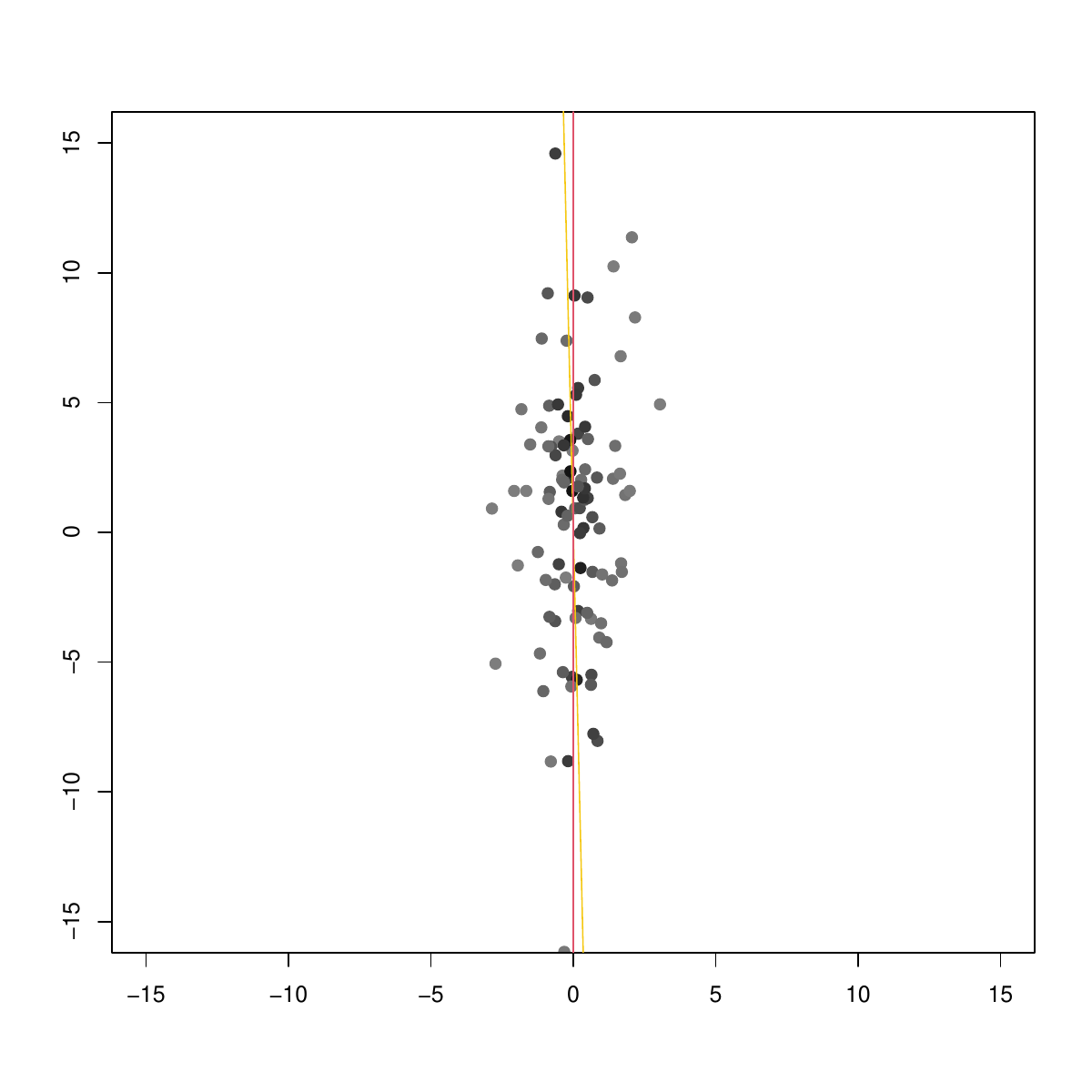}
  \includegraphics[width=0.32\textwidth]{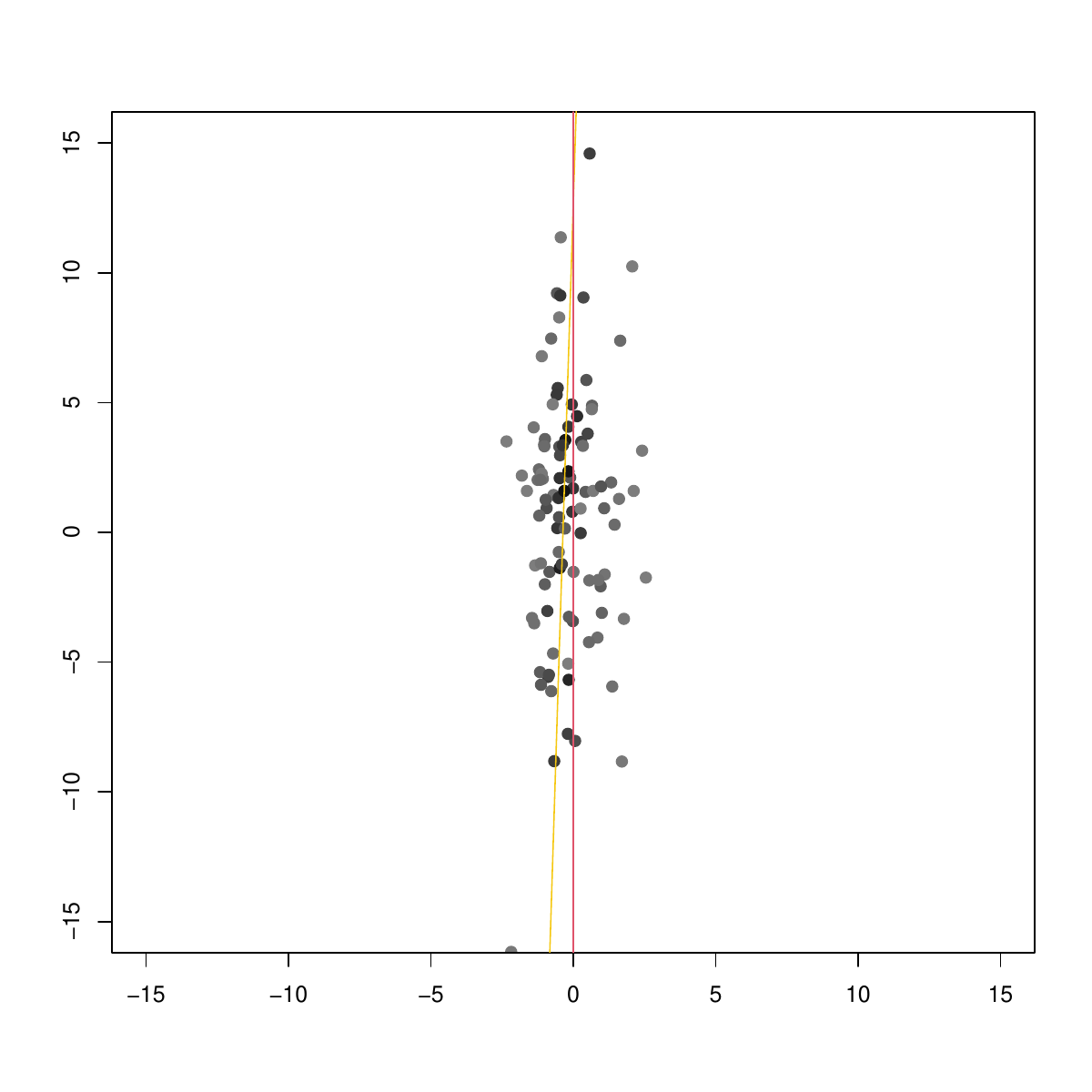}
}
\caption{Sample points in simulation scenario (i). First and second coordinate (left panel), second and third coordinate (central panel), and first and third coordinate (right panel). The region of maximum depth is highlighted in yellow. The true maximal direction is plotted in red.}
\label{simulation_i}
\end{figure}
Turning to scenario (ii), Raylegh test on the full sample $\mathcal{X}$ yields a p-value of $0.041$, whereas testing on the projected sample $\tilde{\mathcal{X}}$ in $\mathbb{R}^{2}$ gives the p-value $0.454$. Thus, the optimal dimensions are $p^{\ast}=1$ and $q^{\ast}=2$. The optimal matrices are given by
\begin{equation*} \label{optimal_matrices_ii}
  B_{p^{\ast}} \approx \begin{pmatrix}
    -0.537 \\
    0.003 \\
    -0.843
\end{pmatrix}
\text{ and } B_{q^{\ast}} \approx \begin{pmatrix}
    0.003 & 0.843 \\
    1.000 & -0.002 \\
    0.002 & 0.537
\end{pmatrix}.
\end{equation*}
The sample points and their depth values are displayed in Figure \ref{simulation_ii}.
\begin{figure} \label{simulation_ii}
\centering{    
  \includegraphics[width=0.32\textwidth]{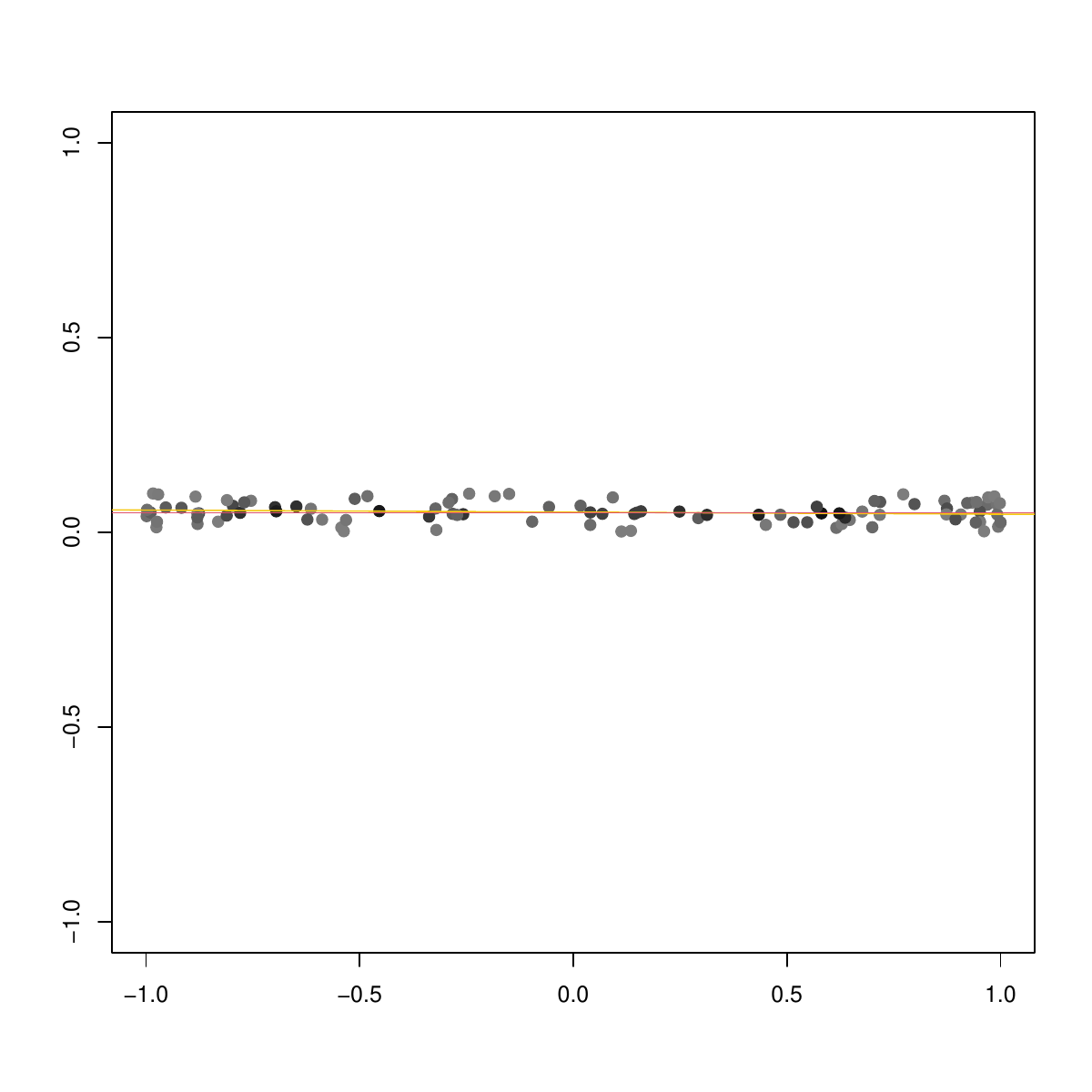}
  \includegraphics[width=0.32\textwidth]{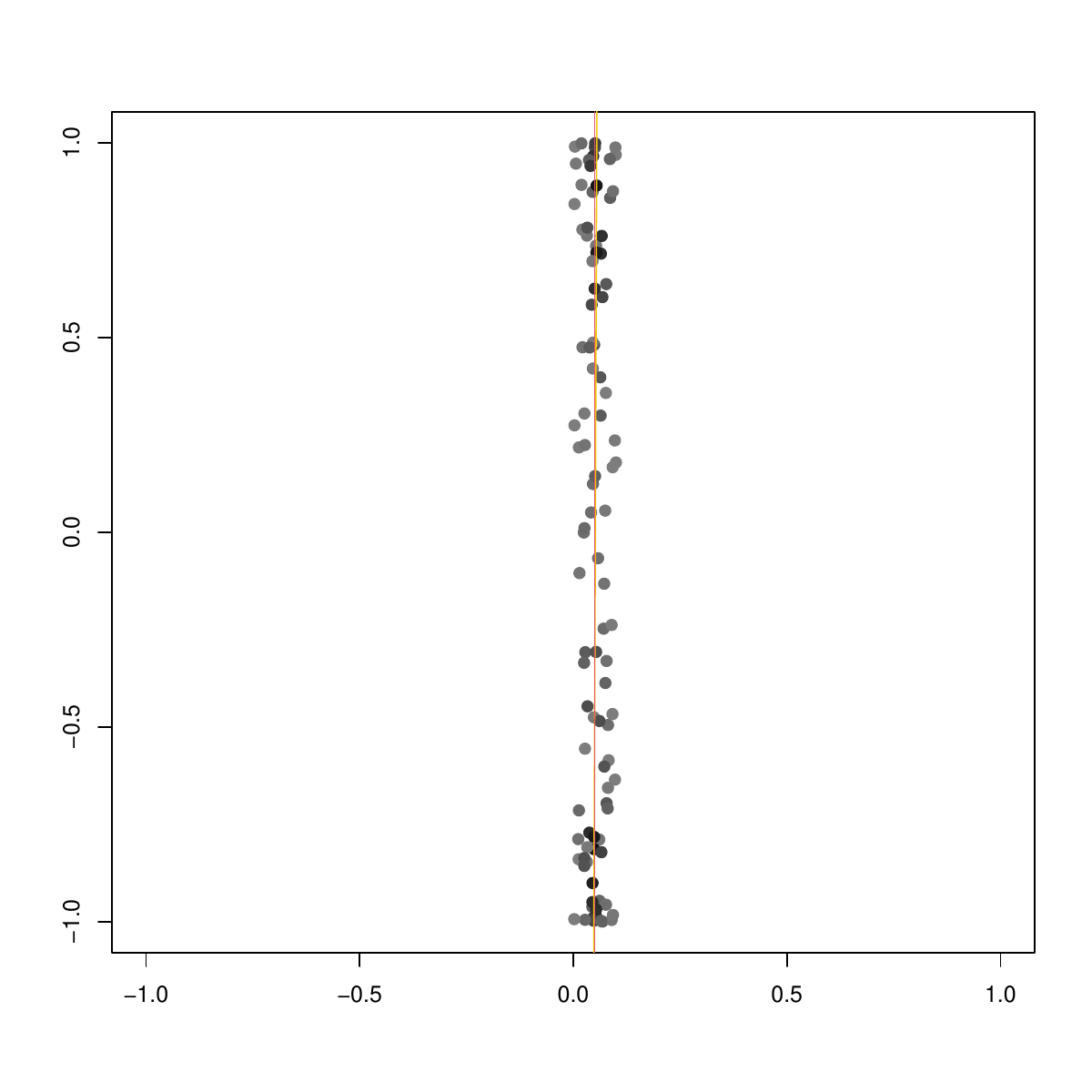}
  \includegraphics[width=0.32\textwidth]{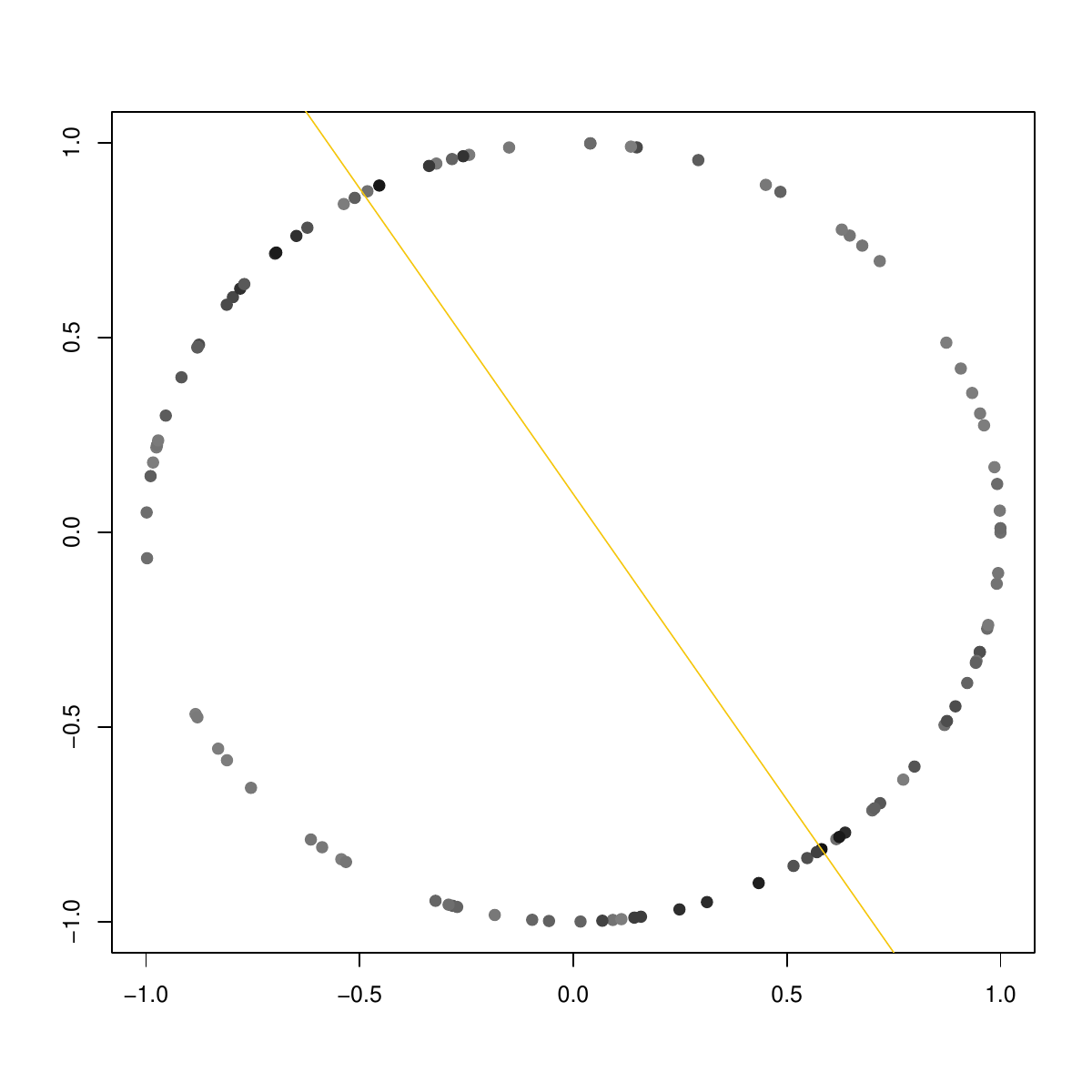}
}
\caption{Sample points in simulation scenario (ii). First and second coordinate (left panel), second and third coordinate (central panel), and first and third coordinate (right panel). The region of maximum depth is highlighted in yellow. The true maximal direction is plotted in red.}
\label{simulation_ii}
\end{figure}
Finally, in scenario (iii), the hypothesis of spherical symmetry on the full sample $\mathcal{X}$ is rejected by Raylegh test with a p-value smaller than $0.0001$. Testing on the projected sample $\tilde{\mathcal{X}}$ in $\mathbb{R}^{4}$ gives a p-value smaller than $0.0001$ and the first accepted test is obtained for $p^{\ast}=2$ and $q^{\ast}=3$, where the p-value is $0.079$. The optimal matrices are given by
\begin{equation*} \label{optimal_matrices_ii}
  B_{p^{\ast}} \approx \begin{pmatrix}
    0.112 & -0.015 \\
    -0.021 & -0.051 \\
    -0.007 & 0.003 \\
    0.124 & 0.991 \\
    -0.986 & -0.126
\end{pmatrix}
\text{ and } B_{q^{\ast}} \approx \begin{pmatrix}
    0.007 & -0.158 & 0.981 \\
    0.003 & -0.986 & -0.157 \\
    1.000 & 0.003 & -0.005 \\
    0.003 & 0.052 & -0.021 \\
    -0.006 & 0.009 & 0.112
\end{pmatrix}.
\end{equation*}

\subsection{Some properties of the dispersion measure and of the central subspace data depth} \label{sec:properties_of_sigma}
We summarize several results regarding the properties of the dispersion measure $\sigma(\cdot)$ and of the central subspace data depth $d_S(\cdot,\cdot)$ when the underlying data depth is either the halfspace or the simplicial depth. Detailed statements and proofs are in the Supplemental Material. The first question is when the dispersion measure $\sigma(\cdot)$ is finite. Propositions \ref{sm:Prop1dDepFin} and \ref{sm:PropDispFinHalfSimp} show that the dispersion measure is finite whenever a moment of order greater than $1$ is finite. Several distributions with infinite variance have a finite dispersion measure based on $\sigma(\cdot)$, e.g., \citet{matsui2016} shows that infinitely divisible distributions (stable distributions, Pareto distribution, geometric stable distribution, Linnik distribution and compound Poisson distributions) have a fractional moment of order smaller than $2$ which is finite. The multivariate t-distribution with $\nu > 0$ degrees of freedom has finite fractional moments of order smaller than $\nu$; of interest is the case $1 < \nu \le 2$ where the variance is infinite but the dispersion measure $\sigma(\cdot)$ is finite (see Proposition \ref{sm:proposition_finiteness_of_fractional_moments_for_t-distribution}).

The second property concerns the continuity of the dispersion measure with respect to the probability distribution. This point is treated in Subsection \ref{sm:subsection_continuity_of_dispersion_measure}. Corollary \ref{sm:CorPolDecay} shows that the dispersion measure is continuous for polynomially decaying probability distributions. Furthermore, continuity is preserved after projection onto subspaces, as it is shown in Corollary \ref{sm:CorPolDecayHalfSimp} of Subsection \ref{sm:subsection_dispersion_measure_in_subspace}. Continuity holds also for empirical probability measures as it is shown in Corollaries \ref{sm:CorPolDecayContSigmatPempHalfSimp} and \ref{sm:CorPolDecayContSigminPempBqHalfSimp} of Subsection \ref{sm:subsection_asymptotic_properties_of_dispersion_measure}.

Third, given a random variable $\vect{X}$ we investigate if it is deeply immersed in a subspace (see Definition \ref{def:DeeplyImmersion}). Existence of such subspace is studied in Subsection \ref{sm:subsection_existence_and_uniqueness_of_minimizers}. If the probability measure of $\vect{X}$ is absolutely continuous with finite mean and decays polynomially, then by Corollary \ref{sm:CorMaxMinDispMeas}  there exist subspaces at which the dispersion measure attains maximum and minimum. When the optimal subspace is unique then Proposition \ref{sm:PropUniqueMaxMinDispMeas} ensures that maximum and minimum of the empirical dispersion measure converge almost surely to their population values.

Fourth, we prove that for elliptically symmetric distributions $F$ the minimization procedure in Definition \ref{def:DeeplyImmersion} is equivalent to Principal Components Analysis (PCA) in the sense that the space $\hat{S}_q$ in Definition \ref{def:DeeplyImmersion} is generated by the last $q$ principal directions obtained from PCA. We also show that this is equivalent to finding the $p$-dimensional subspace $\hat{S}_p$ in which the distribution is more dispersed, that is, finding
\begin{equation} \label{maximization}
  (\hat{S}_p, \hat{S}_q) = \arg\max_{(S_p, S_q) \in \mathcal{S}_{p,q}} \sigma(F_{B_p}).
\end{equation}
See Proposition \ref{proposition:pca} below. However, in general, the minimization procedure in Definition \ref{def:DeeplyImmersion} and the maximization procedure in \eqref{maximization} are not equivalent. To see this, we consider a mixture of four bivariate normal distributions with variance $\eta^{2} I$ and means the vertices $(1,1)$, $(1,-1)$, $(-1,1)$, and $(-1,-1)$ of a square. In Section \ref{sm:sec:mixture_of_normal_distributions} of the Supplemental Material, we compute the dispersion measure based on the halfspace depth of the projected distribution along the directions $B_{1}(u) = (u, \sqrt{1-u^{2}})$ for all $u \in [-1,1]$ and show that it is maximized at $u=-1,0,1$ and minimized at $u=-1/\sqrt{2}, 1/\sqrt{2}$. We deduce that maximization of the dispersion measure yields the directions given by the sides of the square whereas minimization gives the directions given by the diagonals of the square. The same calculations hold (letting $\eta\downarrow 0$) for the uniform distribution on the vertices of the square. The dispersion measure is plotted in Figure \ref{figure_dispersion_mixture_normal} as a function of $u$.

\begin{figure}
\centering{    
  \includegraphics[width=0.69\textwidth]{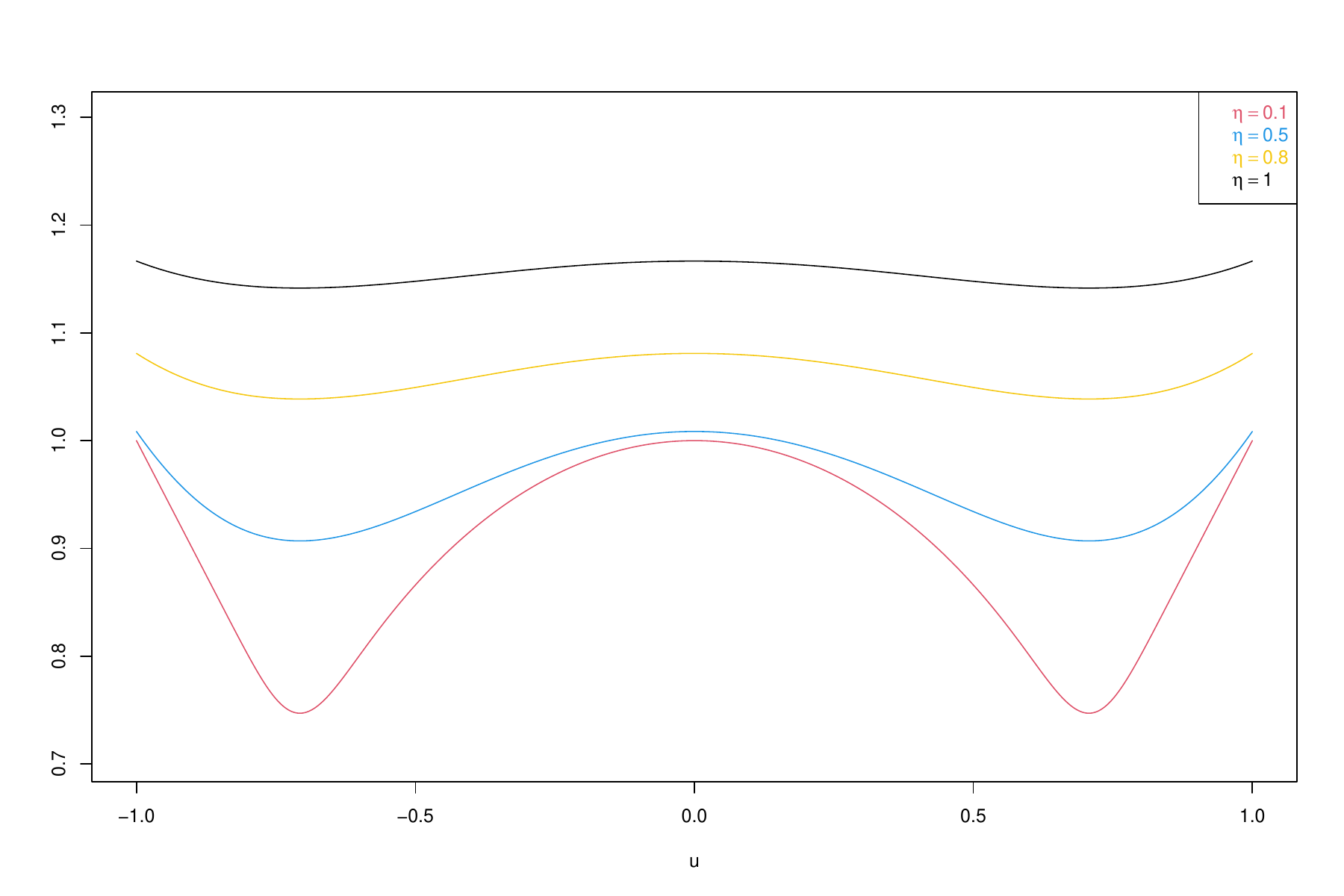}
}
\caption{Mixture of normal distributions. Dispersion measure of the projected distribution, as a function of $u \in [-1,1]$, of a mixture of normal distributions with variance $\eta^{2} I$ and means the vertices $(1,1)$, $(1,-1)$, $(-1,1)$, and $(-1,-1)$ of a square. The analysis is performed using halfspace depth.}
\label{figure_dispersion_mixture_normal}
\end{figure}

We conclude this section by establishing the equivalence of the minimization and maximization of the dispersion measure to PCA for elliptically symmetric distributions.
\begin{proposition} \label{proposition:pca}
  Let $\vect{X} \sim F$ be absolutely continuous and elliptically symmetric with mean $\vect{\mu}$ and covariance matrix $\Sigma$. Consider a statistical data depth $d(\cdot , \cdot)$ with respect to at least elliptical symmetry and the corresponding dispersion measure $\sigma(\cdot)$. If $\sigma(F)<\infty$, then
\begin{equation*}
(\hat{S}_p, \hat{S}_q) = \arg\min_{(S_p, S_q) \in \mathcal{S}_{p,q}} \sigma(F_{B_q})
\end{equation*}
if and only if $\hat{S}_p$ and $\hat{S}_q$ are orthogonal subspaces in $\mathbb{R}^m$ generated by the first $p$ and the last $q$ principal directions obtained from PCA. In turn, this is equivalent to
\begin{equation*}
  (\hat{S}_p, \hat{S}_q) = \arg\max_{(S_p, S_q) \in \mathcal{S}_{p,q}} \sigma(F_{B_p}).
\end{equation*}  
\end{proposition}
\begin{proof}
For any nonnegative scalar $r$, vector $\vect{b} \in \mathbb{R}^{k}$ and $k \times k$ matrix $A$, we define the ellipsoid 
\begin{equation*}
\mathcal{E}(\vect{b}, A, r) = \{ \vect{y} \in \mathbb{R}^{k} \, : \, (\vect{y}-\vect{b})^{\top} A^{-1} (\vect{y}-\vect{b}) = r^{2} \}.
\end{equation*} 
Since $\Sigma$ is symmetric there exists an orthogonal matrix $U$ such that $\Sigma=U^{\top} D U$, where $D$ is the diagonal matrix with diagonal elements the real and positive eigenvalues $(\lambda_i)_{i=1}^{m}$ of $\Sigma$. In particular, the column vectors $(\vect{u}_i)_{i=1}^{m}$ of $U$ are the eigenvectors of $\Sigma$ and they constitute the principal axes of the ellipsoid $\mathcal{E}(\vect{\mu},\Sigma,r)$. Additionally, the rescaled eigenvalues $(r^2 \lambda_i)_{i=1}^m$ yield the square of the semi-axes length of the elliptical surface. Now, the last $q$ principal directions from principal component analysis are the eigenvectors $(\vect{u}_{i_j})_{j=1}^{q}$ associated with the $q$ smallest eigenvalues $(\lambda_{i_j})_{j=1}^{q}$. We assume without loss of generality (w.l.o.g.) that $i_1<i_2<\dots<i_q$ and $\lambda_{i_1} \le \lambda_{i_2} \le \dots \le \lambda_{i_q}$. We show below that the minimization of the dispersion measure $\sigma(F_{B_{q}})$ over all matrix $B_{q}$ with orthonormal row vectors yields once again the vectors $(\vect{u}_{i_j})_{j=1}^{q}$, that is, the minimum is obtained by taking $B_{q} = (\vect{u}_{i_1} \dots \vect{u}_{i_q})^{\top}$. We notice that ties, such as $\lambda_{i_q}=\lambda_{i_{q+1}}$, are possible, in which case, the eigenvectors $\vect{u}_{i_q}$ and $ \vect{u}_{i_{q+1}}$ yield both the same minimal value. However, this is true for both methods. Similarly, one can show that maximization of the dispersion measure $\sigma(F_{B_{p}})$ is achieved for $B_{p} = (\vect{u}_{i_{q+1}} \dots \vect{u}_{i_m})^{\top}$ yielding the final statement.

For any matrix $B_{q}$ with orthonormal row vectors, $B_{q} \vect{X}$ is elliptically symmetric with mean $B_{q} \vect{\mu}$ and covariance matrix $B_{q} \Sigma B_{q}^{\top}$ (see Proposition \ref{sm:prop:sym:sub:elliptical} of the Supplemental Material). By Theorem 3.3 in \citet{zuoserfling2000c} (see also Lemma 3.1 in \citet{liu1993}), there exists a non-increasing function $g$ such that
\begin{equation*}
d(\vect{x},F_{B_{q}}) = g((\vect{x}-B_{q} \vect{\mu})^{\top} (B_{q} \Sigma B_{q}^{\top})^{-1} (\vect{x}-B_{q} \vect{\mu})).
\end{equation*}
Using the coarea formula we obtain that
\begin{equation*}
  \sigma(F_{B_{q}}) = \int_{0}^{\infty} g(r^2) \, H^{q-1}(\mathcal{E}(B_{q}\vect{\mu}, B_{q} \Sigma B_{q}^{\top},r)) \, dr,
\end{equation*}
where $H^{q-1}(\cdot)$ is the $(q-1)$-dimensional Hausdorff measure and is used to measure the surface area of the ellipsoid $\mathcal{E}(B_{q}\vect{\mu}, B_{q} \Sigma B_{q}^{\top},r)$. We notice that the assumption $\sigma(F)<\infty$ implies that $\sigma(F_{B_{q}})<\infty$. Using the translation invariance of the Hausdorff measure and $\Sigma=U^{\top} D U$, we see that
\begin{equation*}
  H^{q-1}(\mathcal{E}(B_{q}\vect{\mu}, B_{q} \Sigma B_{q}^{\top},r)) = H^{q-1}(\mathcal{E}(\vect{0}, B_{q} \Sigma B_{q}^{\top},r)) 
\end{equation*}  
and
\begin{equation*}
  \mathcal{E}(\vect{0}, B_{q} \Sigma B_{q}^{\top},r) = B_{q} U^{\top} \mathcal{E}(\vect{0}, D, r).
\end{equation*}
The matrix $B_{q}^{\prime}=B_{q} U^{\top}$ yielding the minimal surface area is $B_{q}^{\prime} = \left(\vect{e}_{i_1} \dots \vect{e}_{i_q} \right)^{\top}$ and corresponds to the centered ellipsoid
\begin{equation*}
  B_{q}^{\prime} \mathcal{E}(\vect{0}, D, r) = \mathcal{E}(\vect{0}, B_{q}^{\prime} D (B_{q}^{\prime})^{\top}, r) = \{ \vect{y} \in \mathbb{R}^{q} \, : \, \sum_{j=1}^q \lambda_{i_j}^{-1} y_{i_{j}}^2 =r^2 \}.
\end{equation*}
We conclude that $H^{q-1}(\mathcal{E}(B_{q}\vect{\mu}, B_{q} \Sigma B_{q}^{\top},r))$ is minimized by taking $B_{q}= B_{q}^{\prime} U =  \left(\vect{u}_{i_1} \dots \vect{u}_{i_q} \right)^{\top}$.
\end{proof}

\section{Dimension reduction}
\label{sec:dimension_reduction}

Dimension reduction techniques aim to reduce the number of components of a multivariate distribution $F$ in $\mathbb{R}^{m}$ while retaining as much of the relevant features as possible. In particular, PCA searches for the $p$-dimensional subspace of $\mathbb{R}^{m}$ where the projected distribution has maximal variance. A critical aspect of the analysis is the choice of the projecting dimensions $p$ and $q$. In the context of depth functions, the optimal dimension can be directly chosen using the dispersion measure $\sigma(\cdot)$. As before let $\mathcal{X}=\{\vect{X}_{1}, \vect{X}_{2}, \dots, \vect{X}_{n}\}$ be a sample of i.i.d.\ random variables in $\mathbb{R}^{m}$ with distribution function $F$. Assume without loss of generality that $\mathcal{X}$ is contained in a sufficiently small ball centered at the origin (if not, multiply all sample points by a positive scalar). We show in Proposition \ref{sm:prop_monotonicity} of the Supplemental Material that, if the sample is contained in a sufficiently small ball, the maximal dispersion measure is non-increasing with respect to $p$. Thus, we may plot the maximal dispersion for different values of $p \in \{1, \dots, m\}$ and identify the optimal value, for instance, by the elbow method. To illustrate this, we consider the  Iris dataset, which consists of $n=150$ observations with $m=4$ measurements each (sepal length, sepal width, petal length, and petal width) from 3 species (Iris Setosa, Iris Versicolour, and Iris Virginica), which are referred hereafter to as groups 1, 2, and 3. After rescaling the data so that their maximal norm is exactly $r_{H}$ (given in Proposition \ref{sm:prop_monotonicity}), we compute the maximal dispersion measure for all projecting dimensions $p=1,2,3,4$. This is plotted in Figure \ref{figure_iris.max} (left). Since the first value is considerably larger, we take $p=1$, which corresponds to projecting the data in one dimension. The projected data and their (halfspace) depth value are plotted in Figure \ref{figure_iris.max} (right).
\begin{figure}
\centering{
  \includegraphics[width=0.49\textwidth]{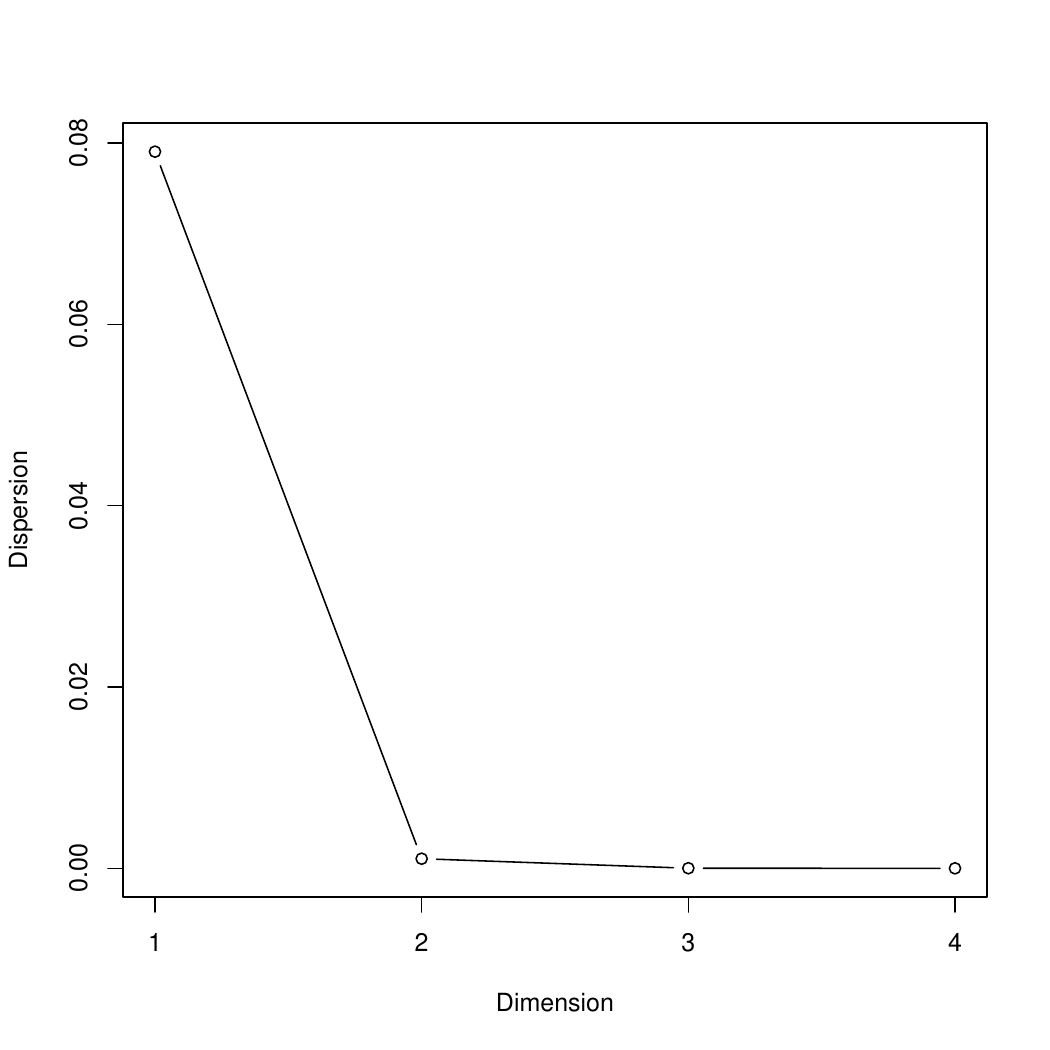}
  \includegraphics[width=0.49\textwidth]{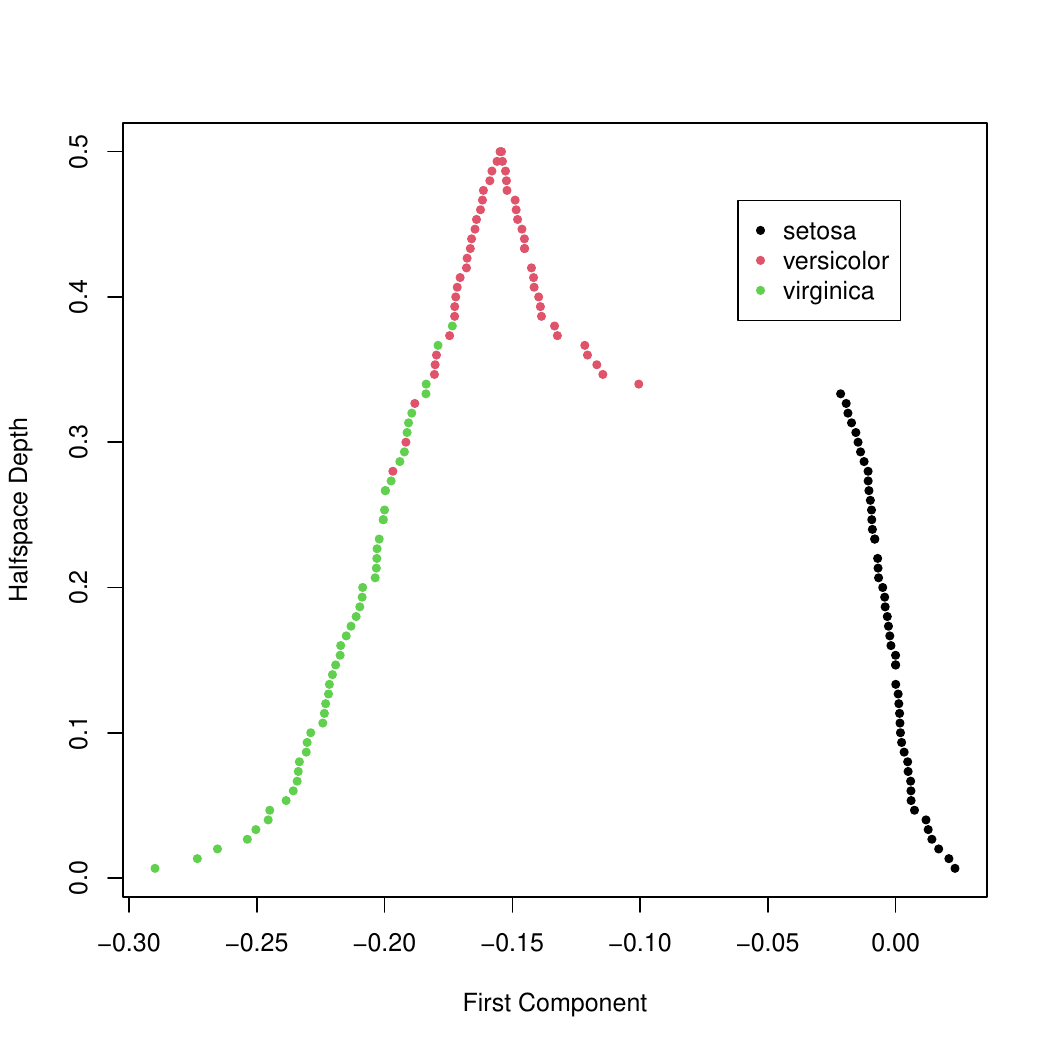}
}
\caption{Iris data set. Maximal dispersion measure for $p=1,2,3,4$ (on the left) and the depth values of the projected data (on the right). The analysis is performed using halfspace depth.}
\label{figure_iris.max}
\end{figure}
We see that the data in group 1 (Iris setosa) are well separated, whereas group 2 and 3 have a small overlap. We perform a hierarchical cluster analysis using the distance between the projected data and the R functions \texttt{hclust} and \texttt{cutree} to identify the 3 groups in the data. The results are reported in Table \ref{table_1} (left) and show correct clustering for 143 out of 150 observations. We repeat the hierarchical clustering procedure using the first component of PCA. The corresponding results are reported in Table \ref{table_1} (right). In this case, 15 observations in group 3 are wrongly assigned to group 2.
\begin{table}[h]
\centering
\begin{tabular}{|c|c|c|c|c|}
\hline
  & $1$ & $2$ & $3$ \\ \hline
$1$ & $50$ & $0$ & $0$ \\ \hline
$2$ & $0$ & $44$ & $6$  \\ \hline
$3$ & $0$ & $1$ & $49$ \\ \hline
\end{tabular} \hspace{1cm}
\begin{tabular}{|c|c|c|c|c|}
\hline
  & $1$ & $2$ & $3$ \\ \hline
$1$ & $50$ & $0$ & $0$ \\ \hline
$2$ & $0$ & $50$ & $0$  \\ \hline
$3$ & $0$ & $15$ & $35$ \\ \hline
\end{tabular}
\caption{Contingency tables for hierarchical clustering of Iris data using one dimensional projection yielding maximal dispersion measure (left) and first component of PCA (right).}
\label{table_1}
\end{table}
This result provides empirical evidence that the procedure based on maximizing the dispersion measure is at least comparable to PCA. For elliptically symmetric distributions, we showed in Proposition \ref{proposition:pca} that the minimization and maximization procedures in Definition \ref{def:DeeplyImmersion} and equation \eqref{maximization} are equivalent to PCA. Further data analysis is discussed in the next section.

\section{Data analysis}
\label{sec:data_analysis}

In this section, we discuss some data sets about products imported to the European Union (EU) from a state outside the EU. They consist of weights and prices of a product (P) from an origin (O) to a destination (D) and are therefore referred to as POD \citep{arsenis2015} followed by a number labeling a specific data set. The joint research center and the European Anti-fraud Office work for the estimation of fair trade prices and price outliers. This facilitates the detection of misdeclarations. In particular, lower unitary prices may signal an import duty fraud and may be further investigated by the authorities in charge. We analyze these datasets on a logarithmic scale using both data depth and central subspace data depth. First, we perform the Rayleigh test for uniformity, as discussed in Section \ref{sec:selection_of_q}. In all cases, the test is rejected with a p-value smaller than the machine's accuracy, indicating that the optimal dimensions are $p=q=1$. Next, we determine the optimal projection and analyze the data using central subspace data depth. In all cases, the data exhibit a correlation of at least $0.98$ in absolute value when projected onto the optimal direction and along the first principal component given by PCA. Figure \ref{figure_pod33} and \ref{figure_pod19} show the data sets POD 33 and POD 19, respectively. Two further data sets (POD 30 and POD 54) are analyzed in the Supplemental Material, Section \ref{sm:sec:further_examples}. Depth values for these data sets were computed using halfspace depth. The results using the simplicial depth are similar and are reported in the Supplemental Material.

\begin{figure}
\centering{    
  \includegraphics[width=0.32\textwidth]{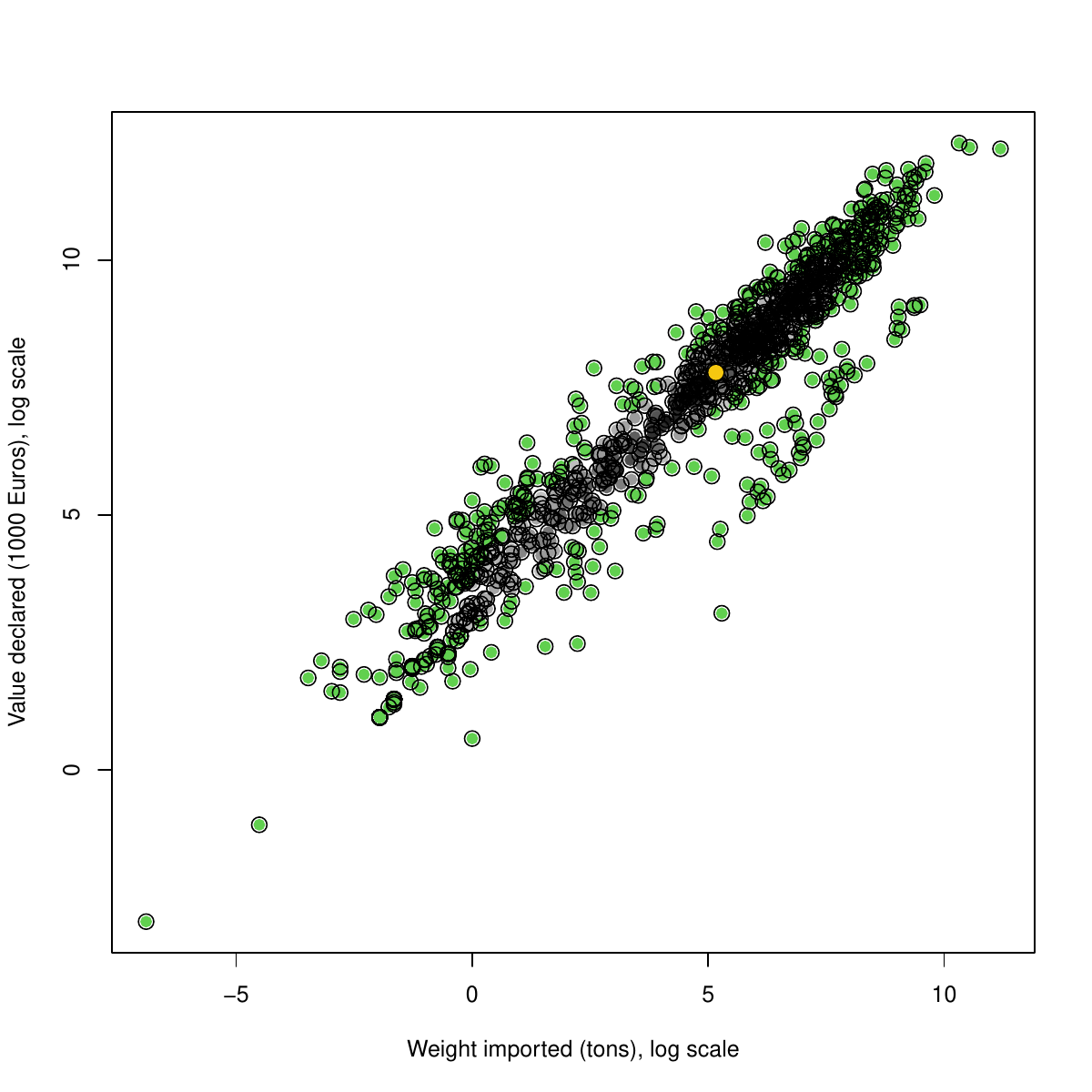}
  \includegraphics[width=0.32\textwidth]{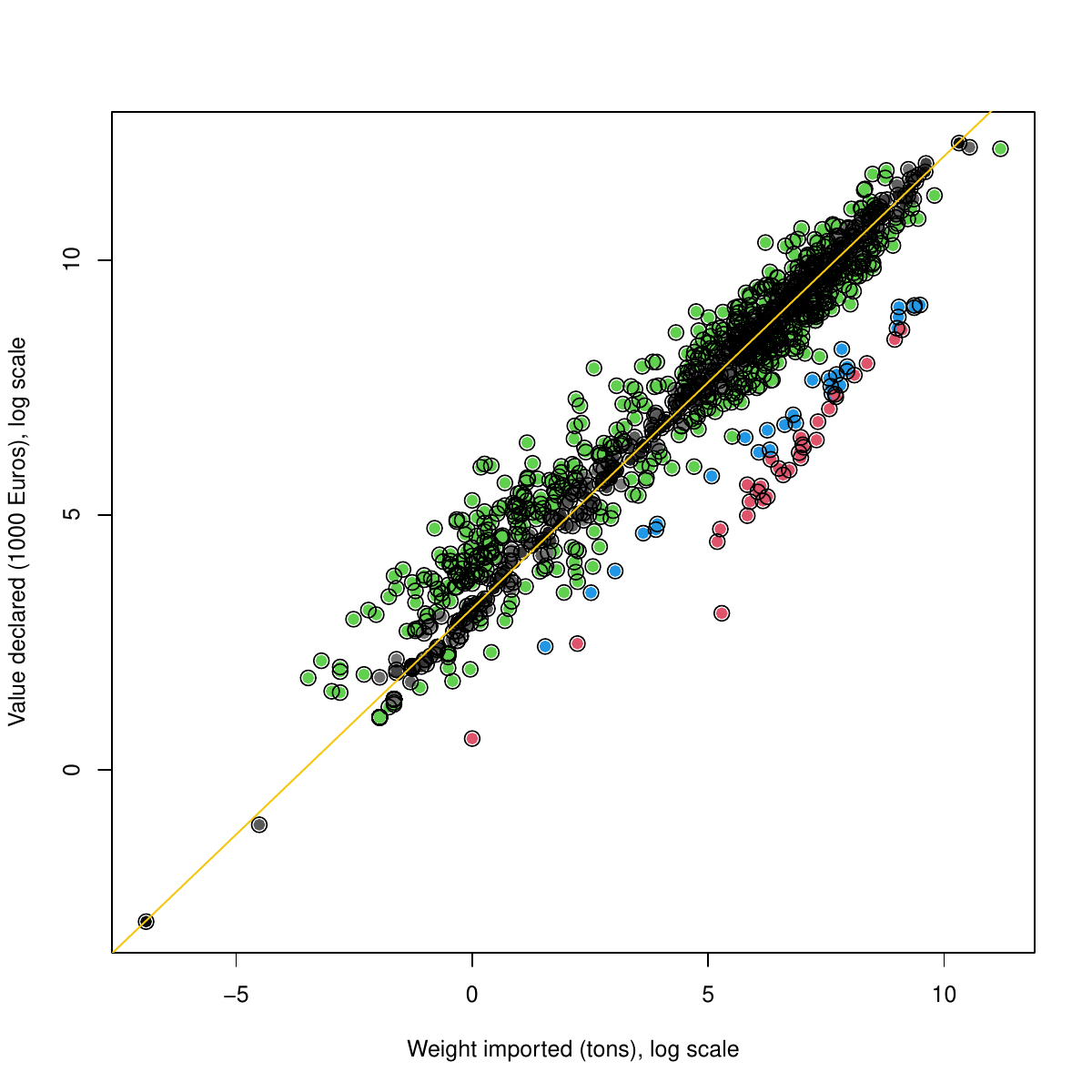}
  \includegraphics[width=0.32\textwidth]{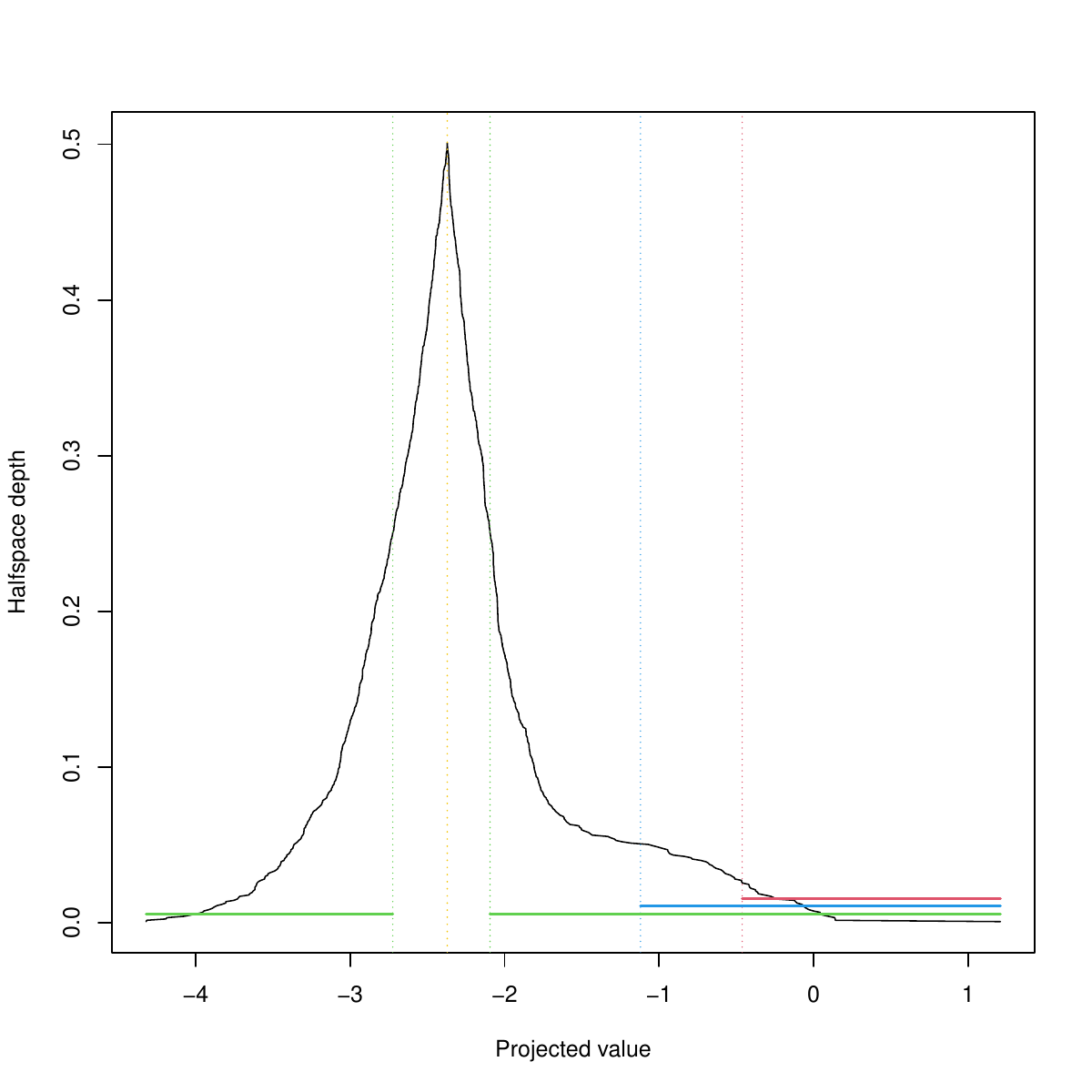}
}
\caption{POD 19 data set. Weights and prices in log scale. Data depth (on the left), central subspace data depth (in the center) and data depth of projected values (on the right). The analysis is performed using halfspace depth.}
\label{figure_pod19}
\end{figure}

In Figure \ref{figure_pod19} (left and center), all points in the central region with $0.5$ probability content are colored based on a gray scale according to their depth values. We highlight in yellow, on the left, the point of maximum depth and, in the center, the subspace of maximum depth. The points with the lowest $50\%$ depth values are shown in green. As we can see from the depth values of the points projected onto the orthogonal subspace, which in this case has dimension $1$, Figure \ref{figure_pod19} (right), this data set has a small inter quantile range. In the center and on the right, we color in blue the points with quantiles of order between $0.95$ and $0.975$ and in red the points with quantiles of order at least $0.975$. These are outliers and may be related to a fraud.

\citet{riani2008} extensively analyzes the fishery data set (FSDA MATLAB toolbox \citep{riani2020}), available in the R package \texttt{fsdaR} \citep{todorov2020} and concerning weights and prices of fish imported in the EU. The data set consists of 677 observations of monthly imports, over a period of three years, from a foreign country to different member states. In particular, the data set contains 35 anomalous flows to a member state (MS) denoted as MS7 and a single anomalous flow to another MS, MS11. Further 11 flows (7 of which to MS2) may also be anomalous. The data set is plotted in Figure \ref{figure_fishery}. The second row is in the original scale, while the first row is in log-scale.
\begin{figure}
\centering{    
  \includegraphics[width=0.49\textwidth]{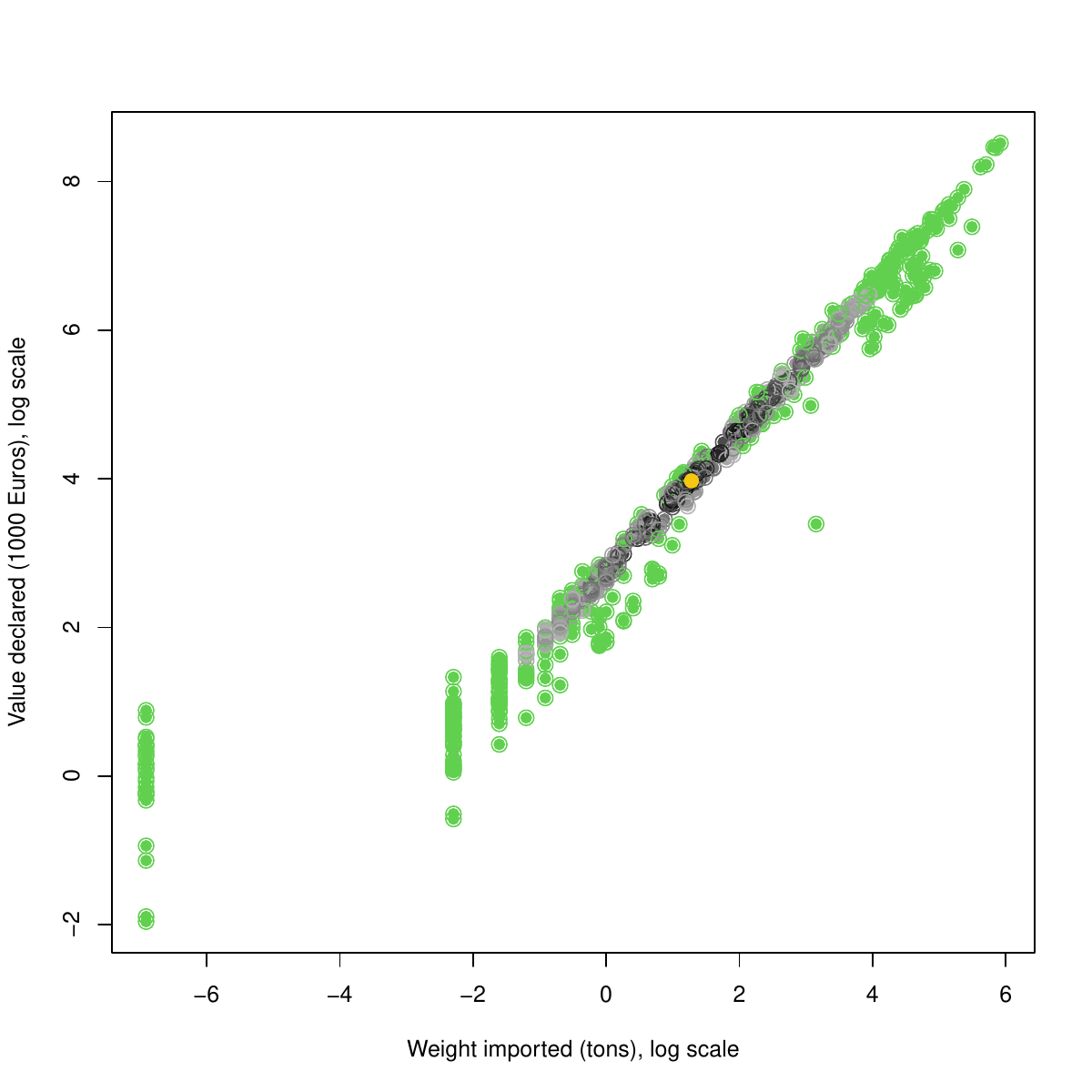}
  \includegraphics[width=0.49\textwidth]{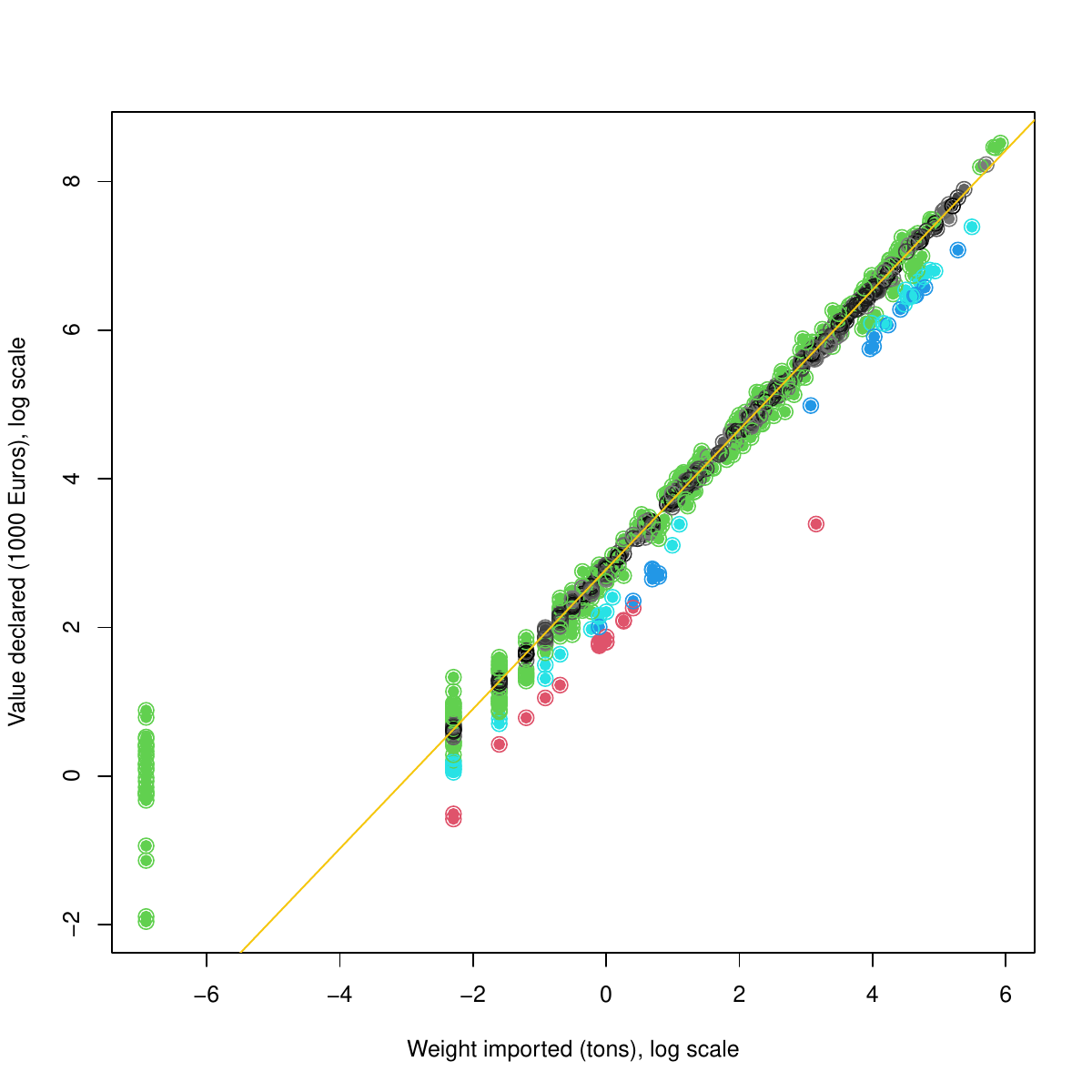}
  \includegraphics[width=0.49\textwidth]{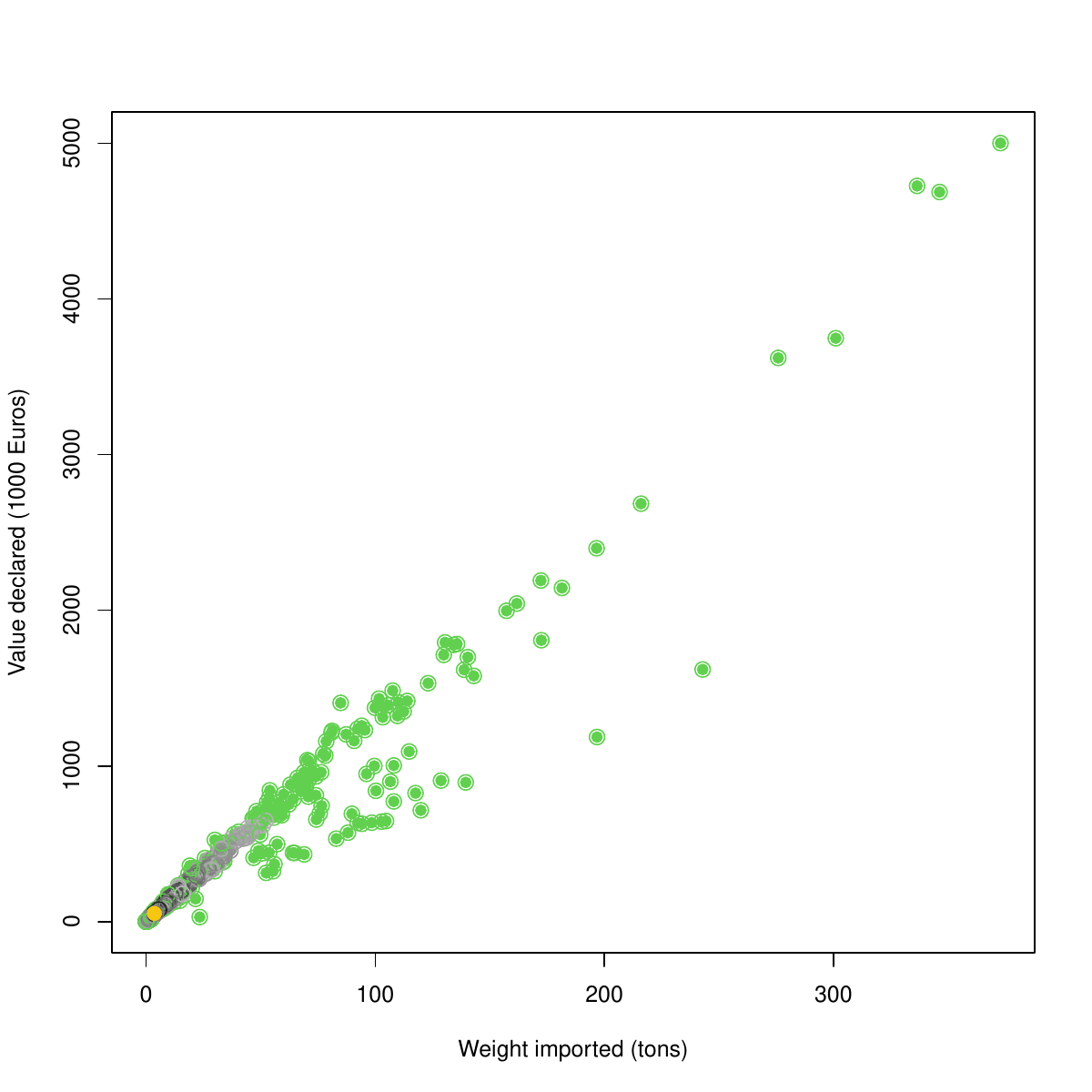}
  \includegraphics[width=0.49\textwidth]{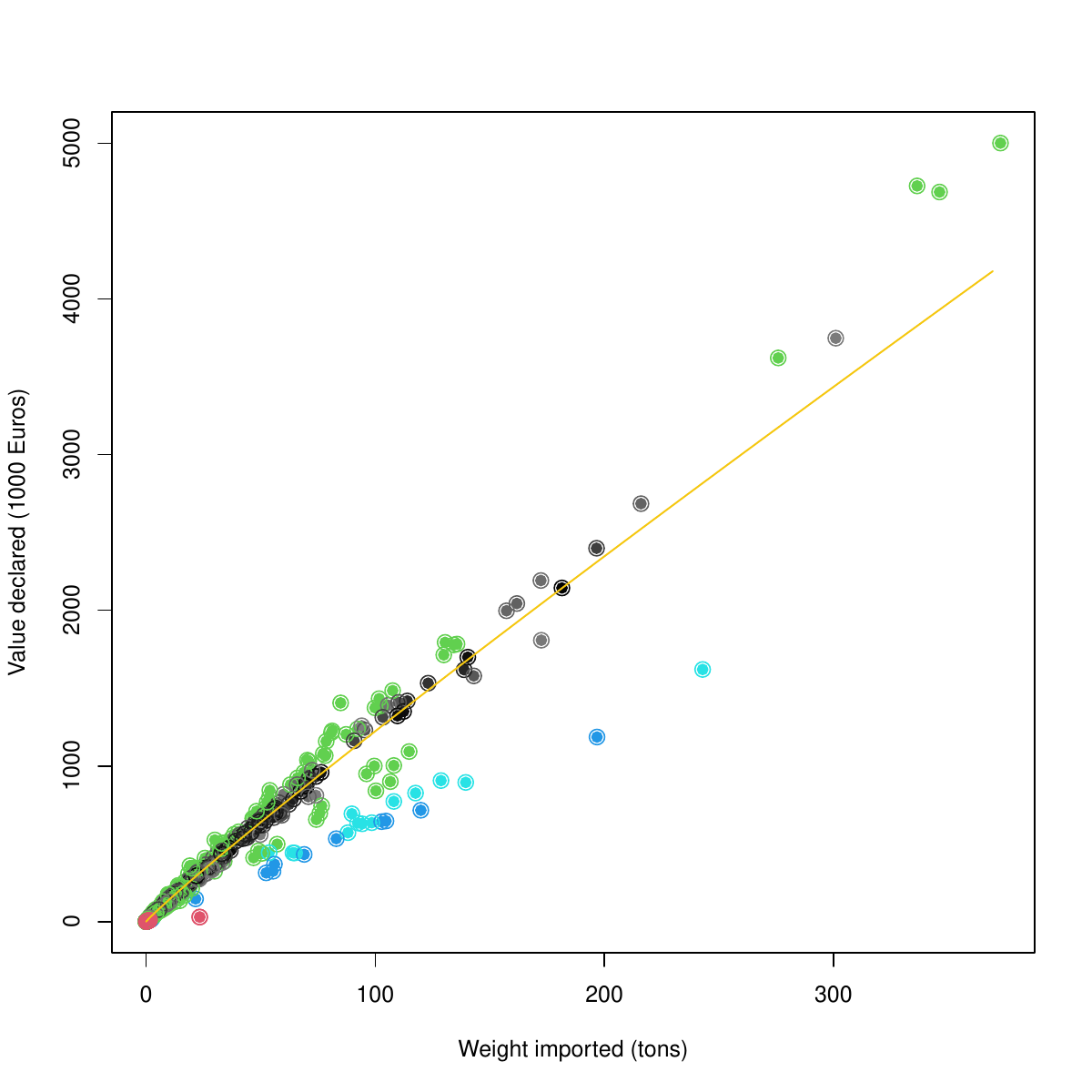}
}
\caption{Fishery data set. In the first row weights and prices are reported in log scale. In the second row the results are back-transformed in the original scale. Data depth (on the left) and central subspace data depth (on the right). The analysis is performed using halfspace depth.}
\label{figure_fishery}
\end{figure}
On the left, the data depth in log-scale is computed. The yellow point is the data point with maximum depth, while the green color highlights the points with the lowest $50\%$ depth values. On the right, the central subspace data depth is computed, the straight line with maximum depth is identified and highlighted in yellow, while the points with the lowest $50\%$ depth values are shown in green. The points with (univariate) quantiles of order between $0.9$ and $0.95$, $0.95$ and $0.975$ or $0.975$ or higher are plotted in cyan, blue or red, respectively (notice that they are a subset of the ``green'' points). These points are identified as possible outliers. A comparison of the plot on the bottom-right with Figures 3, 6 and 9 in \citet{riani2008} shows a good correspondence between the cyan and blue points and those identified in that manuscript. One of the red points corresponds to the flow to MS11, while the other red points are outliers not previously identified.

\section{Concluding remarks}
\label{sec:concluding_remarks}

We introduce a generalization of the concept of statistical data depths that allows to analyze more structured data, in which the role of center is played by a subspace rather than by a point. The new introduced data depths have the same properties of the usual data depths apart from invariance to affine transformations, which is relaxed to location, scale, rotation and reflection invariance. The direction of the central subspace, and its parallel subspaces, is obtained by minimizing a scalar measure of dispersion based on data depth so that the whole procedure is completely based on tools provided by the depth. We study the analytical and asymptotic properties of this generalization and emphasize its potential with the analysis of several real data examples.  We also investigate the use of dispersion measure based on the depth in projection pursuit and dimension reduction.

\newpage

\appendix

\textbf{ \huge Supplemental material for ``Central subspace data depth''}

\section{Introduction}

The Supplemental Material is divided in four main sections. Section \ref{sm:section_symmetry} is about symmetry: Subsection \ref{sm:subsec:symmetries} gives a brief overview of notions of symmetry for multivariate random variables; Subsection \ref{sm:subsec:symmetryinsubspaces} shows that those notions of symmetry are preserved by projections onto a subspace; and Subsection \ref{sm:subsec:symmetrywrtsubspaces} discusses further notions of symmetry with respect to subspaces. In Section \ref{sm:sec:depthproperties}, we recall the definition of statistical data depth and provide some examples from the literature. The dispersion measures $\sigma(\cdot)$ is studied in Section \ref{sm:sec:dispersion}. The (general) notion of dispersion measure is given in Subsection \ref{sm:subsection_dispersion_measures}. Subsection \ref{sm:subsection_properties_of_dispersion_measures} summarizes, from \citet{romanazzi2009}, the most relevant properties of $\sigma(\cdot)$. Subsection \ref{sm:subsection_finiteness_of_dispersion_measure} gives simple conditions for the finiteness of the dispersion measure $\sigma(\cdot)$ based on halfspace and simplicial depth. Subsection \ref{sm:subsection_continuity_of_dispersion_measure} studies continuity of  $\sigma(\cdot)$. These continuity results are extended to subspaces in Subsection \ref{sm:subsection_dispersion_measure_in_subspace}. Subsection \ref{sm:subsection_asymptotic_properties_of_dispersion_measure} contains asymptotic results for the empirical dispersion measure, both on the entire space and on subspaces. In Subsection \ref{sm:subsection_existence_and_uniqueness_of_minimizers}, we show the existence of subspaces that minimize/maximize the dispersion measure and show that, under some conditions, the maxima and minima of the empirical dispersion measure converge to the population maxima and minima. Next, in Subsection \ref{sm:subsec:monotonicity} we show that the dispersion measure is monotone with respect to the projecting dimension. Furthermore, Subsection \ref{sm:sec:mixture_of_normal_distributions} contains an example showing that the minimization procedure in Definition \ref{def:DeeplyImmersion} is not, in general, equivalent to the maximization procedure in \eqref{maximization}. Finally, Section \ref{sm:sec:further_examples} contains further data analysis.

\section{Symmetry}
\label{sm:section_symmetry}

In this section, after a revision of the symmetry of a distribution with respect to a point in a multivariate Euclidean space (Subsection \ref{sm:subsec:symmetries}), we extend these concepts to symmetry of a distribution with respect to subspaces (Subsection \ref{sm:subsec:symmetrywrtsubspaces}). Subsection \ref{sm:subsec:symmetryinsubspaces} studies the connection between symmetry with respect to a point and symmetry with respect to a subspace.

\subsection{Symmetry with respect to a point}
\label{sm:subsec:symmetries}

The center $\vect{\mu} \in \mathbb{R}^m$ of a multivariate random variable $\vect{X} \in \mathbb{R}^m$ is related to the corresponding notion of symmetry. So far, several notions of multivariate symmetry have been used. They include, in increasing order of generality, spherical, elliptical, central, angular and halfspace symmetry \citep{zuo1998,zuoserfling2000c,serfling2006}. For univariate distributions, spherical, elliptical and central symmetry all coincide with the usual notion of symmetry, while a point of halfspace symmetry is a median of the distribution. If the distribution is continuous, the same holds for a point of angular symmetry. Hereafter we provide the definitions of these notions of symmetry. Given two random variables $\vect{X}$ and $\vect{Y}$, we write $\vect{X} \stackrel{d}{=} \vect{Y}$ for ``$\vect{X}$ and $\vect{Y}$ have the same distribution''.

\begin{definition}[Spherical symmetry]
$\vect{X}$ is \textbf{spherically symmetric} about $\vect{\mu} \in \mathbb{R}^m$ if
\begin{equation*}
\vect{X} - \vect{\mu} \stackrel{d}{=} U (\vect{X} - \vect{\mu}),
\end{equation*}
for all orthogonal $m \times m$ matrices $U$.
\end{definition}
Examples of spherically symmetric distributions are multivariate normal distributions with covariance matrix of the form $\sigma^2 I$ and some instances of standard multivariate $t$ and logistic distributions. In particular, the standard $m$-variate $t$-distribution with $k$ degrees of freedom has distribution $k^{1/2} \vect{Z}/S$, with $\vect{Z}$ standard $m$-variate normal and $S$ independently distributed as a chi-square with $k$ degrees of freedom.

\begin{definition}[Elliptical Symmetry]
$\vect{X}$ is \textbf{elliptically symmetric} (or ellipsoidally symmetric) with parameters $\vect{\mu}$ and $\Sigma$ if it is affinely equivalent to a spherically symmetric random variable $\vect{Y}$, that is
\begin{equation*}
\vect{X} \stackrel{d}{=} A \vect{Y} + \vect{\mu} \ ,
\end{equation*}
where the $k \times m$ matrix $A$ satisfies $A^\top A = \Sigma$ with $\rank{\Sigma} = k \le m$.
\end{definition}

For absolutely continuous and elliptically symmetric distributions, the contours of equal density are elliptical in shape. The class of elliptically symmetric distributions is readily seen to be closed under affine transformations and conditioning. For $\vect{Y}$ $m$-variate standard normal $N(\vect{0}, I_m)$, $\vect{\mu} \in \mathbb{R}^m$, and a $k \times m$ matrix $A$, the above definition of elliptically symmetric distributions implies that $\vect{X}$ is $m$-variate normal $N(\vect{\mu}, \Sigma)$ with $\Sigma = A^\top A$. Similarly, for $\vect{Y}$ $m$-variate (spherically symmetric) $t$, $\vect{X}$ is multivariate $t$ with parameters $\vect{\mu}$ and $\Sigma = A^\top A$ and $k$ degrees of freedom.

\begin{definition}[Central symmetry]
$\vect{X}$ is \textbf{centrally symmetric} about $\vect{\mu} \in \mathbb{R}^m$ if
\begin{equation*}
\vect{X} - \vect{\mu} \stackrel{d}{=} \vect{\mu} - \vect{X}.
\end{equation*}
\end{definition}

\begin{definition}[Angular symmetry]
$\vect{X}$ is \textbf{angularly symmetric} about $\vect{\mu} \in \mathbb{R}^m$ if
\begin{equation*}
\frac{\vect{X} - \vect{\mu}}{\norm{\vect{X}-\vect{\mu}}} \stackrel{d}{=} \frac{\vect{\mu} - \vect{X}}{\norm{\vect{X} - \vect{\mu}}},
\end{equation*}
where $\norm{\cdot}$ is the Euclidean norm. Or, equivalently, if $(\vect{X} - \vect{\mu})/\norm{\vect{X}-\vect{\mu}}$ is centrally symmetric.
\end{definition}

\begin{definition}[Halfspace symmetry]
$\vect{X}$ is \textbf{halfspace symmetric} about $\vect{\mu} \in \mathbb{R}^m$ if
\begin{equation*}
\bP(\vect{X} - \vect{\mu} \in H_0) \ge \frac{1}{2},
\end{equation*}
for any closed halfspace $H_0$ with the origin on the boundary.
\end{definition}
We refer to \citet{zuoserfling2000c} for the relationships between different concepts of symmetry and further properties.

\subsection{Symmetry with respect to a subspace}
\label{sm:subsec:symmetrywrtsubspaces}

In this section, we examine further the concept of symmetry with respect to a subspace introduced in Definition \ref{definition_symmetry_in_a_subspace} of the main paper. Let $\vect{X} \in \mathbb{R}^m$ be a random variable, $B_q$ a $q \times m$ matrix with orthonormal row vectors and $S_q$ the subspace generated by $B_q$. Similarly, $B_p$ is a $p \times m$ matrix with orthogonal row vectors, which are also orthogonal to $B_q$, and $S_p$ is the subspace generated by $B_p$. We recall that $\mathcal{S}_{p,q}$ is the set of all such pairs $(S_p, S_q)$. Among the notions of symmetry with respect to a subspace below, the notion of ``symmetry with respect to a subspace'' of Definition \ref{definition_symmetry_in_a_subspace} in the main paper is the weakest, and therefore it is called here ``weak symmetry''.

\begin{definition}[Weak Symmetry]
$\vect{X}$ is weak symmetric with respect to the subspace $S_p$ if the random variable $\vect{Y} = B_q \vect{X}$ is symmetric in $\mathbb{R}^q$.
\end{definition}

\begin{definition}[Mild Symmetry]
Let $(S_p, S_q) \in \mathcal{S}_{p,q}$, $\mathcal{A}_p$ a partition of measurable subsets of $\mathbb{R}^p$, $\vect{Y}_{B_q} = B_q \vect{X}$ and $\vect{Y}_{B_p} = B_p \vect{X}$. $\vect{X}$ is mild symmetric with respect to the subspace $S_p$ and the partition $\mathcal{A}_p$, if the conditional distribution of $\vect{Y}_{B_q}$ given $\vect{Y}_{B_p} \in A$ is symmetric for all $A \in \mathcal{A}_p$, whenever the conditional distribution is well defined.
\end{definition}

\begin{definition}[Strong Symmetry]
  Let $(S_p, S_q) \in \mathcal{S}_{p,q}$, $\vect{Y}_{B_q} = B_q \vect{X}$ and $\vect{Y}_{B_p} = B_p \vect{X}$. $\vect{X}$ is strong symmetric with respect to the subspace $S_p$ if the conditional distribution of $\vect{Y}_{B_q}$ given $\vect{Y}_{B_p} = \vect{y}$ is symmetric for all $\vect{y} \in \mathbb{R}^p$, whenever the conditional distribution is well defined.
\end{definition}

\begin{remark}
Since symmetry in $\mathbb{R}^q$ is invariant to rigid--body transformations, all the above notions of symmetry are invariant to the choice of the orthonormal bases of $S_p$ and $S_q$.
\end{remark}

\begin{remark}
Strong symmetry implies mild symmetry and mild symmetry implies weak symmetry. For $\mathcal{A}_p = \{ \mathbb{R}^p \}$ mild symmetry is equivalent to weak symmetry. For $p=0$ all notions of symmetry are equivalent and coincide with the usual notion of symmetry in $\mathbb{R}^m$.

In our definition of central subspace data depth (see Definition \ref{definition_depth_for_subspaces}) we used weak symmetry. However, the others two notions might be used as well.
\end{remark}

\subsection{Symmetry in a subspace}
\label{sm:subsec:symmetryinsubspaces}

When the subspace $S_p$ has dimension $0$, i.e.\ $p=0$, symmetry with respect to (w.r.t.)\ the subspace is equivalent to symmetry w.r.t.\ a point. In this subsection we show that symmetry w.r.t.\ points in $\mathbb{R}^{m}$ is preserved by the matrix $B_{q}$.

\begin{proposition}
\label{sm:prop:sym:sub:spherical}
If $\vect{X}$ is spherically symmetric about $\vect{\mu} \in \mathbb{R}^m$, then $\vect{Y} = B_q \vect{X}$ is spherically symmetric about $B_q \vect{\mu}$.
\end{proposition}

\begin{proof}
We assume without loss of generality that $\vect{\mu} = \vect{0}$. Since $\vect{X}$ is spherically symmetric $U \vect{X} \stackrel{d}{=} \vect{X}$ for all orthogonal matrices $U$. Write $\vect{X} = (\vect{X}_p, \vect{X}_q)$ and choose $U$ as
\begin{equation*}
U = 
\begin{pmatrix}
U_p & 0 \\
0 & U_q
\end{pmatrix},
\end{equation*}
where $U_p$ and $U_q$ are two orthogonal matrices of dimension $p \times p$ and $q\times q$, respectively. Then, 
\begin{equation}
\label{sm:equ:sphe1}
U_q \vect{X}_q \stackrel{d}{=} \vect{X}_q, 
\end{equation}
that is, the subvector $\vect{X}_q$ is spherically symmetric; and similarly for $\vect{X}_p$. To show that $\vect{Y}$ is spherically symmetric we need to prove that $U_q B_q \vect{X} \stackrel{d} = B_q \vect{X}$ for all orthogonal matrices $U_q$. Let $B = (B_p^\top, B_q^\top)$ and observe that $B$ is orthogonal. Therefore, $B_q B = (0_{q \times p}, I_q)$ and hence
\begin{equation}
\label{sm:equ:sphe2}
B_q \vect{X} \stackrel{d}{=} B_q B \vect{X} = \vect{X}_q. 
\end{equation}
Now, the result follows from (\ref{sm:equ:sphe1}) and (\ref{sm:equ:sphe2}).
\end{proof}

\begin{proposition}
 \label{sm:prop:sym:sub:elliptical}
If $\vect{X}$ is elliptically symmetric with parameters $\vect{\mu}$ and $\Sigma$, then $\vect{Y} = B_q \vect{X}$ is elliptically symmetric with parameters $B_q \vect{\mu}$ and $B_q \Sigma B_q^{\top}$.
\end{proposition}
\begin{proof}
See Section 1.2.4 in \citet{frahm2004}.
\end{proof}

\begin{proposition}
\label{sm:prop:sym:sub:central}
If $\vect{X}$ is centrally symmetric about $\vect{\mu} \in \mathbb{R}^m$, then $\vect{Y} = B_q \vect{X}$ is centrally symmetric about $B_q \vect{\mu}$.
\end{proposition}

\begin{proof}
Since $\vect{X}$ is centrally symmetric $(\vect{X} - \vect{\mu}) \stackrel{d}{=} (\vect{\mu} - \vect{X})$. Therefore,
\begin{align*}
(B_q \vect{X} - B_q \vect{\mu}) & \stackrel{d}{=} B_q(\vect{X} - \vect{\mu}) \\
& \stackrel{d}{=} B_q(\vect{\mu} - \vect{X}) \\
& \stackrel{d}{=} (B_q\vect{\mu} - B_q\vect{X}).
\end{align*}
\end{proof}

\begin{proposition}
If $\vect{X}$ is angularly symmetric about $\vect{\mu} \in \mathbb{R}^m$, then $\vect{Y} = B_q \vect{X}$ is angularly symmetric about $B_q \vect{\mu}$.
\end{proposition}

\begin{proof}
By Theorem 2.2 in \citet{zuoserfling2000b} angular symmetry is equivalent to 
\begin{equation*}
\bP( \vect{u}^\top (\vect{X} - \vect{\mu}) \geq 0) = \bP( \vect{u}^\top (\vect{\mu} - \vect{X}) \geq 0)
\end{equation*}
for all unit vectors $\vect{u} \in \mathbb{R}^m$. Hence, to prove that $\vect{Y}$ is angular symmetric about $B_q \vect{\mu}$ it is enough to show that 
\begin{equation*}
  \bP( \vect{u}_q^\top (B_q\vect{X} - B_q\vect{\mu}) \geq 0) = \bP( \vect{u}_q^\top (B_q\vect{\mu} - B_q\vect{X}) \geq 0)
\end{equation*}
for all unit vectors $\vect{u}_q \in \mathbb{R}^q$. Since $\frac{\vect{u}_q^\top B_q}{\norm{\vect{u}_q^\top B_q}}$ is a unit vector in $\mathbb{R}^m$ we obtain
\begin{align*}
\bP( \vect{u}_q^\top (B_q\vect{X} - B_q\vect{\mu}) \geq 0)
&= \bP( \frac{\vect{u}_q^\top B_q}{\norm{\vect{u}_q^\top B_q}} (\vect{X} - \vect{\mu}) \geq 0) \\
&= \bP( \frac{\vect{u}_q^\top B_q}{\norm{\vect{u}_q^\top B_q}} (\vect{\mu} - \vect{X}) \geq 0) \\
&= \bP( \vect{u}_q^\top (B_q\vect{\mu} - B_q\vect{X}) \geq 0).
\end{align*}
\end{proof}

\begin{proposition}
\label{sm:prop:sym:sub:halfspace}
If $\vect{X}$ is halfspace symmetric about $\vect{\mu} \in \mathbb{R}^m$, then $\vect{Y} = B_q \vect{X}$ is halfspace symmetric about $B_q \vect{\mu}$.
\end{proposition}

\begin{proof}
Recall that halfspace symmetry is equivalent to the existence of a median such that
\begin{equation*}
\vect{u}^\top \vect{\mu} = \Med(\vect{u}^\top \vect{X})
\end{equation*}
for all unit vector $\vect{u} \in \mathbb{R}^m$ (see Theorem 2.4 in \citet{zuoserfling2000b}). Hence, to prove that $\vect{Y}$ is halfspace symmetric about $B_q \vect{\mu}$, it is sufficient to show that 
\begin{equation*}
\vect{u}_q^\top B_q \vect{\mu} = \Med(\vect{u}_q^\top B_q \vect{X})
\end{equation*}
for all unit vectors $\vect{u}_q \in \mathbb{R}^q$. Recall that the median is equivariant to scalar transformations, i.e.\ $\Med(a \vect{X}) = a \Med(\vect{X})$ for all $a \in \mathbb{R}$. Since $\frac{\vect{u}_q^\top B_q}{\norm{\vect{u}_q^\top B_q}}$ is a unit vector in $\mathbb{R}^m$, it follows that
\begin{align*}
\frac{1}{\norm{\vect{u}_q^\top B_q}} \vect{u}_q^\top B_q \vect{\mu} &=
\frac{\vect{u}_q^\top B_q}{\norm{\vect{u}_q^\top B_q}} \vect{\mu} \\
&= \Med \biggl(\frac{\vect{u}_q^\top B_q}{\norm{\vect{u}_q^\top B_q}} \vect{X} \biggr)\\
&=\frac{1}{\norm{\vect{u}_q^\top B_q}} \Med(\vect{u}_q^\top B_q \vect{X}).
\end{align*}
\end{proof}
  
\section{Statistical data depths}
\label{sm:sec:depthproperties}

In this section we recall the definition of statistical data depth. Let $\mathcal{F}$ be the class of distribution functions on $\mathbb{R}^m$. For a random variable $\vect{X} \in \mathbb{R}^m$ with distribution function $F \in \mathcal{F}$, an $m \times m$ nonsingular matrix $A$ and a vector $\vect{b} \in \mathbb{R}^m$, we denote by $F_{A,\vect{b}}$ the distribution function of $A \vect{X}+\vect{b}$. A statistical data depth $d(\cdot,\cdot)$ is a bounded function from $\mathcal{F} \times \mathbb{R}^m$ to $\mathbb{R}_{+}=[0,\infty)$ that satisfies the following properties \citep{liu1990,zuoserfling2000}:
\begin{enumerate}[label=(\textbf{D\arabic*})]
\item \label{sm:PropDepthAffineInvariance} Affine invariance: $d(\vect{x}, F)=d(A \vect{x}+ \vect{b}, F_{A,\vect{b}})$;

The depth of a point does not depend on the underlying coordinate system and the scales of the underlying measurements.  

\item \label{sm:PropDepthMaximalityAtCenter} Maximality at center: if $F$ is ``symmetric'' about $\vect{\mu}$ then $d(\vect{x}, F) \leq d(\vect{\mu}, F)$ for all $\vect{x}$;
  
For a distribution having a uniquely defined ``center'' (e.g., the point of symmetry w.r.t. some notion of symmetry), the data depth attains maximum value at the center.

\item \label{sm:PropDepthMonotonicity} Monotonicity: let $\vect{\mu} = \arg\max_{\vect{x} \in \mathbb{R}^m} d(\vect{x}, F)$, then 
\begin{equation*}
d(\vect{x}, F) \le d(\vect{\mu} + \alpha (\vect{x} - \vect{\mu}), F), \quad \alpha \in [0,1];
\end{equation*}

The depth of a point is monotonically non-increasing along any ray from the deepest point (i.e.\ the point with maximum depth).

\item \label{sm:PropDepthZero} Approaching zero: 
\begin{equation*}
d(\vect{x}, F) \rightarrow 0 \text{ as } \norm{\vect{x}} \to \infty.
\end{equation*}
\end{enumerate}

Among the most well-known examples of statistical data depths are halfspace and simplicial depth \citep{tukey1975,liu1990}. The halfspace depth of a point $\vect{x} \in \mathbb{R}^m$ with respect to the distribution $F$ is given by
\begin{equation*}
d_{H}(\vect{x},F)=\inf_{\vect{u} \in S^{m-1}} \bP(\vect{X} \in H_{\vect{x},\vect{u}}),
\end{equation*}
where $S^{m-1}$ is the unit sphere in $\mathbb{R}^m$ and $H_{\vect{x},\vect{u}} = \{ \vect{y} \in \mathbb{R}^m : \, \langle \vect{u}, \vect{y} \rangle \leq \langle \vect{u}, \vect{x} \rangle \}$ is the closed halfspace with boundary point $\vect{x} \in \mathbb{R}^m$ and outer normal $\vect{u} \in S^{m-1}$. The simplicial depth of $\vect{x}$ with respect to $F$ is
\begin{equation*}
d_\Delta(\vect{x},F) = \bP(\vect{x} \in \Delta[\vect{X}_1,\dots,\vect{X}_{m+1}]),
\end{equation*}
where $\vect{X}_i$ are independent and identically distributed (i.i.d.)\ with distribution $F$ and $\Delta[\vect{x}_1,\dots,\vect{x}_{m+1}]$ is the closed simplex with vertices $\vect{x}_1,\dots,\vect{x}_{m+1}$.

\begin{remark}
In some cases \ref{sm:PropDepthMonotonicity} is required to hold only under \ref{sm:PropDepthMaximalityAtCenter}, e.g. simplicial depth.
\end{remark}

Other depth notions were originally introduced as a measure of dispersion or outlyingness. These include simplicial volume depth \citep{oja1983} and Mahalanobis depth \citep{mahalanobis1936}. Let $\lambda(\cdot)$ be the Lebesgue measure on $\mathbb{R}^m$ and consider the average simplex volume
\begin{equation*}
V(\vect{x},F) = \E[\lambda(\Delta[\vect{x},\vect{X}_1,\dots,\vect{X}_{m}])],
\end{equation*}
where $\vect{X}_i$ are i.i.d.\ with distribution $F$. The simplicial volume depth $d_{V}(\vect{x},F)$ of $\vect{x}$ with respect to $F$ is given by a monotone decreasing transformation of $V(\vect{x},F)$. Finally, if $F$ admits mean $\vect{\mu}$ and invertible covariance matrix $\Sigma$, the Mahalanobis depth $d_{M}(\vect{x},F)$ of $\vect{x}$ with respect to $F$ is given by a monotone decreasing transformation of the Mahalanobis distance
\begin{equation*}
M(\vect{x}, F) = \sqrt{(\vect{x}-\vect{\mu})^\top \Sigma^{-1} (\vect{x}-\vect{\mu})}.
\end{equation*}
See \citep{zuoserfling2000} for further details. Simplicial volume depth is only invariant to rigid-body transformations, i.e.\ \ref{sm:PropDepthAffineInvariance} holds for orthogonal matrices only.

\section{Dispersion measures based on statistical data depths}
\label{sm:sec:dispersion}

\subsection{Dispersion measures}
\label{sm:subsection_dispersion_measures}

Recall that $\mathcal{F}$ is the set of all distribution functions on $\mathbb{R}^m$ and $\lambda(\cdot)$ is the Lebesgue measure. Let $d(\cdot , \cdot)$ be a statistical data depth and for $\alpha >0$ let $R_{\alpha}(F) = \{ \vect{y} \in \mathbb{R}^m \, : \, d(\vect{y},F) \geq \alpha \}$ be the corresponding $\alpha$-trimmed region.

\begin{definition}[\citet{zuoserfling2000d}]
\label{sm:definition_more_scattered}
The distribution $F \in \mathcal{F}$ is \emph{more scattered} than $G \in \mathcal{F}$ if $\lambda(R_{\alpha}(F)) \geq \lambda(R_{\alpha}(G))$ for all $\alpha>0$.
\end{definition}

\citet{oja1983} defines dispersion measures (or scatter measures) as follows (see also \citet{zuoserfling2000d} and \citet{romanazzi2009}).

\begin{definition}[\citet{oja1983}] \label{sm:definition_dispersion_measure}
The function $\psi: \mathcal{F} \rightarrow \mathbb{R}_{+}$ is a dispersion measure if
\begin{enumerate}
\item $\psi(F) \ge \psi(G)$ for any $F$ more scattered than $G$.
\item $\psi(F_{A, \vect{b}}) = \abs{\det(A)} \psi(F)$ for any $m \times m$ matrix $A$ and $\vect{b} \in \mathbb{R}^m$.
\end{enumerate}
\end{definition}

For any $\alpha>0$ the function $\psi_{\alpha}(F)=\lambda(R_{\alpha}(F))$ is a dispersion measure in the sense of Definitions \ref{sm:definition_more_scattered}-\ref{sm:definition_dispersion_measure} (see Theorem 2.1 in \citet{zuoserfling2000d}).

\subsection{Properties of the dispersion measure $\sigma(\cdot)$}
\label{sm:subsection_properties_of_dispersion_measures}

In this subsection we summarize from \citet{romanazzi2009} some important properties of $\sigma(\cdot)$ (see Definition \ref{def:dispersion} in the main paper). We first notice that, since data depths are affine invariant, $\sigma(\cdot)$ is a dispersion measure in the sense of Definitions \ref{sm:definition_more_scattered}-\ref{sm:definition_dispersion_measure}. If $F$ is concentrated on a $k$-dimensional subspace $S_{k,\vect{\mu}} \subset \mathbb{R}^m$, where $0 \leq k \leq m-1$, then the (halfspace and simplicial) depth of $F$ is $0$ outside $S_{k,\vect{\mu}}$, from which follows that $\sigma(F)=0$ \citep{romanazzi2009}[Remark 11]. We denote by
\begin{equation*}
  \sigma_H(F) = \int_{\mathbb{R}^m} d_H(\vect{x}, F) \ d\vect{x} \quad \text{and} \quad \sigma_\Delta(F) = \int_{\mathbb{R}^m} d_\Delta(\vect{x}, F) \ d\vect{x}
\end{equation*}
the dispersion measure of $F$ based on the halfspace and simplicial depth. Proposition 9 in \citet{romanazzi2009} shows that $\sigma_\Delta(F)$ equals to multivariate Gini coefficient, that is,
\begin{equation} \label{sm:multivariate_gini_coefficient}
\sigma_\Delta(F)=\E[\lambda(\Delta[\vect{X}_1,\dots,\vect{X}_{m+1}])],
\end{equation}
where $\vect{X}_i$ are i.i.d.\ with distribution $F$. Recall that, for univariate distributions $F$,
\begin{equation*}
  d_\Delta(x,F) = 2 F(x)(1-F(x)) + (F^{2}(x) - F^{2}(x^{-})),
\end{equation*}  
where $F(x^{-})$ is the left limit of $F$ at the point $x$ \citep{nagy2023}. Hence, using that $F$ is continuous almost everywhere,
\begin{equation*}
  \sigma_\Delta(F) = 2 \int_{-\infty}^{\infty} F(x) (1-F(x)) \, dx,
\end{equation*}
and, by \eqref{sm:multivariate_gini_coefficient}, it follows that
\begin{equation*}
\E\left[ \abs{X_{1}-X_{2}} \right] =  2 \int_{-\infty}^{\infty} F(x) (1-F(x)) \, dx,
\end{equation*}
which is a well-known computational formula for Gini's mean difference \citep{romanazzi2009}[Remark 10].

Let $h \, : \, \mathbb{R}_{+} \to \mathbb{R}$ a decreasing and bounded function such that $\tau_{0}=\int_{\mathbb{R}^m} h(\norm{\vect{x}}) \, d \vect{x}$ exists and is finite. If $F \in \mathcal{F}$ has mean $\vect{\mu}$ and covariance matrix $\Sigma$, then by \citet{romanazzi2009}[Proposition 14],
\begin{equation*}
  \sigma_M(F)=\int_{\mathbb{R}^m} h(M(\vect{x},F)) \, d\vect{x} = \tau_{0} (\det(\Sigma))^{1/2}.
\end{equation*}
In particular, Wilks' generalized variance \citep{wilks1960} is the expected squared volume of a random simplex, that is,
\begin{equation*}
  W(F) = (p+1)^{-1} \E[ \lambda^2(\Delta[\vect{X}_1,\dots,\vect{X}_{m+1}])],
\end{equation*}
and it satisfies
\begin{equation*}
  W(F) = \det(\Sigma) = (\sigma_M(F)/\tau_0)^2.
\end{equation*}
In the next subsection, we show that, under suitable tail and moment conditions on the underlying distribution $F$, the dispersion measures $\sigma_H(F)$ and $\sigma_\Delta(F)$ are finite.

\subsection{Finiteness of the dispersion measure $\sigma(\cdot)$}
\label{sm:subsection_finiteness_of_dispersion_measure}

We begin with a result concerning the univariate halfspace and simplicial depth.

\begin{proposition} \label{sm:Prop1dDepFin}
Let $m=1$ and $X \sim F$. Then, $\sigma_H(F) < \infty$ if and only if $\E\vert X \vert  < \infty$. Additionally, if $\E\vert X \vert  < \infty$, then $\sigma_\Delta(F) < \infty$.
\end{proposition}
\begin{proof}
We write $X=X^{+}-X^{-}$, where $X^{+}=\max(X,0)$ and $X^{-}=-\min(X,0)$, and observe that $\vert X \vert = X^{+} + X^{-}$. Then,
\begin{align} \label{sm:1dimExpAbsVal}
\E\vert X \vert = \E X^{+} + \E X^{-} &= \int_{0}^{\infty} \bP(X^{+}>x) \, dx + \int_{0}^{\infty} \bP(X^{-}>x) \, dx \nonumber \\
&= \int_{0}^{\infty} 1-F(x) \, dx + \int_{-\infty}^{0} F(x) \, dx.
\end{align}
Let $m$ be a median for $X$, then
\begin{align}  \label{sm:1dimDispHalfDep}
\sigma_{H}(F) &= \int_{-\infty}^{\infty} \min(F(x),1-F(x^{-})) \, dx \nonumber \\
&= \int_{-\infty}^{m} F(x) \, dx + \int_{m}^{\infty} 1-F(x) \, dx \nonumber \\
&= \int_{-\infty}^{0} F(x) \, dx + \int_{0}^{\infty} 1-F(x) \, dx + \int_{m}^{0} 1 - 2 F(x) \, dx,
\end{align}
where for $m>0$
\begin{equation*}
\int_{m}^{0} 1 - 2 F(x) \, dx = - \int_{0}^{m} 1 - 2 F(x) \, dx.
\end{equation*}
Since $\int_{m}^{0} 1- 2F(x) \, dx < \infty$, the first claim follows from \eqref{sm:1dimExpAbsVal} and \eqref{sm:1dimDispHalfDep}. For the second claim, we have that
\begin{align*}
  \sigma_\Delta(F) &= 2 \int_{-\infty}^{\infty} F(x) (1-F(x)) \, dx \\
&= 2 \left[ \int_{-\infty}^{m} F(x) (1-F(x)) \, dx + \int_{m}^{\infty} F(x) (1-F(x)) \, dx \right] \\
&\leq 2 \left[ \int_{-\infty}^{m} F(x) \, dx + \int_{m}^{\infty} 1-F(x) \, dx \right],
\end{align*}
where the last line is twice the value in equation \eqref{sm:1dimDispHalfDep}. Hence \eqref{sm:1dimExpAbsVal} implies that $\sigma_\Delta(F)<\infty$.
\end{proof} \\

 In the following, we replace, for convenience, the cumulative distribution function $F$ by the corresponding probability distribution $P$ and write, for instance, $d_H(\cdot,P)$ and $\sigma_H(P)$ for $d_H(\cdot,F)$ and $\sigma_H(F)$. The next result is an extension of Proposition \ref{sm:Prop1dDepFin} to multivariate random variables.

\begin{proposition} \label{sm:PropDispFinHalfSimp}
Let $\vect{X} \sim P$ and $\bP(\norm{\vect{X}} \geq \norm{\cdot})$ be an integrable function. Then, $\sigma_H(P), \sigma_\Delta(P) < \infty$.
\end{proposition}
\begin{proof}
For $\vect{x} \neq \vect{0}$, by choosing the direction $u=-\frac{\vect{x}}{\norm{\vect{x}}}$ and using Cauchy-Schwarz inequality, we have that
\begin{equation} \label{bound_halfspace_depth}
d_H(\vect{x},P) \leq \bP( \vect{X} \in H_{\vect{x},-\frac{\vect{x}}{\norm{\vect{x}}}}) \leq \bP(\norm{\vect{X}} \geq \norm{\vect{x}}).
\end{equation}
For the simplicial depth, notice that $\vect{x} \in \Delta[\vect{x}_1,\dots,\vect{x}_{m+1}]$ if and only if there exist $\alpha_i \geq 0$ with $\sum_{i=1}^{m+1} \alpha_i = 1$ such that $\vect{x}=\sum_{i=1}^{m+1}\alpha_i \vect{x}_i$. In particular, $\norm{\vect{x}} \leq \sum_{i=1}^{m+1} \alpha_i \norm{\vect{x}_i} \leq \max_{i=1,\dots,m+1} \norm{\vect{x}_i}$ and
\begin{equation} \label{bound_simplicial_depth}
\begin{aligned}
d_\Delta(\vect{x},P) &\leq \bP( \max_{i=1,\dots,m+1} \norm{\vect{X}_i} \geq \norm{\vect{x}}) =\bP(\cup_{i=1}^{m+1} \{ \norm{\vect{X}_i} \geq \norm{\vect{x}} \} ) \\
&\leq (m+1) \bP( \norm{\vect{X}} \geq \norm{\vect{x}}),
\end{aligned}
\end{equation}
where $\vect{X}_{i} \sim P$. Since $\bP(\norm{\vect{X}} \geq \norm{\cdot})$ is integrable, using \eqref{bound_halfspace_depth} and \eqref{bound_simplicial_depth}, we conclude that both $\sigma_{H}(P)$ and $\sigma_\Delta(P)$ are finite.
\end{proof} \\

The assumptions in Proposition \ref{sm:PropDispFinHalfSimp} are clearly satisfied if $P$ has bounded support. More generally, the absolute integrability requires a sufficiently quick decay of the tail probabilities.

\begin{definition}[Polynomial decay] \label{sm:DefPolDecay}
Let $\Pi$ be a class of probability measures. $\Pi$ is said to decay polynomially if there exist constants $C>0$ and $\alpha>1$ such that
\begin{equation} \label{sm:EqPolDecayFam}
\bP(\norm{\vect{X}} \geq t) \leq \frac{C}{(1+t)^{\alpha}},
\end{equation}
for all $\vect{X} \sim P \in \Pi$ and all $t \geq 0$.
\end{definition}

\begin{remark} \label{remark_polynomial_decay}
By Markov's inequality condition \eqref{sm:EqPolDecayFam} is satisfied if, for some $\alpha>1$, $\E[\norm{\vect{X}}^{\alpha}] \leq C_0 < \infty$ for all $\vect{X} \sim P \in \Pi$. The constants are given by $C=2^{\alpha} \max(1,C_0)$ and $\alpha>1$.
\end{remark}

\begin{corollary} \label{sm:CorDispFinHalfSimp}
If $\{ P \}$ decays polynomially, then $\sigma_{H}(P),\sigma_\Delta(P)<\infty$.
\end{corollary}

\begin{proof}
  If $\{ P \}$ decays polynomially with constants $C>0$ and $\alpha>1$, then
\begin{equation*}
  \bP(\norm{\vect{X}} \geq \norm{\vect{x}}) \leq \frac{C}{(1+\norm{\vect{x}})^{\alpha}}
\end{equation*}
for all $\vect{x}$. It follows that $\sigma_{H}(P),\sigma_\Delta(P)<\infty$.
\end{proof} \\

In the next subsection we show that fractional moments of order $\alpha > 1$ for multivariate t-distribution are finite whenever $\alpha$ is smaller than the degrees of freedom.  

\subsection{Finiteness of fractional moments for t-distribution} \label{sm:finiteness_of_fractional_moments_for_t-distribution}

Let $\vect{X}$ be a multivariate t-distribution with parameters $\Sigma,\vect{\mu},\nu$, that is, $\vect{X}$ has density
\begin{equation*}
f(\vect{x})=\frac{\Gamma((\nu+m)/2)}{\Gamma(\nu/2) \nu^{m/2} \pi^{m/2} \det(\Sigma)^{1/2}} \left( 1 + \frac{1}{\nu} (\vect{x}-\vect{\mu})^{\top} \Sigma^{-1}  (\vect{x}-\vect{\mu})  \right)^{-(\nu+m)/2},
\end{equation*}
where $\Gamma(\cdot)$ is the gamma function. The next proposition shows that $\vect{X}$ has finite moments of order $\alpha$, for $0 \leq \alpha < \nu$, while the moment of order $\nu$ is infinite. In particular, if $1 < \nu \leq 2$, $\vect{X}$ has infinite variance, but still has finite moments of order $0 \leq \alpha < \nu$.

\begin{proposition} \label{sm:proposition_finiteness_of_fractional_moments_for_t-distribution}
  Let $\alpha \geq 0$. Then
\begin{equation*}
\int \norm{\vect{x} - \vect{\mu}}^{\alpha} f(\vect{x}) \, d\vect{x} < \infty
\end{equation*}
if and only if $\alpha < \nu$.
\end{proposition}
\begin{proof}
Notice that there exists a unitary matrix $U$ and a diagonal matrix $D$ such that $\Sigma=U^{T}DU$ and let $\Sigma^{-1/2}=U^{T}D^{-1/2}U$. By a change of variable ($\vect{z} = \Sigma^{-1/2} (\vect{x} - \vect{\mu})$) we see that
\begin{equation*}
  \int \norm{\vect{x} - \vect{\mu}}^{\alpha} f(\vect{x}) \, d\vect{x} = \frac{\Gamma((\nu+m)/2)}{\Gamma(\nu/2) \nu^{m/2} \pi^{m/2}} \int \norm{\Sigma^{1/2}\vect{z}}^{\alpha} \left(1+\norm{\vect{z}}^2/\nu \right)^{-(\nu+m)/2} \, d\vect{z}.
\end{equation*}
Notice that $\norm{\cdot}_{\Sigma^{1/2}}$, given by $\norm{\vect{z}}_{\Sigma^{1/2}}=\norm{\Sigma^{1/2} \vect{z}}$, is a norm on $\mathbb{R}^m$. In particular, there exist constants $0<C_1,C_2<\infty$ such that $C_1 \norm{\vect{z}} \leq \norm{\Sigma^{1/2} \vect{z}} \leq C_2 \norm{\vect{z}}$ for all $\vect{z} \in \mathbb{R}^m$. Therefore, it is enough to check the finiteness of
\begin{equation*}
  \int \norm{\vect{z}}^{\alpha} \left(1+\norm{\vect{z}}^2/\nu \right)^{-(\nu+m)/2} \, d\vect{z}
\end{equation*}  
Using the the coarea formula, we see that the above integral is equal to
\begin{equation*}
  \int_{0}^{\infty} r^{\alpha} (1+r^2/\nu)^{-(\nu+m)/2} H^{m-1}(\partial B_r(\vect{0})) \, dr,
\end{equation*}
where $H^{m-1}$ is the $m-1$ dimensional Hausdorff measure and $\partial B_r(\vect{0})$ is the hypersphere with center $\vect{0}$ and radius $r$. Using that the surface area of $\partial B_r(\vect{0})$ is $2 \pi^{m/2} r^{m-1}/\Gamma(m/2)$, the above integral reduces to 
\begin{equation*}
\frac{2 \pi^{m/2}}{\Gamma(m/2)} \int_{0}^{\infty} r^{\alpha+m-1} (1+r^2/\nu)^{-(\nu+m)/2} \, dr.
\end{equation*}
Now, since for $r \geq \sqrt{\nu}$
\begin{equation*}
  \left(\frac{\nu}{2} \right)^{(\nu+m)/2} r^{-(\nu+m)} \le (1+r^2/\nu)^{-(\nu+m)/2} \leq (r^2/\nu)^{-(\nu+m)/2},
\end{equation*}  
we have
\begin{align*}
 \left( \frac{\nu}{2} \right)^{(\nu+m)/2}  \int_{\sqrt{\nu}}^{\infty} r^{-(\nu-\alpha+1)} \, dr &\leq \int_{0}^{\infty} r^{\alpha+m-1} (1+r^2/\nu)^{-(\nu+m)/2} \, dr \\
  &\leq \int_{0}^{\sqrt{\nu}} \nu^{(\alpha+m-1)/2} \, dr + \nu^{(\nu+m)/2} \int_{\sqrt{\nu}}^{\infty} r^{-(\nu-\alpha+1)} \, dr.
\end{align*}
The result follows using that $\int_{\sqrt{\nu}}^{\infty} r^{-(\nu-\alpha+1)} \, dr$ is finite if and only if $\alpha<\nu$.
\end{proof}

\subsection{Continuity of the dispersion measure $\sigma(\cdot)$}
\label{sm:subsection_continuity_of_dispersion_measure}

\citet{nieto2016} and \citet{gijbels2017} suggest further properties that a statistical data depth should satisfy.  Here, we recall the definition of continuity and uniform continuity in $P$. Clearly, uniform continuity at $P$ implies continuity at $P$. We begin by recalling the definition of weak convergence of probability distributions.

\begin{definition}[Weak convergence of probability distributions]
  A sequence of probability distributions $(P_n)_{n=1}^\infty$ is said to converge weakly to $P$ (in symbols, $P_n \stackrel{w}{\rightarrow} P$), if, for all bounded and continuous functions $f:\mathbb{R}^{m} \to \mathbb{R}$,
\begin{equation*}
  \lim_{n \to \infty} \int_{\mathbb{R}^{m}} f(\vect{x}) dP_{n}(x) = \int_{\mathbb{R}^{m}} f(\vect{x}) dP(x).
\end{equation*}
\end{definition}
In the above definition, we can equivalently take $f(\cdot)$ to be Lipschitz continuous (see Theorem 11.3.3 in \citet{Dudley-2018}).

\begin{definition}[Continuity at $P$ of the depth]
Let $(P_n)_{n=1}^\infty$ be a sequence of probability distributions that converges weakly to a probability distribution $P$. A statistical data depth $d(\cdot , \cdot)$ is continuous at $P$ if
\begin{equation*}
d(\vect{x},P_n) \rightarrow d(\vect{x},P)  \quad \text{as} \ n \rightarrow \infty
\end{equation*}
for all $\vect{x} \in \mathbb{R}^m$.
\end{definition}

\begin{definition}[Uniform continuity at $P$ of the depth]
Let $(P_n)_{n=1}^\infty$ be a sequence of probability distributions that converges weakly to a probability distribution $P$. A statistical data depth $d(\cdot , \cdot)$ is uniformly continuous at $P$ if
\begin{equation*}
\sup_{\vect{x} \in \mathbb{R}^m} \vert d(\vect{x},P_n) - d(\vect{x},P) \vert \rightarrow 0 \qquad \text{as} \ n \rightarrow \infty.
\end{equation*}
\end{definition}

\begin{proposition}[Continuity at $P$ of the dispersion measure] \label{sm:PropContSigmatP}
Let $(P_n)_{n=1}^\infty$ be a sequence of probability distributions that converges weakly to a probability distribution $P$ and suppose that $d(\cdot , \cdot)$ is continuous at $P$. If there is an integrable function $g(\cdot)$ such that $d(\vect{x},P_n) \leq g(\vect{x})$ for all $n \in \mathbb{N}$, then 
\begin{equation*}
\sigma(P_n)  \rightarrow \sigma(P) \qquad \text{as} \ n \rightarrow \infty.
\end{equation*}
\end{proposition}

\begin{proof}
Since $d(\vect{x},P_n) \rightarrow d(\vect{x}, P)$ and $d(\vect{x},P_n) \leq g(\vect{x})$ for every given $\vect{x}$, the result follows from Lebesgue's dominated convergence theorem.
\end{proof} \\

Uniform continuity at $P$ is shown, for the halfspace depth, in Theorem A.3 in \citet{nagy2016} and, for the simplicial depth, in Corollary 2 in \citet{dumbgen1992}, under the hypothesis that $P$ assigns probability $0$ to all hyperplanes in $\mathbb{R}^m$, that is,
\begin{equation} \label{sm:Prob0Hyperplanes}
P(L) = 0 \text{ for all hyperplanes } L \subset \mathbb{R}^m.
\end{equation}
Notice that \eqref{sm:Prob0Hyperplanes} is satisfied if $P$ is absolutely continuous. Halfspace and simplicial depths are given in terms of probabilities of halfspaces and simplices. If \eqref{sm:Prob0Hyperplanes} is not satisfied their boundaries may have positive probability, and, thus, continuity at $P$ may fail (see Portmanteau theorem \citep[][Theorem 2.1]{billingsley1999}). An immediate consequence of the uniform continuity at $P$ of halfspace and simplicial depths and Proposition \ref{sm:PropContSigmatP} is the following corollary.

\begin{corollary} \label{sm:CorContSigminPHalfSimp}
Let $(P_n)_{n=1}^\infty$ be a sequence of probability distributions that converges weakly to a probability distribution $P$. If there is an integrable function $g(\cdot)$ such that $d_H(\vect{x},P_n) \leq  g(\vect{x})$ (resp.\ $d_\Delta(\vect{x},P_n) \leq g(\vect{x})$), for all $n \in \mathbb{N}$, and $P$ satisfies \eqref{sm:Prob0Hyperplanes}, then $\sigma_H(P_n) \to \sigma_H(P)$ (resp.\ $\sigma_\Delta(P_n) \to \sigma_\Delta(P)$) as $n \to \infty$.
\end{corollary}

As the following corollary shows, the existence of $g(\cdot)$ is ensured if $(P_n)_{n=1}^{\infty}$ decays polynomially.

\begin{corollary} \label{sm:CorPolDecay}
Let $(P_n)_{n=1}^\infty$ be a sequence of probability distributions that decays polynomially and converges weakly to a probability distribution $P$. If $P$ satisfies \eqref{sm:Prob0Hyperplanes}, then $\sigma_H(P_n) \to \sigma_H(P)$ and $\sigma_\Delta(P_n) \to \sigma_\Delta(P)$ as $n \to \infty$.
\end{corollary}

\begin{proof}
By \eqref{bound_halfspace_depth} and \eqref{bound_simplicial_depth} with $P$ replaced by $P_{n}$ and $\vect{X}$ replaced by $\vect{X}_{n}$, we obtain that $d_H(\vect{x},P_n) \leq \bP(\norm{\vect{X}_n} \geq \norm{\vect{x}})$ and $d_\Delta(\vect{x},P_n) \leq (m+1) \, \bP(\norm{\vect{X}_n} \geq \norm{\vect{x}})$. Moreover, since $(P_n)_{n=1}^{\infty}$ decays polynomially, there exist $C>0$ and $\alpha>1$ such that
\begin{equation} \label{dominating_function_empirical_halfspace_depth}
 \bP(\norm{\vect{X}_n} \geq \norm{\vect{x}}) \leq g(\vect{x}) = \frac{C}{(1+\norm{\vect{x}})^{\alpha}}.
\end{equation}
The result now follows from Corollary \ref{sm:CorContSigminPHalfSimp}.
\end{proof}

\subsection{Dispersion measure $\sigma(\cdot)$ in a subspace}
\label{sm:subsection_dispersion_measure_in_subspace}

For $1 \leq q \leq m$ let $\mathcal{B}^{q \times m}$ be the set of $q \times m$ matrices $B_q$ with orthonormal row vectors i.e.\ $B_q B_q^{\top} = I_{q}$, where $I_{q}$ is the $q \times q$ identity matrix. $\mathcal{B}^{q \times m}$ is endowed with the matrix norm $\norm{B_q}_{\mathcal{B}^{q \times m}}=\sup_{\vect{z} \in \mathbb{R}^m \setminus \{ \vect{0} \} } \frac{\norm{B_q \vect{z}}}{\norm{\vect{z}}}$. The next result is concerned with preservation of weak convergence under mappings in $\mathcal{B}^{q \times m}$. Results for general mappings are given by \citet{billingsley1967,topsoe1967}. For a probability distribution $P$ on $\mathbb{R}^m$ and $B_{q} \in \mathcal{B}^{q \times m}$, we denote by $P_{B_q}$ the probability distribution on $\mathbb{R}^q$ given by $P_{B_q}(A)=P(B_q^{-1}A)$ for all Borel sets $A \subset \mathbb{R}^q$.

\begin{proposition} \label{sm:PropConvPBq}
  Suppose that $\vect{X}_{n} \sim P_{n}$ satisfies $\limsup_{n \to \infty} \E[\norm{\vect{X}_{n}}] < \infty$. If $P_{n} \stackrel{w}{\rightarrow} P$ and $B_{q}^k \rightarrow B_q$, then $(P_{n})_{B_{q}^{k}} \stackrel{w}{\rightarrow} P_{B_q}$ as $n,k \to \infty$.
\end{proposition}

\begin{proof}
  Let $f : \mathbb{R}^q \rightarrow \mathbb{R}$ be bounded and Lipschitz continuous with Lipschitz constant $0<L<\infty$. For $\vect{x} \in \mathbb{R}^m$ it holds that
\begin{equation} \label{f_lipschitz}
  \left\vert f(B_{q}^{k}\vect{x})-f(B_{q} \vect{x}) \right\vert \leq L \norm{B_{q}^{k} \vect{x} - B_{q} \vect{x} } \, \leq \, L \norm{B_{q}^{k}-B_q}_{\mathcal{B}^{q \times m}} \norm{\vect{x}}.
\end{equation}
Using that $P_{n} \stackrel{w}{\rightarrow} P$ and $C_{0} \coloneqq \limsup_{n \to \infty} \E[\norm{\vect{X}_{n}}]+1<\infty$; for $\epsilon>0$, let $n^{*} \in \mathbb{N}$ such that, for all $n \geq n^{*}$,
\begin{equation*}
  \left\vert \int_{\mathbb{R}^m} f(B_{q}\vect{x}) \, d(P_n(\vect{x}) - P(\vect{x})) \right\vert \leq \epsilon/2
\end{equation*}
and
\begin{equation*}
  \int_{\mathbb{R}^{m}} \norm{\vect{x}} \, dP_{n}(\vect{x}) = \E[\norm{\vect{X}_{n}}] \leq C_{0}.
\end{equation*}
Next, let $k^{*} \in \mathbb{N}$ such that $\norm{B_{q}^{k}-B_q}_{\mathcal{B}^{q \times m}} \leq (2 L C_{0})^{-1} \epsilon$ for all $k \geq k^{*}$. Using the triangle inequality and \eqref{f_lipschitz}, we conclude that
\begin{align*}
  &\left\vert \int_{\mathbb{R}^q} f(\vect{y}) \, d(P_{n})_{B_{q}^{k}}(\vect{y}) - \int_{\mathbb{R}^q} f(\vect{y}) \, dP_{B_{q}}(\vect{y}) \right\vert \\
= &\left\vert \int_{\mathbb{R}^m} f(B_{q}^{k}\vect{x}) \, dP_n(\vect{x}) - \int_{\mathbb{R}^m} f(B_{q} \vect{x}) \, dP(\vect{x}) \right\vert \\
\leq & \int_{\mathbb{R}^m}  \left\vert f(B_{q}^{k}\vect{x})-f(B_{q} \vect{x}) \right\vert \, dP_{n}(\vect{x}) + \left\vert \int_{\mathbb{R}^m} f(B_{q}\vect{x}) \, d(P_n(\vect{x}) - P(\vect{x})) \right\vert \leq \epsilon
\end{align*}
for all $k \geq k^{*}$ and $n \geq n^{*}$.
\end{proof}

\begin{corollary} \label{sm:CorContSigminPBq}
Suppose that $\vect{X}_{n} \sim P_{n}$ satisfies $\limsup_{n \to \infty} \E[\norm{\vect{X}_{n}}] < \infty$, $P_{n} \stackrel{w}{\rightarrow} P$, and $B_{q}^k \rightarrow B_q$. If $d(\cdot , \cdot)$ is continuous at $P_{B_q}$ and there is an integrable function $g(\cdot)$ such that $d(\vect{y},(P_{n})_{B_{q}^{k}}) \leq g(\vect{y})$ for all $\vect{y} \in \mathbb{R}^q$ and $n,k \in \mathbb{N}$, then $\sigma((P_{n})_{B_{q}^{k}})  \rightarrow \sigma(P_{B_q})$ as $n,k \to \infty$.
\end{corollary}

\begin{proof}
Using Proposition \ref{sm:PropConvPBq} we have that $(P_{n})_{B_{q}^{k}} \stackrel{w}{\rightarrow} P_{B_q}$ as $n,k \rightarrow \infty$. Now the result follows from Proposition \ref{sm:PropContSigmatP}.
\end{proof} \\

An immediate consequence of Corollary \ref{sm:CorContSigminPHalfSimp} and Corollary \ref{sm:CorContSigminPBq} is the following result for halfspace and simplicial depth.

\begin{corollary} \label{sm:CorgHalfSimp}
Suppose that $\vect{X}_{n} \sim P_{n}$ satisfies $\limsup_{n \to \infty} \E[\norm{\vect{X}_{n}}] < \infty$, $P_{n} \stackrel{w}{\rightarrow} P$, and $B_{q}^k \rightarrow B_q$. Let $g(\cdot)$ be an integrable function such that $d_H(\vect{y},(P_{n})_{B_{q}^{k}}) \leq g(\vect{y})$ (resp.\ $d_\Delta(\vect{y},(P_{n})_{B_{q}^{k}}) \leq g(\vect{y})$) for all $\vect{y} \in \mathbb{R}^q$ and $n,k \in \mathbb{N}$. If $P_{B_{q}}$ satisfies \eqref{sm:Prob0Hyperplanes} in $\mathbb{R}^q$, then, $\sigma_H((P_{n})_{B_{q}^{k}})  \rightarrow \sigma_H(P_{B_q})$ (resp.\ $\sigma_\Delta((P_{n})_{B_{q}^{k}})  \rightarrow \sigma_\Delta(P_{B_q})$) as $n,k \to \infty$.
\end{corollary}

As we see in the next result, the existence of $g(\cdot)$ in Corollary \ref{sm:CorgHalfSimp} is ensured if $(P_n)_{n=1}^{\infty}$ decays polynomially. By Remark \ref{remark_polynomial_decay}, a sufficient condition for this to hold is that $\limsup_{n \to \infty} \E[\norm{\vect{X}_{n}}^{\alpha}] < \infty$, for some $\alpha>1$, where $\vect{X}_{n} \sim P_{n}$. Clearly, this also implies that $\limsup_{n \to \infty} \E[\norm{\vect{X}_{n}}] < \infty$, thus, the assumption of Proposition \ref{sm:PropConvPBq} is satisfied. Furthermore, by Example 2.21 in \citet{vandervaar2000}, if $P_{n} \stackrel{w}{\rightarrow} P$ and $\limsup_{n \to \infty} \E[\norm{\vect{X}_{n}}^{\alpha}] < \infty$, then $\lim_{n \to \infty} \E[\norm{\vect{X}_{n}}]=\E[\norm{\vect{X}}]$, where $\vect{X} \sim P$.

\begin{corollary} \label{sm:CorPolDecayHalfSimp}
Suppose that $\vect{X}_{n} \sim P_{n}$ satisfies $\limsup_{n \to \infty} \E[\norm{\vect{X}_{n}}] < \infty$, $P_{n} \stackrel{w}{\rightarrow} P$, and $B_{q}^k \rightarrow B_q$. If $(P_n)_{n=1}^{\infty}$ decays polynomially and $P_{B_{q}}$ satisfies \eqref{sm:Prob0Hyperplanes} in $\mathbb{R}^q$, then $\sigma_H((P_{n})_{B_{q}^{k}})  \rightarrow \sigma_H(P_{B_q})$ and $\sigma_\Delta((P_{n})_{B_{q}^{k}})  \rightarrow \sigma_\Delta(P_{B_q})$ as $n,k \to \infty$.
\end{corollary}
\begin{proof}
  By Proposition \ref{sm:PropConvPBq}, $(P_{n})_{B_{q}^{k}} \stackrel{w}{\rightarrow} P_{B_q}$ as $n,k \rightarrow \infty$. We show next that, if $(P_n)_{n=1}^{\infty}$ decays polynomially in $\mathbb{R}^m$, then $((P_{n})_{B_{q}^{k}})_{n,k=1}^{\infty}$ decays polynomially in $\mathbb{R}^q$.  Let $\vect{X}_n \sim P_n$. Since, for all $k \in \mathbb{N}$ and $\vect{x} \in \mathbb{R}^m$, $\norm{B_q^k \vect{x}} \leq \norm{\vect{x}}$, it follows that, for all $t \geq 0$ and $k,n \in \mathbb{N}$,
\begin{equation*}
\bP(\norm{B_q^k \vect{X}_n} \geq t) \leq \bP(\norm{\vect{X}_n} \geq t) \leq \frac{C}{(1+t)^{\alpha}},
\end{equation*}
where $C>0$ and $\alpha>1$. The result now follows from Corollary \ref{sm:CorPolDecay}.
\end{proof} \\

As the next result shows, $P_{B_{q}}$ satisfies \eqref{sm:Prob0Hyperplanes} in Corollaries \ref{sm:CorgHalfSimp} and \ref{sm:CorPolDecayHalfSimp} whenever $P$ is absolutely continuous.

\begin{proposition} \label{proposition_absolutely_continuity_of_projection}
If $P$ is absolutely continuous in $\mathbb{R}^{m}$, then $P_{B_{q}}$ is absolutely continuous in $\mathbb{R}^{q}$ for all $B_{q} \in \mathcal{B}^{q \times m}$.
\end{proposition}

\begin{proof}[Proof of Proposition \ref{proposition_absolutely_continuity_of_projection}]
If $P$ has density $f(\cdot)$, then, for all Borel sets $A \subset \mathbb{R}^{q}$,
\begin{equation*}
      P_{B_{q}}(A) = P(B_{q}^{-1}A) = \int_{B_{q}^{-1}A} f(\vect{x}) \, d\vect{x} = \int_{A \times \mathbb{R}^{p}} g_{B_{q}}(\vect{y}) \, d\vect{y} = \int_{A} h_{B_{q}}(\vect{v}) \, d\vect{v},
\end{equation*}
where $B_{p}$ is a $p \times m$ matrix such that $U = \left(\begin{smallmatrix} B_{q} \\ B_{p} \end{smallmatrix}\right)^{-1}$ is orthogonal, $g_{B_{q}}(\cdot)=f(U(\cdot))$ and 
\begin{equation*}
   h_{B_{q}}(\vect{v}) = \int_{\mathbb{R}^{p}} g_{B_{q}}((\vect{v},\vect{w})) \, d \vect{w}.
\end{equation*}
It follows that $P_{B_{q}}$ is absolutely continuous in $\mathbb{R}^{q}$ with density $h_{B_{q}}(\cdot)$.
\end{proof}

\subsection{Asymptotic properties of the dispersion measure $\sigma(\cdot)$}
\label{sm:subsection_asymptotic_properties_of_dispersion_measure}

Let $(\vect{X}_i)_{i=1}^{\infty}$ be i.i.d.\ with distribution $P$ and for $n \in \mathbb{N}$ let $\hat{P}_n$ be the empirical measure based on $(\vect{X}_i)_{i=1}^{n}$.

\begin{proposition}[Continuity at $P$ w.r.t.\ empirical measure] \label{sm:PropContSigmatPemp}
Suppose that $d(\cdot , \cdot)$ is continuous at $P$ and let $g(\cdot)$ be an integrable (random) function such that $d(\vect{x},\hat{P}_n) \leq g(\vect{x})$ a.s.\ for all $n \in \mathbb{N}$ and $\vect{x} \in \mathbb{R}^m$. Then, $\sigma(\hat{P}_n)  \rightarrow \sigma(P)$ a.s.\ as $n \rightarrow \infty$.
\end{proposition}

\begin{proof}
\citet{varadarajan1958} shows that $\hat{P}_n \stackrel{w}{\rightarrow} P$ a.s. Therefore, $d(\cdot,\hat{P}_n)$ converges pointwise to $d(\cdot,P)$ a.s. Finally, Lebesgue's dominated convergence theorem implies that $\sigma(\hat{P}_n)  \rightarrow \sigma(P)$ a.s.
\end{proof}

\begin{remark}
The function $g(\cdot)$ in Proposition \ref{sm:PropContSigmatPemp} may depend on the sequence $(\hat{P}_n)_{n=1}^{\infty}$, but not on $n$: $\{ g(\vect{x}) \}_{\vect{x} \in \mathbb{R}^m}$ is a random field.
\end{remark}

The following corollary follows immediately from Proposition \ref{sm:PropContSigmatPemp}.

\begin{corollary} \label{sm:CorContSigmatPempHalfSimp}
Let $g(\cdot)$ be an integrable random field such that $d_H(\vect{x},\hat{P}_n) \leq g(\vect{x})$ (resp.\ $d_\Delta(\vect{x},\hat{P}_n) \leq g(\vect{x})$) a.s.\ for all $n \in \mathbb{N}$ and $\vect{x} \in \mathbb{R}^m$. If $P$ satisfies \eqref{sm:Prob0Hyperplanes}, then $\sigma_H(\hat{P}_n)  \rightarrow \sigma_H(P)$ (resp.\ $\sigma_\Delta(\hat{P}_n)  \rightarrow \sigma_\Delta(P)$) a.s.\ as $n \rightarrow \infty$.
\end{corollary}

As the next result shows, the existence of $g(\cdot)$ in Corollary \ref{sm:CorContSigmatPempHalfSimp} is ensured by the finiteness of a moment of order $\alpha>1$ of $\vect{X} \sim P$.

\begin{corollary} \label{sm:CorPolDecayContSigmatPempHalfSimp}
If $\vect{X} \sim P$ satisfies \eqref{sm:Prob0Hyperplanes} and $\E[\norm{\vect{X}}^{\alpha}] < \infty$ for some $\alpha>1$, then $\sigma_H(\hat{P}_n)  \rightarrow \sigma_H(P)$ and $\sigma_\Delta(\hat{P}_n)  \rightarrow \sigma_\Delta(P)$ a.s.\ as $n \rightarrow \infty$.
\end{corollary}
\begin{proof}
By \eqref{bound_halfspace_depth} and \eqref{bound_simplicial_depth} with $P$ replaced by $\hat{P}_{n}$, it holds that
\begin{equation*}
d_H(\vect{x},\hat{P}_n) \leq \frac{1}{n} \sum_{i=1}^n I( \norm{\vect{X}_i} \geq \norm{\vect{x}})
\end{equation*}
and
\begin{equation*}
d_\Delta(\vect{x},\hat{P}_n) \leq (m+1) \frac{1}{n} \sum_{i=1}^n I( \norm{\vect{X}_i} \geq \norm{\vect{x}}).
\end{equation*}
Now, 
\begin{equation*}
\frac{1}{n} \sum_{i=1}^n I( \norm{\vect{X}_i} \geq \norm{\vect{x}}) \leq \frac{1}{\norm{\vect{x}}^{\alpha}} \frac{1}{n} \sum_{i=1}^{n} \norm{\vect{X}_i}^{\alpha}
\end{equation*}
and by the strong law of large numbers, as $n \to \infty$,
\begin{equation} \label{sm:ConvEmpMeans}
\frac{1}{n} \sum_{i=1}^{n} \norm{\vect{X}_i}^{\alpha} \to \E \norm{\vect{X}}^{\alpha} \quad \text{a.s.}
\end{equation}
$\E[\norm{\vect{X}}^{\alpha}]<\infty$ implies that, there exists a random variable $C_{0} > \E\norm{\vect{X}}^{\alpha}$ (depending on the sequence $(\vect{X}_i)_{i=1}^{\infty}$) such that $\frac{1}{n} \sum_{i=1}^{n} \norm{\vect{X}_i}^{\alpha} \leq C_0$ a.s.\ for all $n \in \mathbb{N}$. Let $g(\cdot)$ be the random field given by 
\begin{equation} \label{dominating_function_projected_empirical_halfspace_depth}
g(\vect{x}) = \frac{C}{(1+\norm{\vect{x}})^{\alpha}},
\end{equation}
where $C=2^{\alpha} \max(1,C_0)$. Since $g(\cdot)$ (resp.\ $(m+1) g(\cdot)$) dominates $d_H(\cdot,\hat{P}_n)$ (resp.\ $d_\Delta(\cdot,\hat{P}_n)$) a.s., we conclude by Corollary \ref{sm:CorContSigmatPempHalfSimp} that $\sigma_H(\hat{P}_n)  \rightarrow \sigma_H(P)$ and $\sigma_\Delta(\hat{P}_n) \rightarrow \sigma_\Delta(P)$ a.s.
\end{proof} \\

As we show in the next corollaries, the above asymptotic results hold also in subspaces.

\begin{corollary} \label{sm:CorContSigmatPempBqHalfSimp}
  Suppose that $\vect{X} \sim P$ satisfies $\E[\norm{\vect{X}}]<\infty$, $B_{q}^k \rightarrow B_q$ and that $d( \cdot , \cdot)$ is continuous at $P_{B_q}$. If $g(\cdot)$ is an integrable random field such that $d(\vect{y},(\hat{P}_{n})_{B_{q}^{k}}) \leq g(\vect{y})$ a.s.\ for all $k,n \in \mathbb{N}$ and $\vect{y} \in \mathbb{R}^q$, then $\sigma((\hat{P}_{n})_{B_{q}^{k}})  \rightarrow \sigma(P_{B_q})$ a.s.\ as $k,n \to \infty$.
\end{corollary}

\begin{proof}
  Since $\vect{X}_{i} \sim P$ are i.i.d., by the strong law of large numbers, it holds that
\begin{equation} \label{strong_law_large_numbers}
  \int_{\mathbb{R}^m} \norm{\vect{x}} \, d \hat{P}_n(\vect{x}) = \frac{1}{n} \sum_{i=1}^{n} \norm{\vect{X}_{i}} \to \E[\norm{\vect{X}}] \text{ a.s.}
\end{equation}
Also, using that $\hat{P}_n \stackrel{w}{\rightarrow} P$ a.s.\ \citep{varadarajan1958}, we have that, for every Lipschitz continuous and bounded function $f : \mathbb{R}^q \rightarrow \mathbb{R}$,
\begin{equation} \label{weak_convergence_of_projected_empirical_measure}
\vert \int_{\mathbb{R}^m} f(B_{q}\vect{x}) \, d (\hat{P}_n(\vect{x}) - P(\vect{x})) \vert \rightarrow 0 \quad \text{a.s.}
\end{equation}
Using \eqref{f_lipschitz}, \eqref{strong_law_large_numbers}, and \eqref{weak_convergence_of_projected_empirical_measure}, we see, as in the proof of Proposition \ref{sm:PropConvPBq}, that $(\hat{P}_{n})_{B_{q}^{k}} \stackrel{w}{\rightarrow} P_{B_q}$ a.s.\ as $n,k \rightarrow \infty$. By Proposition \ref{sm:PropContSigmatPemp}, we conclude that $\sigma((\hat{P}_{n})_{B_{q}^{k}})  \rightarrow \sigma(P_{B_q})$ a.s.
\end{proof} \\

The next result shows convergence of the dispersion measure based on the empirical halfspace and simplicial depth.

\begin{corollary} \label{sm:CorPolDecayContSigminPempBqHalfSimp}
  Suppose that $\vect{X} \sim P$ satisfies $\E[\norm{\vect{X}}^{\alpha}] < \infty$ for some $\alpha>1$ and $B_{q}^k \rightarrow B_q$. If $P_{B_q}$ satisfies \eqref{sm:Prob0Hyperplanes} in $\mathbb{R}^q$, then $\sigma_H((\hat{P}_{n})_{B_{q}^{k}})  \rightarrow \sigma_H(P_{B_q})$ and $\sigma_\Delta((\hat{P}_{n})_{B_{q}^{k}})  \rightarrow \sigma_\Delta(P_{B_q})$ a.s.\ as $k,n \to \infty$.
\end{corollary}

Corollary \ref{sm:CorPolDecayContSigminPempBqHalfSimp} follows from Corollary \ref{sm:CorContSigmatPempBqHalfSimp} using the continuity of halfspace and simplicial depth at $P_{B_q}$. Since $\E [\norm{B_{q}^k \vect{X}}^{\alpha}] \leq \E [\norm{\vect{X}}^{\alpha}] < \infty$ for all $k \in \mathbb{N}$, the dominating functions for empirical halfspace and simplicial depth in $\mathbb{R}^{q}$ are $g(\cdot)$ and $(q+1)g(\cdot)$, where $g(\cdot)$ is given by \eqref{dominating_function_projected_empirical_halfspace_depth}. Finally, we recall that a sufficient condition for $P_{B_{q}}$ to satisfy \eqref{sm:Prob0Hyperplanes} is that $P$ is absolutely continuous (cf.\ Proposition \ref{proposition_absolutely_continuity_of_projection}).

\subsection{Existence and uniqueness of minimizers and maximizers for the dispersion measure $\sigma(\cdot)$}
\label{sm:subsection_existence_and_uniqueness_of_minimizers}

In this subsection we give conditions for the existence and uniqueness of minimizers and maximizers for the dispersion measure $\sigma(\cdot)$ with respect to subspaces $S_q \subset \mathbb{R}^m$. For $k,l \in \mathbb{N}$ let $(\mathcal{M}^{k \times l},\norm{\cdot}_{\mathcal{M}^{k \times l}})$ be the set of $k \times l$ matrices endowed with the matrix norm $\norm{M}_{\mathcal{M}^{k \times l}}= \sup_{\vect{z} \in \mathbb{R}^l \setminus \{ \vect{0} \}} \frac{\norm{M \vect{z}}}{\norm{\vect{z}}}$. We begin with the following lemma.

\begin{lemma} \label{sm:LemmaBqmCompact}
$\mathcal{B}^{q \times m}$ is a compact subset of $(\mathcal{M}^{q \times m},\norm{\cdot}_{\mathcal{M}^{q \times m}})$.
\end{lemma}

\begin{proof}
  It is enough to show that $\mathcal{B}^{q \times m}$ is closed and bounded. Let $\varphi \, : \, \mathcal{M}^{q \times m} \rightarrow \mathcal{M}^{q \times q}$ be given by $M \mapsto M M^{\top}$. If $M_{n} \rightarrow M$ as $n \to \infty$, then
\begin{align*}
  \norm{ \varphi(M_{n})-\varphi(M)}_{\mathcal{M}_{q \times q}} &\leq \norm{ M_{n} }_{\mathcal{M}^{q \times m}} \norm{ (M_{n} - M)^{\top}}_{\mathcal{M}^{m \times q}} \\
  &+ \, \norm{ M_{n} - M}_{\mathcal{M}^{q \times m}} \norm{M^{\top}}_{\mathcal{M}^{m \times q}} \rightarrow 0.
\end{align*}
Hence, $\varphi(\cdot)$ is continuous. Now, since $\{ I_q \}$ is closed in $(\mathcal{M}^{q \times q},\norm{\cdot}_{\mathcal{M}^{q \times q}})$, we have that $\mathcal{B}^{q \times m}=\varphi^{-1}(\{ I_q \})$ is closed in $(\mathcal{M}^{q \times m},\norm{\cdot}_{\mathcal{M}^{q \times m}})$. Finally, since, for all $B_q \in \mathcal{B}^{q \times m}$, $\norm{B_q \vect{z}} \leq \norm{\vect{z}}$, we have that
\begin{equation*}
\norm{B_q}_{\mathcal{B}^{q \times m}} = \sup_{\vect{z} \in \mathbb{R}^m \setminus \{ \vect{0} \} } \frac{\norm{B_q \vect{z}}}{\norm{\vect{z}}} \leq 1
\end{equation*}
and $\mathcal{B}^{q \times m}$ is bounded.
\end{proof}

\begin{proposition} \label{sm:PropMaxMinDispMeas}
Suppose that $\E[\norm{\vect{X}}]<\infty$, where $\vect{X} \sim P$, $d(\cdot , \cdot)$ is continuous at $P_{B_q}$ for all $B_q \in \mathcal{B}^{q \times m}$, and there exists an integrable function $g(\cdot)$ such that $\sup_{B_q \in \mathcal{B}^{q \times m}} d(\vect{y},P_{B_q}) \leq g(\vect{y})$ for all $\vect{y} \in \mathbb{R}^q$. Then, the function $\phi: \mathcal{B}^{q \times m} \rightarrow \mathbb{R}$ given by $B_q \mapsto \sigma(P_{B_q})$ is continuous and it assumes maximum and minimum values on $\mathcal{B}^{q \times m}$.
\end{proposition}

\begin{proof} 
By Proposition \ref{sm:PropConvPBq} with $P_{n} = P$, if $B_q^k \rightarrow B_q$, then $P_{B_{q}^{k}} \stackrel{w}{\rightarrow} P_{B_q}$. Since $d(\cdot , \cdot)$ is continuous at $P_{B_q}$, Proposition \ref{sm:PropContSigmatP} implies that $\phi(\cdot)$ is continuous. Since $\mathcal{B}^{q \times m}$ is compact by Lemma \ref{sm:LemmaBqmCompact}, $\phi(\cdot)$ assumes maximum and minimum values on $\mathcal{B}^{q \times m}$.
\end{proof} \\

Let $\phi_H \, : \, \mathcal{B}^{q \times m} \rightarrow \mathbb{R}$ be given by $B_q \mapsto \sigma_H(P_{B_q})$ and $\phi_\Delta \, : \, \mathcal{B}^{q \times m} \rightarrow \mathbb{R}$ be given by $B_q \mapsto \sigma_\Delta(P_{B_q})$. We recall from Remark \ref{remark_polynomial_decay} that a sufficient condition for $\vect{X} \sim P$ to decay polynomially is that $\E[\norm{\vect{X}}^{\alpha}]<\infty$ for some $\alpha>1$.

\begin{corollary} \label{sm:CorMaxMinDispMeas}
  If $\vect{X} \sim P$ is absolutely continuous, decays polynomially and satisfies $\E[\norm{\vect{X}}]<\infty$, then the functions $\phi_H(\cdot)$ and $\phi_\Delta(\cdot)$ are continuous and assume maximum and minimum value in $\mathcal{B}^{q \times m}$.
\end{corollary}

\begin{proof}
  We first observe that, by Proposition \ref{proposition_absolutely_continuity_of_projection}, $P_{B_q}$ is absolutely continuous for all $B_q \in \mathcal{B}^{q \times m}$. In particular, $P_{B_q}$ satisfies \eqref{sm:Prob0Hyperplanes} in $\mathbb{R}^q$, and $d_H(\cdot,\cdot)$, $d_\Delta(\cdot,\cdot)$ are continuous at $P_{B_q}$. Let  $\vect{X} \sim P$. Since $\norm{B_q \vect{X}} \leq \norm{\vect{X}}$ for all $B_q \in \mathcal{B}^{q \times m}$, and $P$ decays polynomially, there exist $C>0$ and $\alpha>1$ such that
\begin{equation*}
P(\norm{{B_q} \vect{X}} \geq t) \leq \bP(\norm{\vect{X}} \geq t) \leq \frac{C}{(1+t)^{\alpha}},
\end{equation*}
that is, $\{ P_{B_q} \}_{B_q \in \mathcal{B}^{q \times m}}$ decays polynomially. It follows that
\begin{equation*}
\sup_{B_q \in \mathcal{B}^{q \times m}} d_H(\vect{y},P_{B_q}) \leq \sup_{B_q \in \mathcal{B}^{q \times m}} \bP(\norm{{B_q} \vect{X}} \geq \norm{\vect{y}}) \leq \frac{C}{(1+\norm{\vect{y}})^{\alpha}}
\end{equation*}
and
\begin{equation*}
\sup_{B_q \in \mathcal{B}^{q \times m}} d_\Delta(\vect{y},P_{B_q}) \leq (q+1) \sup_{B_q \in \mathcal{B}^{q \times m}} \bP(\norm{{B_q} \vect{X}} \geq \norm{\vect{y}}) \leq \frac{(q+1)C}{(1+\norm{\vect{y}})^{\alpha}}.
\end{equation*}
The result now follows from Proposition \ref{sm:PropMaxMinDispMeas}.
\end{proof} \\

Corollary \ref{sm:CorMaxMinDispMeas} gives the existence of maxima and minima for the dispersion measures $\sigma_H(\cdot)$ and $\sigma_\Delta(\cdot)$ with respect to subspaces $S_q \subset \mathbb{R}^m$. In general, they are not unique. For example, if $P$ is spherically symmetric then $\{ P_{B_q} \}_{B_q \in \mathcal{B}^{q \times m}}$ have all the same distribution. 

\begin{definition}[Unique maximizer and minimizer]
$B_{q}^{*}$ is a unique maximizer (resp.\ minimizer) for $\phi(\cdot)$ with respect to $P$ if $\sigma(P_{B_{q}^{*}}) > \sigma(P_{B_{q}})$ (resp.\ $\sigma(P_{B_{q}^{*}}) < \sigma(P_{B_{q}})$) for all $B_{q} \in \mathcal{B}^{q \times m} \setminus \{ B_{q}^{*} \}$.
\end{definition}

\begin{proposition} \label{sm:PropUniqueMaxMinDispMeas}
  Suppose that $\vect{X} \sim P$ is absolutely continuous and satisfies $\E[\norm{\vect{X}}^{\alpha}] < \infty$ for some $\alpha>1$. Let $B_{q}^{*}$ be a unique maximizer (resp.\ minimizer) for $\phi_L(\cdot)$, $L=H,\Delta$. If $(\hat{B}_q^{n})_{n=1}^{\infty}$ satisfies
\begin{equation*}
  \sigma_L((\hat{P}_{n})_{\hat{B}_{q}^{n}})=\max_{B_q \in \mathcal{B}^{q \times m}} \sigma_L((\hat{P}_{n})_{B_{q}}) \text{ (resp.\ } =\min_{B_q \in \mathcal{B}^{q \times m}} \sigma_L((\hat{P}_{n})_{B_{q}})\text{)},
\end{equation*}
then $\hat{B}_q^{n} \to B_{q}^{*}$ a.s.\ as $n \to \infty$.
\end{proposition}

\begin{proof}
  Since $B_{q}^{*}$ is a unique maximizer or minimizer it holds that, for all $\epsilon>0$,
\begin{equation*}
  \delta = \inf_{B_q \in \mathcal{B}^{q \times m} \, : \, \norm{B_q^{*}-B_q}_{\mathcal{B}^{q \times m}} \geq \epsilon} \abs{\sigma_L(P_{B_{q}^{*}})-\sigma_L(P_{B_q})}  > 0
\end{equation*}
and, for all $k \in \mathbb{N}$,
\begin{align*}
  \sup_{n \geq k} \abs{\sigma_L(P_{B_q^{*}})-\sigma_L(P_{\hat{B}_q^{n}})} & \leq \sup_{n \geq k} \abs{\sigma_L(P_{B_q^{*}})-\sigma_L((\hat{P}_{n})_{B_{q}^{*}})} \\
  &+ \sup_{n \geq k} \abs{\sigma_L((\hat{P}_{n})_{\hat{B}_{q}^{n}})-\sigma_L(P_{\hat{B}_q^{n}})}.
\end{align*}
Since
\begin{equation*}
  \sup_{n \geq k} \abs{\sigma_L((\hat{P}_{n})_{\hat{B}_{q}^{n}})-\sigma_L(P_{\hat{B}_q^{n}})} \leq \sup_{n \geq k} \sup_{B_q \in \mathcal{B}^{q \times m}} \abs{\sigma_L((\hat{P}_{n})_{B_{q}}))-\sigma_L(P_{B_q})},
\end{equation*}
we have that
\begin{align*}
\bP( \sup_{n \geq k} \norm{B_q^{*}-\hat{B}_q^{n}}_{\mathcal{B}^{q \times m}} \geq \epsilon) & \leq \bP( \sup_{n \geq k} \abs{\sigma_L(P_{B_q^{*}})-\sigma_L(P_{\hat{B}_q^{n}})} \geq \delta) \\
&\leq \bP( \sup_{n \geq k} \abs{\sigma_L(P_{B_q^{*}})-\sigma_L((\hat{P}_{n})_{B_{q}^{*}})} \geq \delta/2) \\
&+ \bP( \sup_{n \geq k} \sup_{B_q \in \mathcal{B}^{q \times m}} \abs{\sigma_L((\hat{P}_{n})_{B_{q}}))-\sigma_L(P_{B_q})} \geq \delta/2).
\end{align*}
We notice that, by Corollary \ref{sm:CorPolDecayContSigminPempBqHalfSimp}, $\abs{\sigma_L(P_{B_q^{*}})-\sigma_L((\hat{P}_{n})_{B_{q}^{*}})} \to 0$ a.s.\ as $n \to \infty$. To conclude the proof, it is enough to show that, as $n \to \infty$,
\begin{equation} \label{sm:EqUnifConvInBq}
\sup_{B_q \in \mathcal{B}^{q \times m}} \abs{\sigma_L((\hat{P}_{n})_{B_{q}}))-\sigma_L(P_{B_q})} \to 0 \quad \text{a.s.}
\end{equation}
Clearly,
\begin{equation*}
\sup_{B_q \in \mathcal{B}^{q \times m}} \abs{\sigma_L((\hat{P}_{n})_{B_{q}}))-\sigma_L(P_{B_q})} \leq \int_{\mathbb{R}^q} \sup_{B_q \in \mathcal{B}^{q \times m}} \abs{d_L(\vect{y},(\hat{P}_{n})_{B_{q}})-d_L(\vect{y},P_{B_q})} \, d \vect{y}.
\end{equation*}
Using that $\E[\norm{\vect{X}}^{\alpha}] < \infty$, as in the proof of Corollaries \ref{sm:CorPolDecayContSigmatPempHalfSimp}-\ref{sm:CorPolDecayContSigminPempBqHalfSimp}, we see that $\sup_{B_q \in \mathcal{B}^{q \times m}} d_L(\cdot,(\hat{P}_{n})_{B_{q}}) \leq g_{1,L}(\cdot)$ a.s., where $g(\cdot)$ is given in \eqref{dominating_function_projected_empirical_halfspace_depth}, $g_{1,H}(\cdot)=g(\cdot)$, and $g_{1,\Delta}(\cdot)=(q+1) g(\cdot)$. Similarly, as in the proof of Corollaries \ref{sm:CorPolDecay}-\ref{sm:CorPolDecayHalfSimp}, it follows that $\sup_{B_q \in \mathcal{B}^{q \times m}} d_L(\cdot ,P_{B_q}) \leq g_{2,L}(\cdot)$, where $g(\cdot)$ is now given in \eqref{dominating_function_empirical_halfspace_depth}, $g_{2,H}(\cdot)=g(\cdot)$, and $g_{2,\Delta}(\cdot)=(q+1) g(\cdot)$. Therefore,
\begin{equation*}
  \sup_{B_q \in \mathcal{B}^{q \times m}} \abs{d_L(\cdot ,(\hat{P}_{n})_{B_{q}}) - d_L(\cdot , P_{B_q})} \leq g_{1,L}(\cdot)+g_{2,L}(\cdot),
\end{equation*}
  where $g_{1,L}(\cdot)+g_{2,L}(\cdot)$ is integrable. Hence, \eqref{sm:EqUnifConvInBq} follows from Lebesgue's dominated convergence theorem if we show that, for $\vect{y} \in \mathbb{R}^q$ and $L=H,\Delta$, as $n \to \infty$,
\begin{equation} \label{sm:EqUnifConvDepthInBq}
\sup_{B_q \in \mathcal{B}^{q \times m}} \abs{d_L(\vect{y},(\hat{P}_{n})_{B_{q}})-d_L(\vect{y},P_{B_q})} \to 0 \quad \text{a.s.}
\end{equation}
Let $\mathcal{A}$ be the class given by intersecting of at most $q+1$ halfspaces in $\mathbb{R}^m$ and notice that $\mathcal{A}$ is Glivenko-Cantelli, i.e.\ $\sup_{A \in \mathcal{A}} \abs{\hat{P}_{n}(A)-P(A)} \to 0$ a.s.\ as $n \to \infty$. Since, for any $B_q \in \mathcal{B}^{q \times m}$ and halfspace $H \subset \mathbb{R}^q$, the preimage $B_q^{-1}H \in \mathcal{A}$, we have that, as $n \to \infty$,
\begin{align*}
\sup_{B_q \in \mathcal{B}^{q \times m}} \abs{d_H(\vect{y},(\hat{P}_{n})_{B_{q}})-d_H(\vect{y},P_{B_q})} &\leq \sup_{B_q \in \mathcal{B}^{q \times m}} \sup_{\vect{v} \in \mathcal{S}^{q-1}} \abs{(\hat{P}_{n})_{B_{q}}(H_{\vect{v},\vect{y}})-P_{B_q}(H_{\vect{v},\vect{y}})} \\
&= \sup_{B_q \in \mathcal{B}^{q \times m}} \sup_{\vect{v} \in \mathcal{S}^{q-1}} \abs{\hat{P}_{n}(B_q^{-1} H_{\vect{v},\vect{y}})-P(B_q^{-1} H_{\vect{v},\vect{y}})} \\
&\leq \sup_{A \in \mathcal{A}} \abs{\hat{P}_{n}(A)-P(A)} \to 0 \quad \text{a.s.}
\end{align*}
Since every simplex $\Delta \subset \mathbb{R}^q$ is given by the intersection of $q+1$ halfspaces $H_i$, i.e.\ $\Delta = \cap_{i=1}^{q+1}H_i$, the preimage $B_q^{-1} \Delta=\cap_{i=1}^{q+1} (B_q^{-1} H_i)$ is the intersection of $q+1$ halfspaces in $\mathbb{R}^m$. As in \citet{dumbgen1992}[Theorem 1], \citet{nagy2016}[Theorem A.7], we conclude that, as $n \to \infty$,
\begin{equation*}
\sup_{B_q \in \mathcal{B}^{q \times m}} \abs{d_\Delta(\vect{y},(\hat{P}_{n})_{B_{q}})-d_\Delta(\vect{y},P_{B_q})} \leq (q+1) \sup_{A \in \mathcal{A}} \abs{\hat{P}_{n}(A)-P(A)}  \to 0 \quad \text{a.s.}
\end{equation*}
\end{proof}

\subsection{Monotonicity of the dispersion measure with respect to the dimension} \label{sm:subsec:monotonicity}

In this subsection we show that the dispersion measure is non-increasing with respect to the projecting dimension if the support of $P$ is contained in the ball $B_{r}(0)$ centered at the origin with radius $r$ and $r$ is sufficiently small.

\begin{proposition} \label{sm:prop_monotonicity}
  Let $r_{H}=1/\sqrt{\pi}$ and $r_{\Delta}=1/(\sqrt{\pi}(m+1))$. If the support of $P$ is contained in $B_{r_{L}}(0)$ for $L \in \{H,\Delta\}$, then $\sigma_L(P) \leq \sigma_L(P_{B_{q}})$ for all dimensions $q \in \{1,\dots,m\}$ and all projection matrices $B_{q}$. In particular,
\begin{align*}
  &\min_{(S_{p^{\prime}}, S_{q^{\prime}}) \in \mathcal{S}_{p^\prime,q^\prime}} \sigma_{L}(P_{B_{q^{\prime}}}) \leq \min_{(S_p, S_q) \in \mathcal{S}_{p,q}} \sigma_{L}(P_{B_{q}}) \text{ and } \\
  &\max_{(S_{p^{\prime}}, S_{q^{\prime}}) \in \mathcal{S}_{p^\prime,q^\prime}} \sigma_{L}(P_{B_{q^{\prime}}}) \leq \max_{(S_p, S_q) \in \mathcal{S}_{p,q}} \sigma_{L}(P_{B_{q}})
\end{align*}
for all $1 \leq q \leq q^{\prime} \leq m$.
\end{proposition}

\begin{proof}
  We first notice that halfspace depth is non-increasing with respect to the projecting dimension $q$. Indeed, we have
\begin{align*}
  d_{H}(\vect{x},P) &= \inf_{\vect{u} \in S^{m-1}} \bP(\vect{X} \in H_{\vect{x},\vect{u}}) \\
  &\leq \inf_{\vect{u} \in S^{q-1}} \bP(B_{q} \vect{X} \in H_{B_{q} \vect{x},\vect{u}}) = d_{H}(B_{q} \vect{x},P_{B_{q}}).
\end{align*}
Using that the support of $P$ is contained in $B_{r_{H}}(0)$ and a change of variable, we obtain
\begin{align*}
  \sigma_{H}(P) &\leq \int_{B_{r_{H}}(0)} d_{H}(B_{q} \vect{x},P_{B_{q}}) \ d\vect{x} \\
  &= \int_{B_{r_{H}}(0)} d_{H}(\vect{x}_{q},P_{B_{q}}) \ d\vect{x},
\end{align*}
where $\vect{x}_{q}$ is given by the first $q$ components of $\vect{x}$. Using the coarea formula we have
\begin{equation*}
  \int_{B_{r_{H}}(0)} d_{H}(\vect{x}_{q},P_{B_{q}}) \ d\vect{x} = \int_{B_{r_{H}}(0)} d_{H}(\vect{x}_{q},P_{B_{q}}) H^{m-q}(B_{\sqrt{r_{H}^{2}-\norm{\vect{x}_{q}}^2}}(0)) \ d\vect{x}_{q},
\end{equation*}
where $\norm{\cdot}$ is the Euclidean norm and the two balls on the right-hand side are in $\mathbb{R}^{q}$ and $\mathbb{R}^{m-q}$, respectively. The volume of the second ball is
\begin{equation*}
  H^{m-q}(B_{\sqrt{r_{H}^{2}-\norm{\vect{x}_{q}}^2}}(0)) = \frac{\pi^{(m-q)/2}}{\Gamma((m-q)/2+1)} (r_{H}^{2}-\norm{\vect{x}_{q}}^2)^{(m-q)/2}.
\end{equation*}
Since this is bounded above by one, we obtain
\begin{equation*}
  \sigma_{H}(P) \leq \int_{B_{r_{H}}(0)} d_{H}(\vect{x}_{q},P_{B_{q}}) \ d\vect{x}_{q} = \sigma_{H}(P_{B_{q}}).
\end{equation*}
Turning to the simplicial depth, using that $\vect{x} \in \Delta[\vect{X}_{1},\dots,\vect{X}_{m+1}]$ implies $B_q \vect{x} \in \Delta[B_q \vect{X}_{i_1},\dots, B_q \vect{X}_{i_{q+1}}]$ for some indices $i_{1}, \dots, i_{q+1} \in \{ 1, \dots, m+1 \}$, we have
\begin{equation*}
  d_{\Delta}(\vect{x},P) \leq { m+1 \choose q+1 } d_{\Delta}(B_{q} \vect{x},P_{B_{q}}).
\end{equation*}
Using that ${ m+1 \choose q+1 } \leq (m+1)^{m-q}$ we obtain as above that $\sigma_{\Delta}(P) \leq \sigma_{\Delta}(P_{B_{q}})$. Finally, let $B_{q}$ be a $q \times m$ matrix with orthonormal row vectors and $C_{q^\prime - q}$ be a $(q^\prime - q) \times m$ matrix such that the matrix
\begin{equation*}
  \tilde{B}_{q^{\prime}} = \begin{pmatrix} B_{q} \\ C_{q^\prime - q} \end{pmatrix}
\end{equation*}  
has again orthonormal rows. Using the first part of the proposition with $P$ replaced by $P_{\tilde{B}_{q^{\prime}}}$, we obtain that $\sigma_{L}(P_{\tilde{B}_{q^{\prime}}}) \leq \sigma_{L}(P_{B_{q}})$. For a minimizer $B_{q^\prime}^{\ast}$ of $B_{q^\prime} \mapsto \sigma_{L}(P_{B_{q^\prime}})$ and a maximizer $B_{q}^{\ast \ast}$ of $B_{q} \mapsto \sigma_{L}(P_{B_{q}})$, it holds that
\begin{equation*}
  \sigma_{L}(P_{B_{q^\prime}^{\ast}}) \leq \sigma_{L}(P_{\tilde{B}_{q^{\prime}}}) \leq \sigma_{L}(P_{B_{q}}) \leq \sigma_{L}(P_{B_{q}^{\ast \ast}}).
\end{equation*}
Since $B_{q}$ and $\tilde{B}_{q^{\prime}}$ are arbitrary, we conclude that
\begin{equation*}
  \sigma_{L}(P_{B_{q^\prime}^{\ast}}) \leq \min_{(S_p, S_q) \in \mathcal{S}_{p,q}} \sigma_{L}(P_{B_{q}}) \text{ and } \max_{(S_{p^{\prime}}, S_{q^{\prime}}) \in \mathcal{S}_{p^\prime,q^\prime}} \sigma_{L}(P_{\tilde{B}_{q^{\prime}}}) \leq \sigma_{L}(P_{B_{q}^{\ast \ast}}).
\end{equation*}
\end{proof}

\subsection{Mixture of bivariate normal distributions with means the vertices of a square} \label{sm:sec:mixture_of_normal_distributions}

In this subsection, we consider a mixture of four bivariate normal distributions with equal weights, variance $\eta^{2} I$, and means the vertices $(1,1)$, $(1,-1)$, $(-1,1)$, and $(-1,-1)$ of a square with side length $2$. We compute the dispersion measure based on the halfspace depth of the projected distribution along the direction $B_{1}(u) = (u, \sqrt{1-u^{2}})$ for all $u \in (-1,1)$. For this purpose, we use that the projected mixture distribution has median zero and the projections of each component have means $\mu_{1}(u) = u+\sqrt{1-u^{2}}$, $\mu_{2}(u) = u-\sqrt{1-u^{2}}$, $\mu_{3}(u) = -\mu_{2}(u)$, $\mu_{4}(u) = -\mu_{1}(u)$, and variance $\eta^{2}$. Using the halfspace depth and the zero median, Proposition 17 of \citet{romanazzi2009} shows that the dispersion measure is equal to the raw absolute moment of the projected distribuion which, in turn, yields
\begin{equation*}
  \sigma(F_{B_{1}(u)}) = \frac{1}{4} \sum_{i=1}^{4} \psi(\mu_{i}(u)),
\end{equation*}
where for $\mu \in \mathbb{R}$
\begin{equation*}
  \psi(\mu) = \mu \left( 2 \Phi \left(\frac{\mu}{\eta} \right) -1 \right) + 2\eta \phi \left( \frac{\mu}{\eta} \right),
\end{equation*}
and $\phi$ and $\Phi$ are the cumulative distribution function and the density function of the standard normal (see (3) of \citet{leone1961}). The function $u \mapsto \sigma(F_{B_{1}(u)})$ is plotted in Figure \ref{figure_dispersion_mixture_normal} of the main paper. Since $\psi^{\prime}(u) = 2 \Phi \left(\frac{\mu}{\eta} \right) -1$ and $\sum_{i=1}^{4} \mu_{i}^{\prime}(u) = 0$, differentiating $\sigma(F_{B_{1}(u)})$ with respect to $u$ yields 
\begin{align*}
  \frac{1}{4} \sum_{i=1}^{4} \psi^{\prime}(\mu_{i}(u)) \mu_{i}^{\prime}(u) &= \frac{1}{2} \sum_{i=1}^{4} \Phi \left(\frac{\mu_{i}(u)}{\eta} \right) \mu_{i}^{\prime}(u) \\
  &= \frac{1}{2} \sum_{i=1}^{2} \left( \Phi \left(\frac{\mu_{i}(u)}{\eta} \right) - \Phi \left(\frac{- \mu_{i}(u)}{\eta} \right) \right) \mu_{i}^{\prime}(u),
\end{align*}
where $\mu_{1}^{\prime}(u) = 1 - u/\sqrt{1-u^{2}}$ and $\mu_{2}^{\prime}(u) = 1 + u/\sqrt{1-u^{2}}$. The derivative of $\sigma(F_{B_{1}(u)})$ is plotted in Figure \ref{figure_dispersion_mixture_normal_derivative}.
\begin{figure}
\centering{    
  \includegraphics[width=0.69\textwidth]{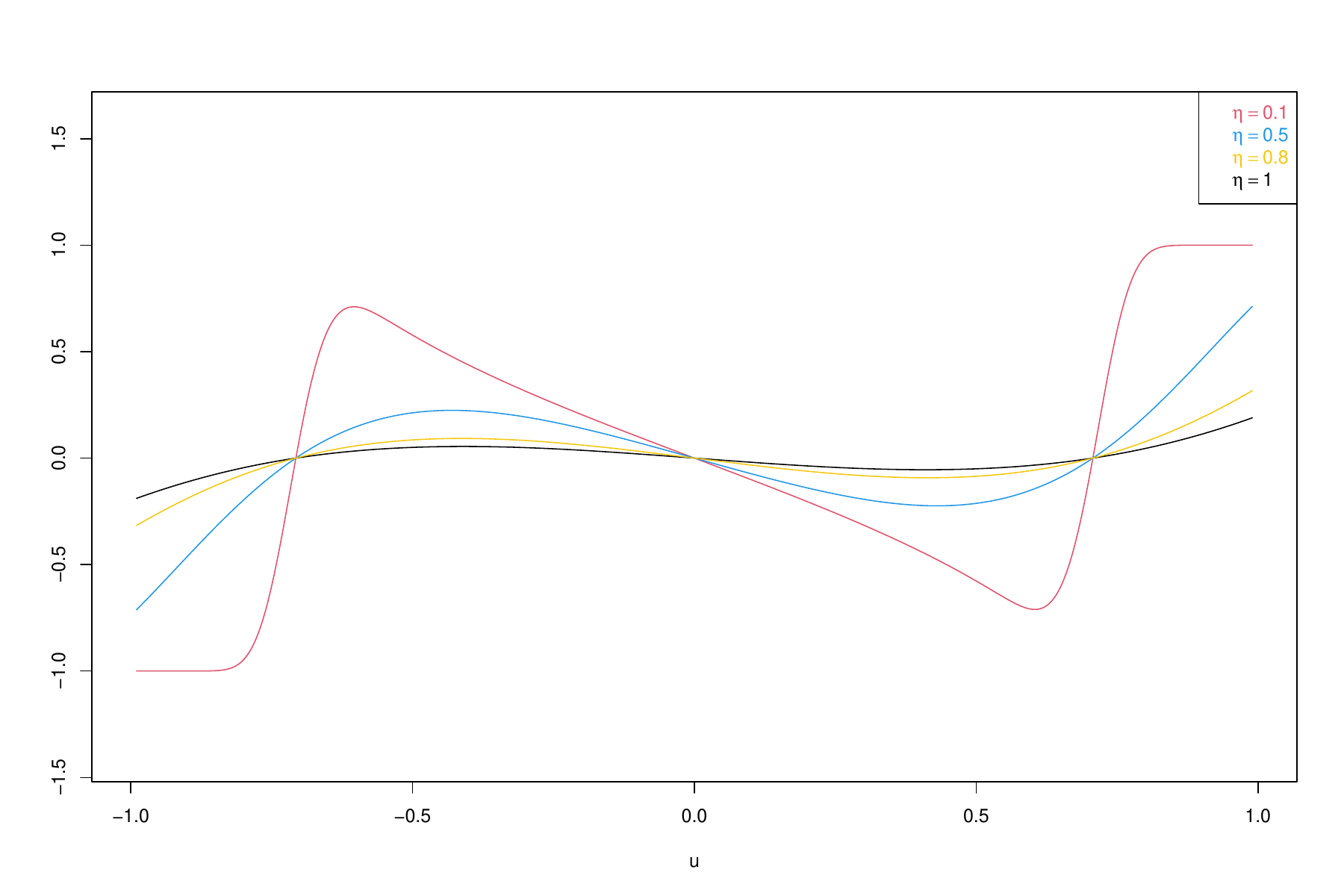}
}
\caption{Mixture of normal distributions. Derivative of the dispersion measure of the projected distribution, as a function of $u \in [-1,1]$, of a mixture of normal distributions with variance $\eta^{2} I$ and means the vertices $(1,1)$, $(1,-1)$, $(-1,1)$, and $(-1,-1)$ of a square. The analysis is performed using halfspace depth.}
\label{figure_dispersion_mixture_normal_derivative}
\end{figure}
One can show that the derivative is zero for $u = -1/\sqrt{2}, 0, 1/\sqrt{2}$, negative for $u \in (-1,-1/\sqrt{2}) \cup (0,1/\sqrt{2})$ and positive for $u \in (-1/\sqrt{2},0) \cup (1/\sqrt{2}, 1)$. Now, for $u=-1,0,1$ the dispersion measure is equal to
\begin{equation} \label{max}
  2 \eta \phi \left( \frac{1}{\eta} \right) + \Phi\left( \frac{1}{\eta} \right) - \Phi \left( -\frac{1}{\eta} \right),
\end{equation}
whereas for $u=-1/\sqrt{2}, 1/\sqrt{2}$ it takes the value
\begin{equation} \label{min}
  \eta \biggl( \phi \biggl( \frac{\sqrt{2}}{\eta} \biggr) + \phi(0) \biggr) + \frac{\sqrt{2}}{2} \biggl( \Phi \biggl( \frac{\sqrt{2}}{\eta} \biggr) - \Phi \biggl( -\frac{\sqrt{2}}{\eta} \biggr) \biggr).
\end{equation}
Since \eqref{min} is strictly smaller than \eqref{max}, the directions $B_{1}(-1) = (-1,0)$, $B_{1}(0) = (0,1)$, and $B_{1}(1) = (1,0)$ yield local maxima, whereas the directions $B_{1}(-1/\sqrt{2}) = (-1/\sqrt{2},1/\sqrt{2})$ and $B_{1}(1/\sqrt{2}) = (1/\sqrt{2},1/\sqrt{2})$ correspond to local minima. This shows that maximization of the dispersion measure yields the directions given by the sides of the square whereas minimization gives the directions given by the diagonals
of the square. The same calculations hold by letting $\eta\downarrow 0$ for the uniform distribution on the vertices of the square.

\section{Further examples}
\label{sm:sec:further_examples}

In Section \ref{sec:data_analysis} of the main paper we have analyzed POD data sets concerning imports in the European Union (UE) of a product to P from an origin (O) to a destination (D). We provide here the analysis of two further POD data sets, POD 30 and POD 54, see Figure \ref{sm:figure_pod_halfspace}.
\begin{figure}
\centering{  
  \includegraphics[width=0.32\textwidth]{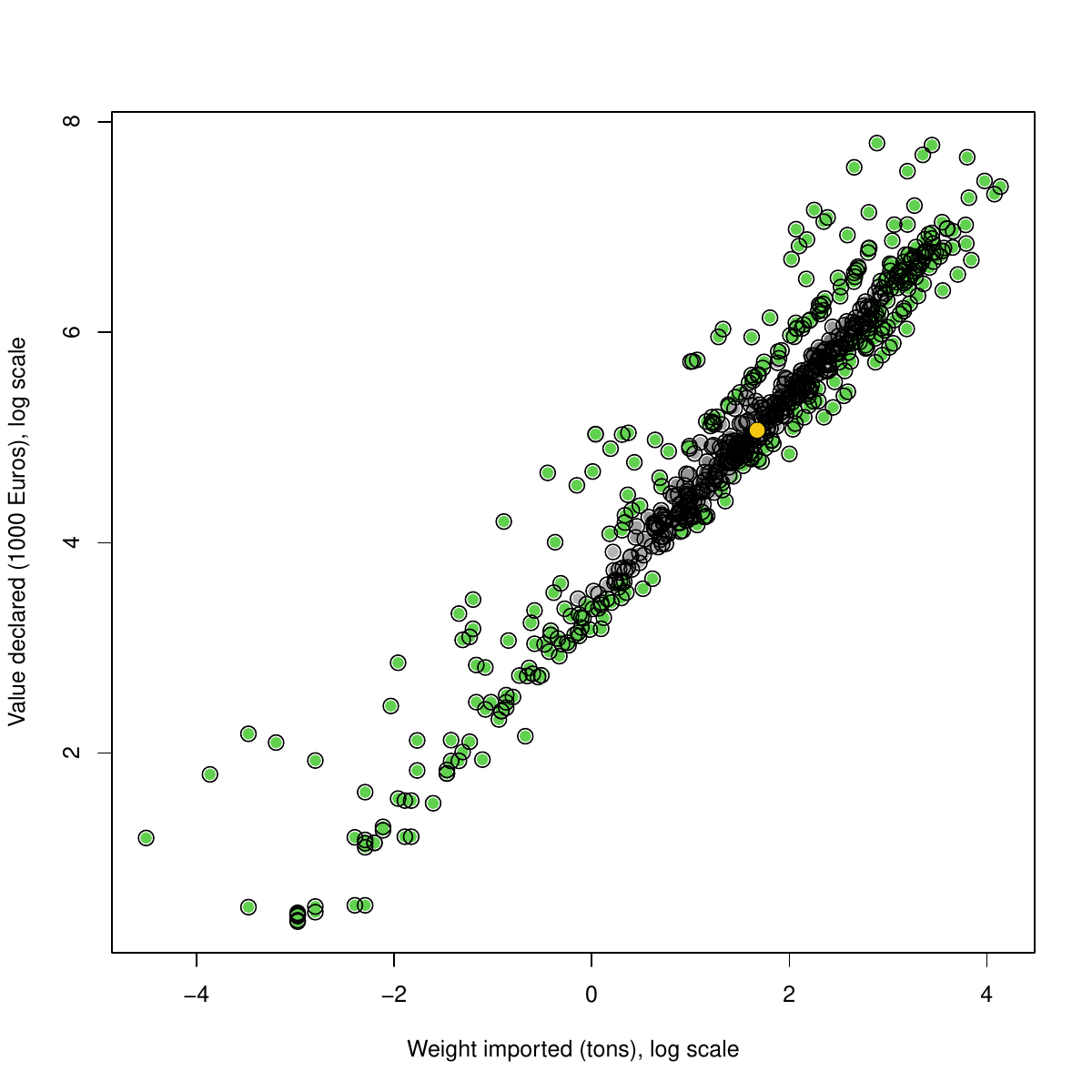}
  \includegraphics[width=0.32\textwidth]{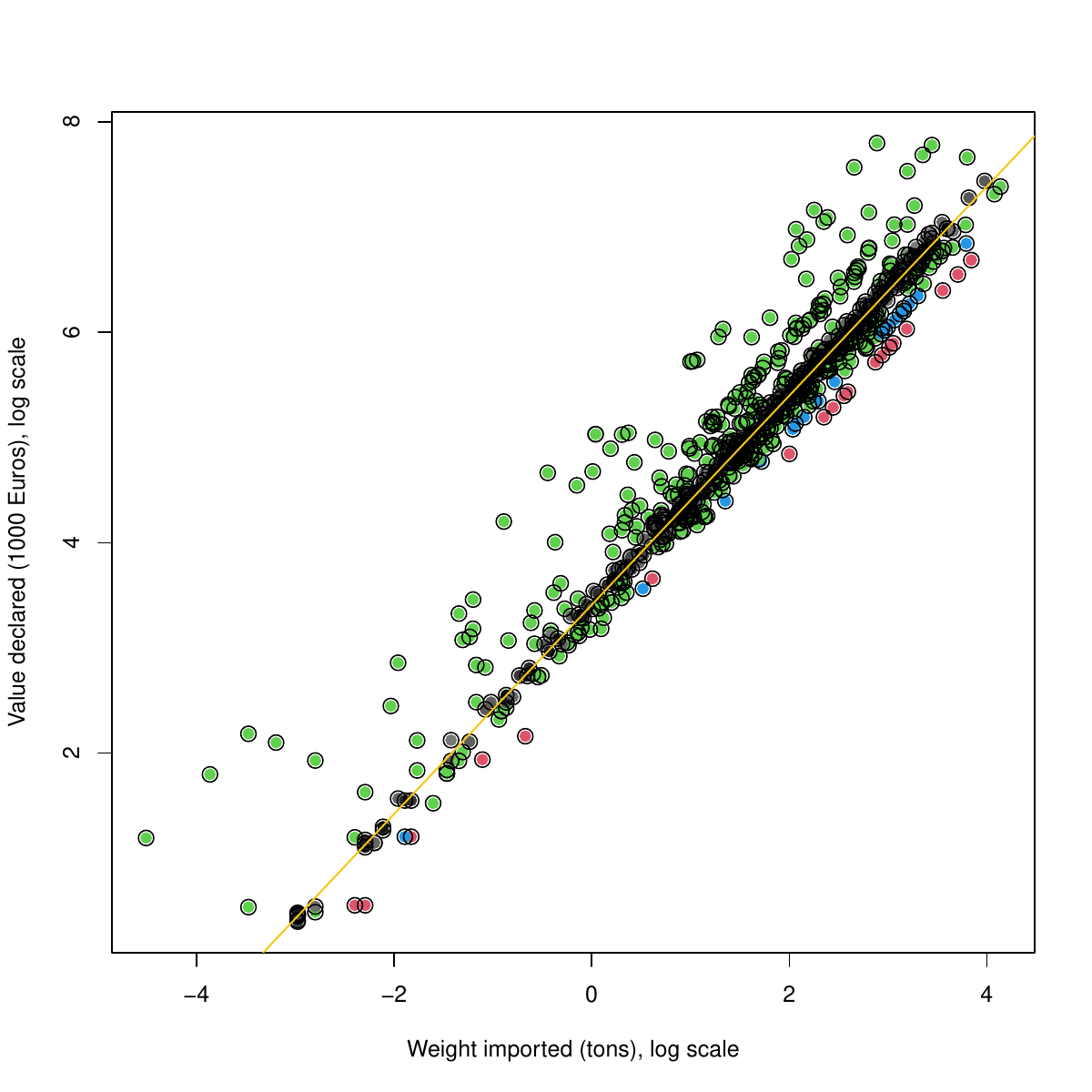}
  \includegraphics[width=0.32\textwidth]{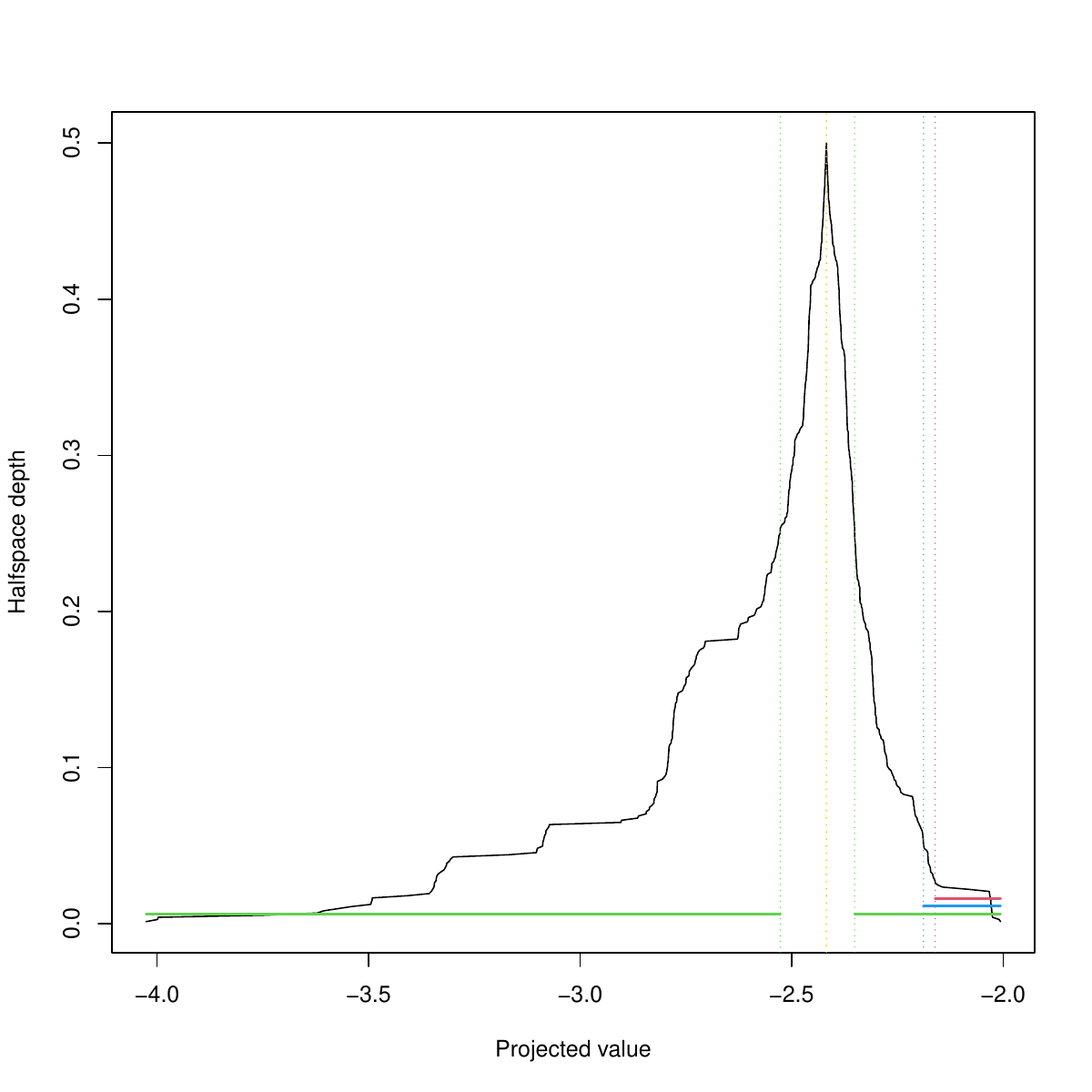}
  \includegraphics[width=0.32\textwidth]{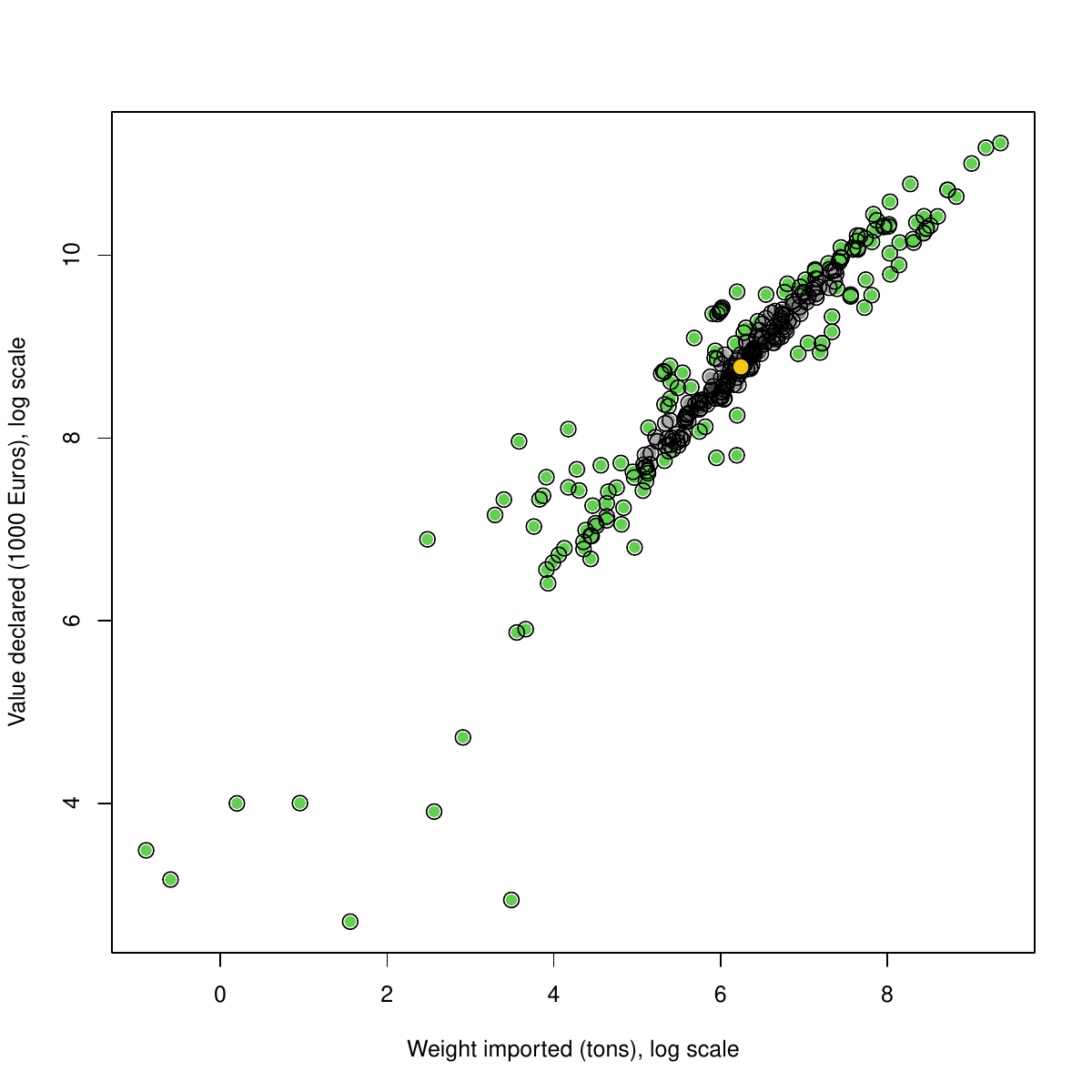}
  \includegraphics[width=0.32\textwidth]{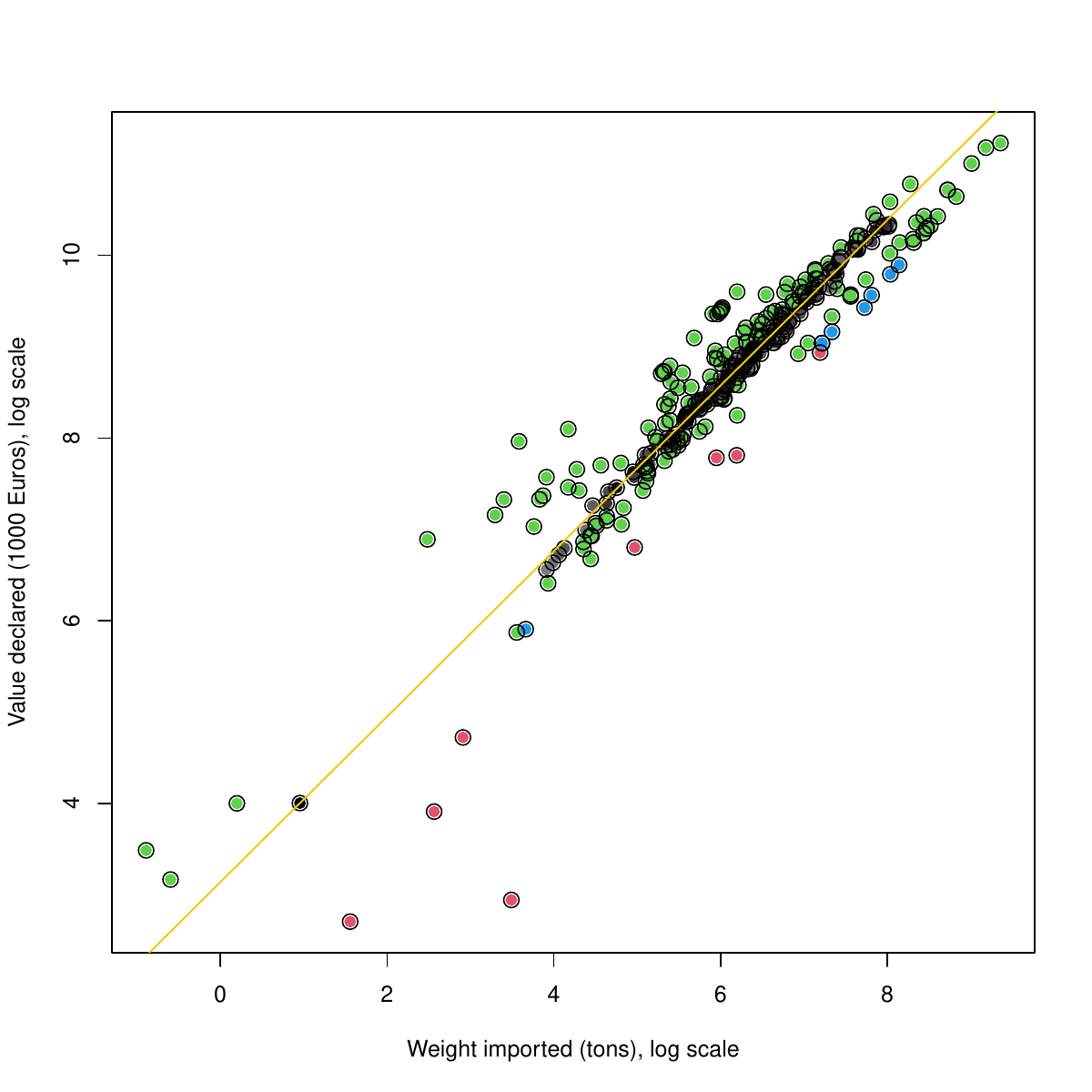}
  \includegraphics[width=0.32\textwidth]{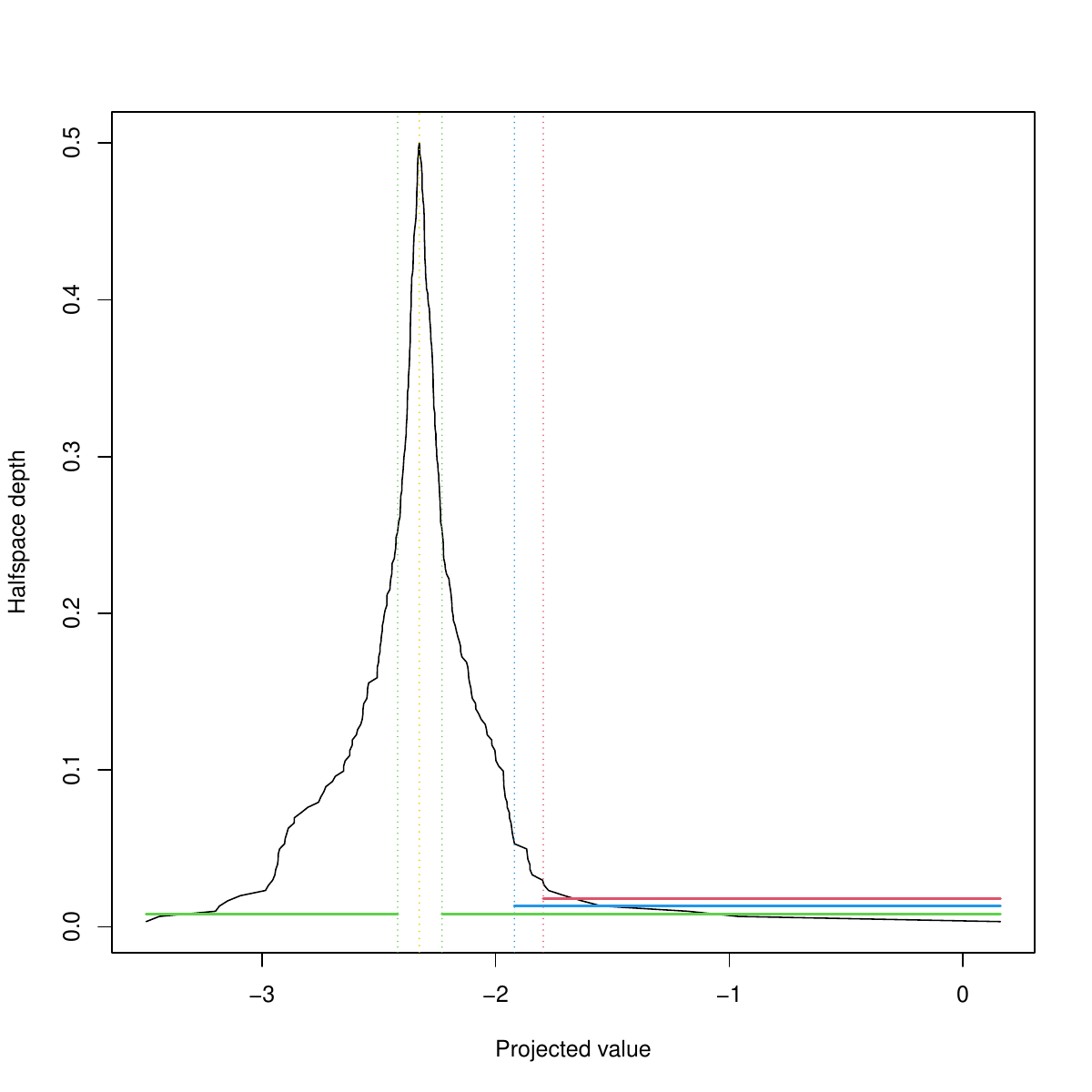}
}
\caption{Product, origin and destination (POD) data. Weights and prices in log scale. POD 30, first row and POD 54, second row. Data depth (on the left), central subspace data depth (on the center) and depth of projected values (on the right). The analysis is performed using halfspace depth.}
\label{sm:figure_pod_halfspace}
\end{figure}
The plots on the left show the depth values of the two data sets in log-scale. The data point of higher depth is shown in yellow, while the $0.5$ points of lower depth are in green. The $50\%$ data points with higher depth are colored in gray scale according to their depth (the darker the higher depth). The plots in the middle show the central subspace depth. Since, in this example, $m=2$ and $q=1$, the remaining dimension ($p=m-q$) is univariate and we can distinguish between lower and upper quantiles. The central straight line is in yellow, the points with quantile order below $0.25$ and between $0.75$ and $0.95$ are in green. Furthermore, the outliers with quantiles order between $0.95$ and $0.975$ are in blue while the outliers with order higher than $0.975$ are in red. The plots on the right shows the depth values of the projected data set and the corresponding quantiles. Both data sets have a strong center with $50\%$ of the projected data points very close to each other (see the region between the two green quantiles). The left tail is heavier, although POD 54 contains a few points far away in the right tail. Similar results are obtained using the simplicial depth for all the analyzed POD as shown in Figure \ref{sm:figure_pod_simplicial}. Finally, the analysis performed on the fishery data set using the simplicial depth is reported in Figure \ref{sm:figure_fishery_simplicial}. 

\begin{figure}
\centering{  
  \includegraphics[width=0.32\textwidth]{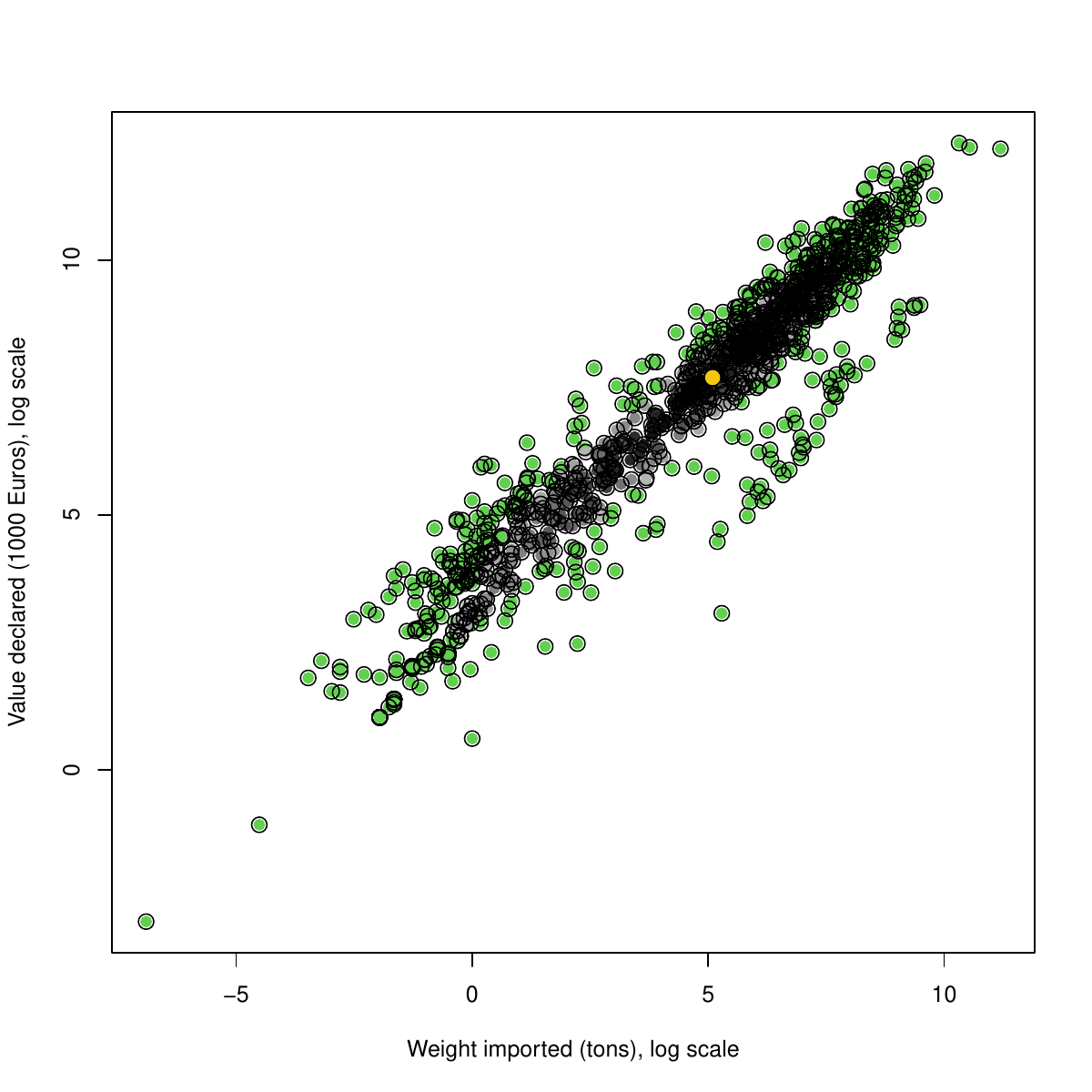}
  \includegraphics[width=0.32\textwidth]{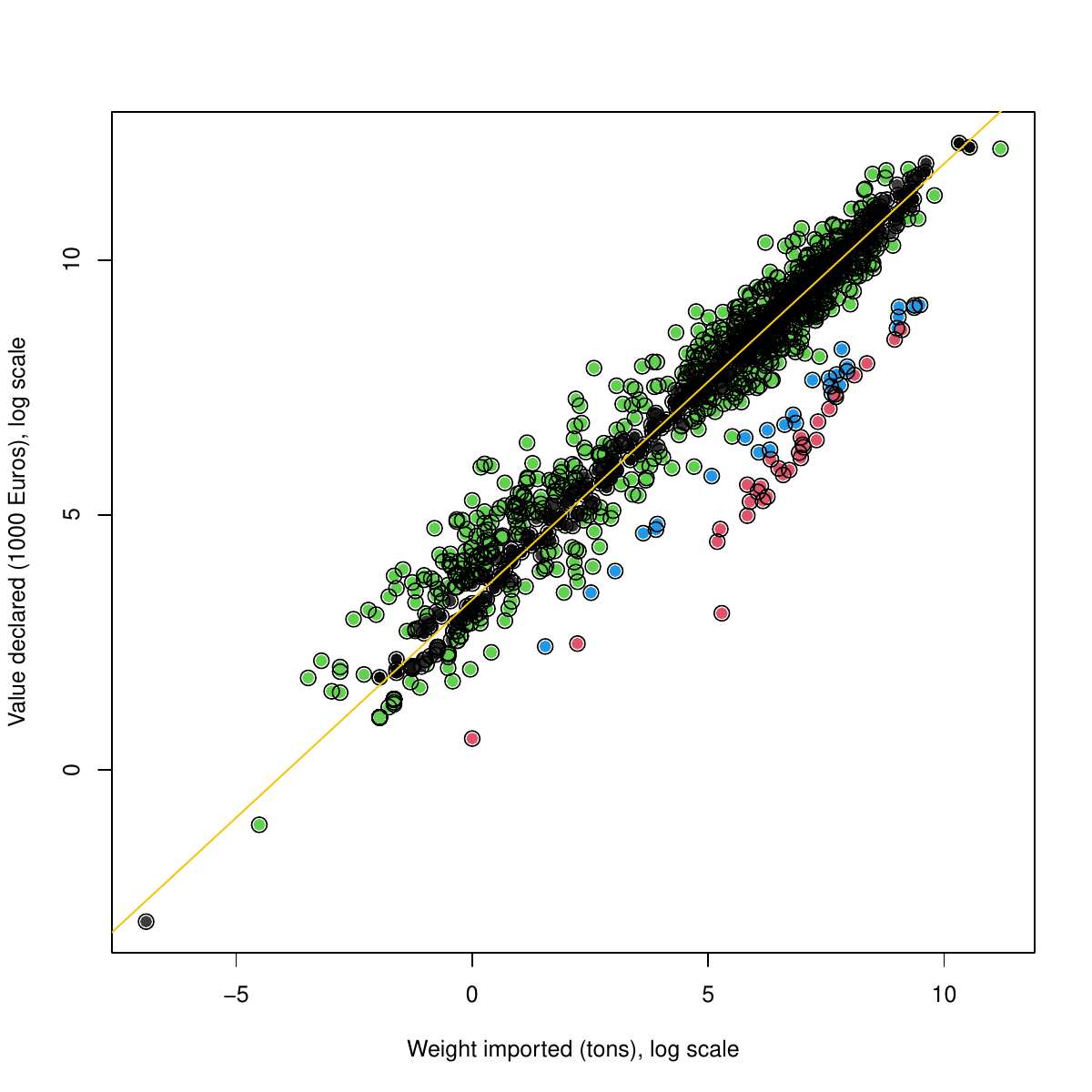}
  \includegraphics[width=0.32\textwidth]{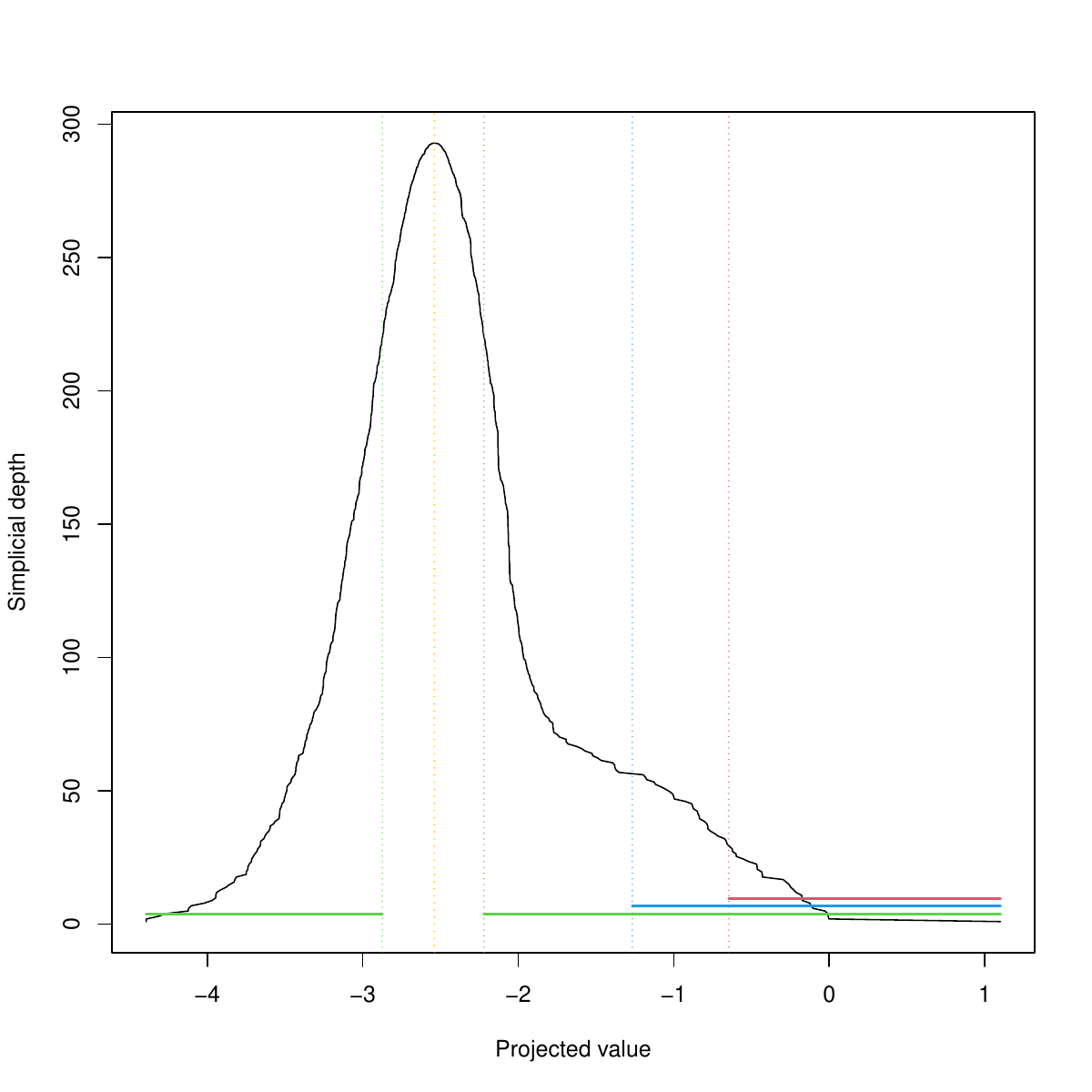}
  \includegraphics[width=0.32\textwidth]{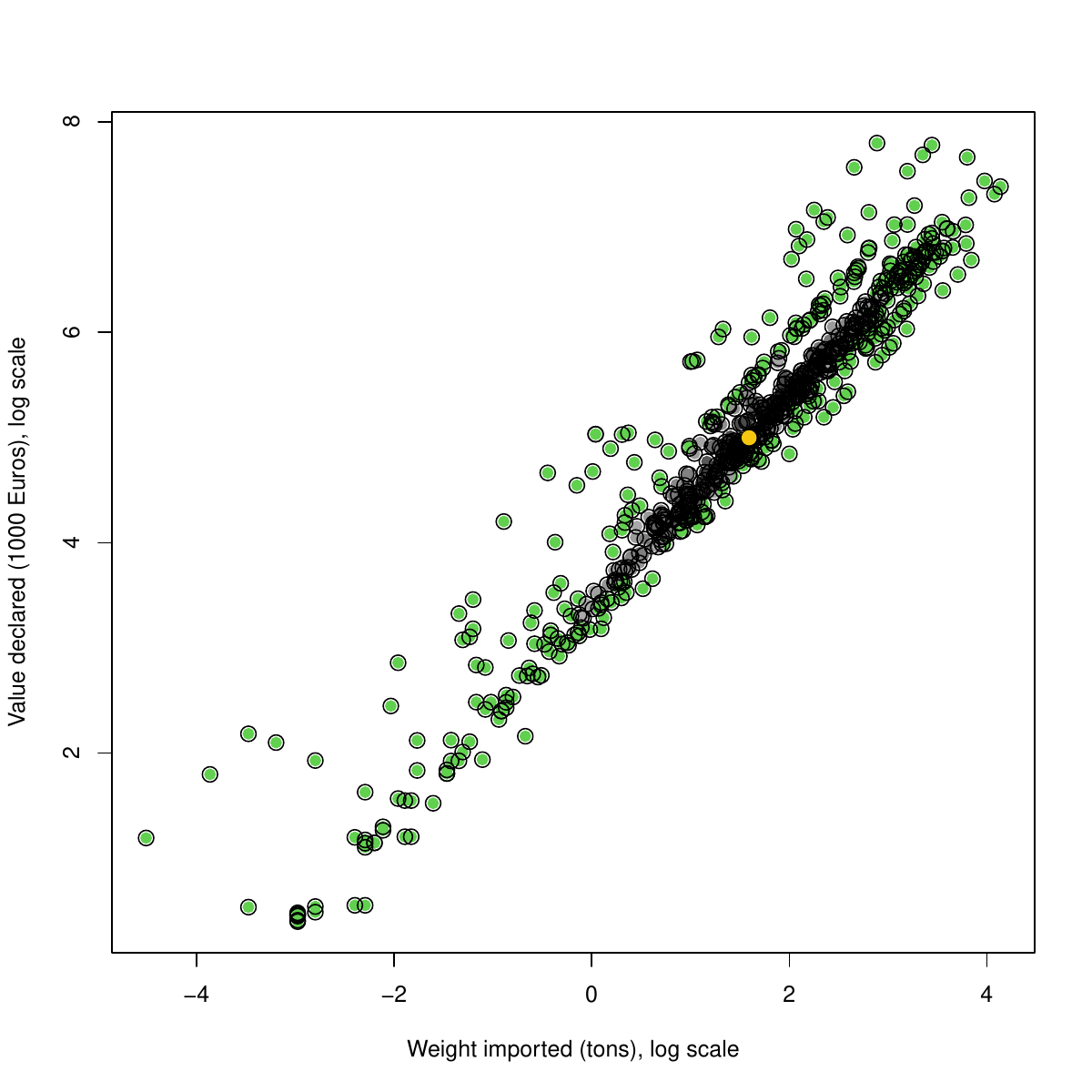}
  \includegraphics[width=0.32\textwidth]{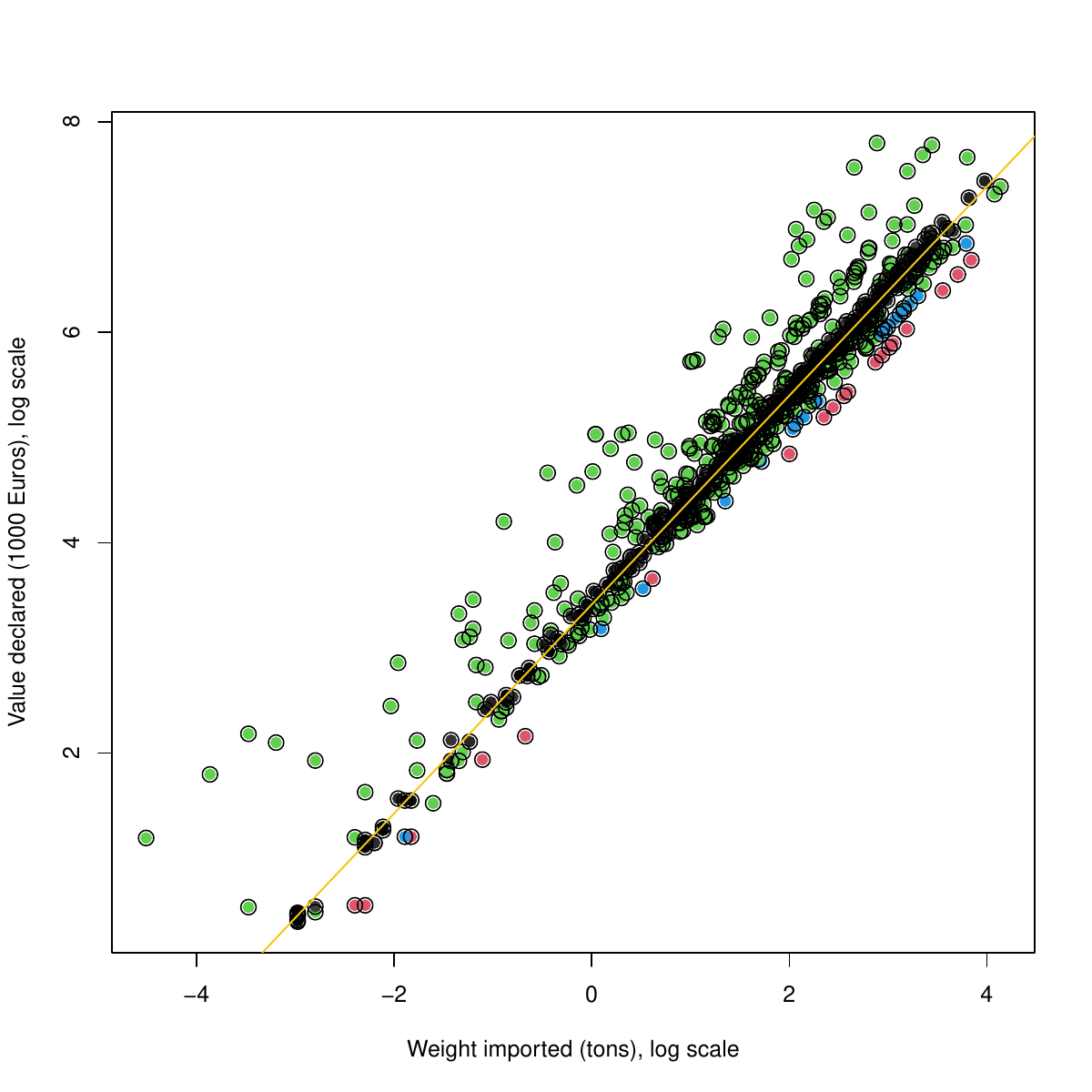}
  \includegraphics[width=0.32\textwidth]{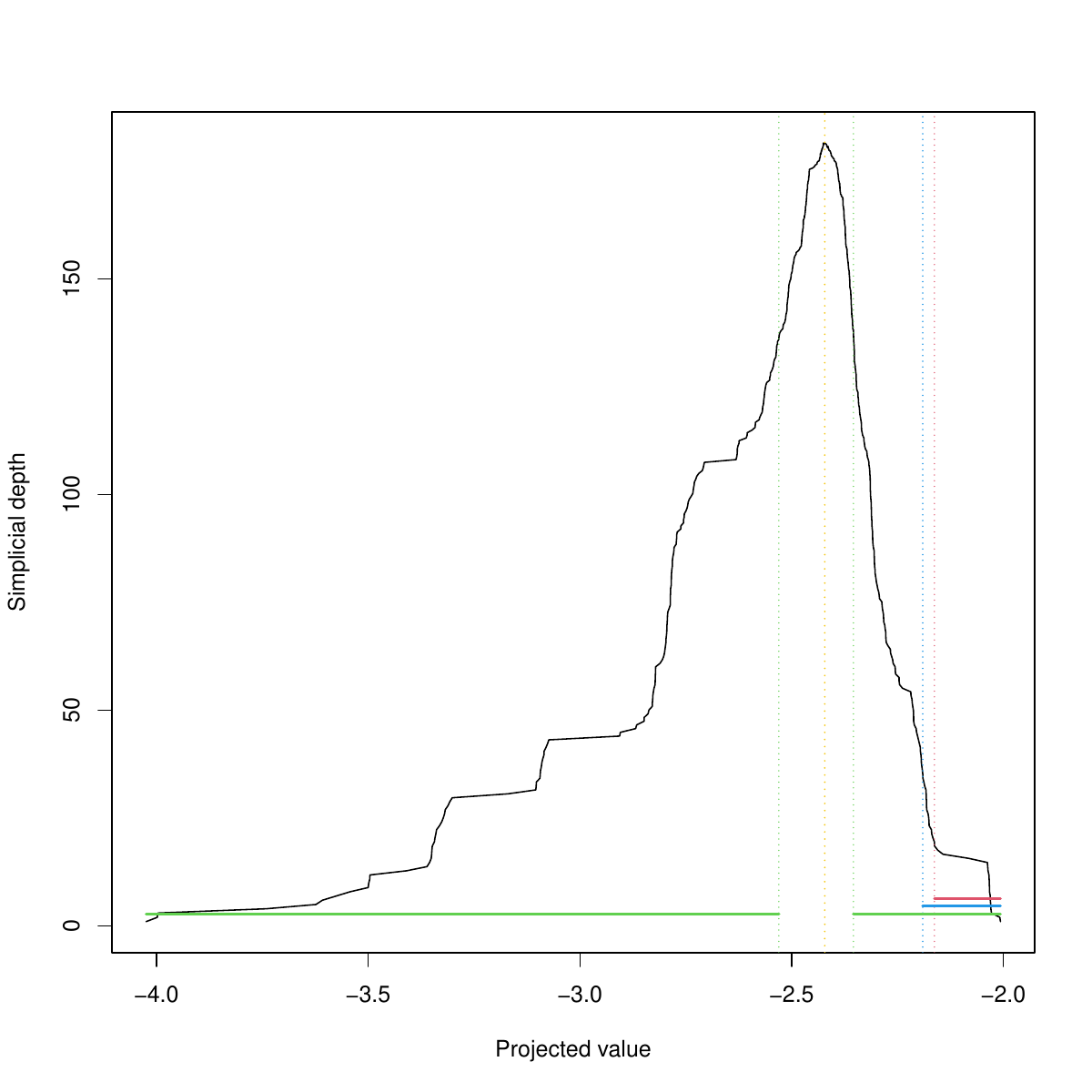}
  \includegraphics[width=0.32\textwidth]{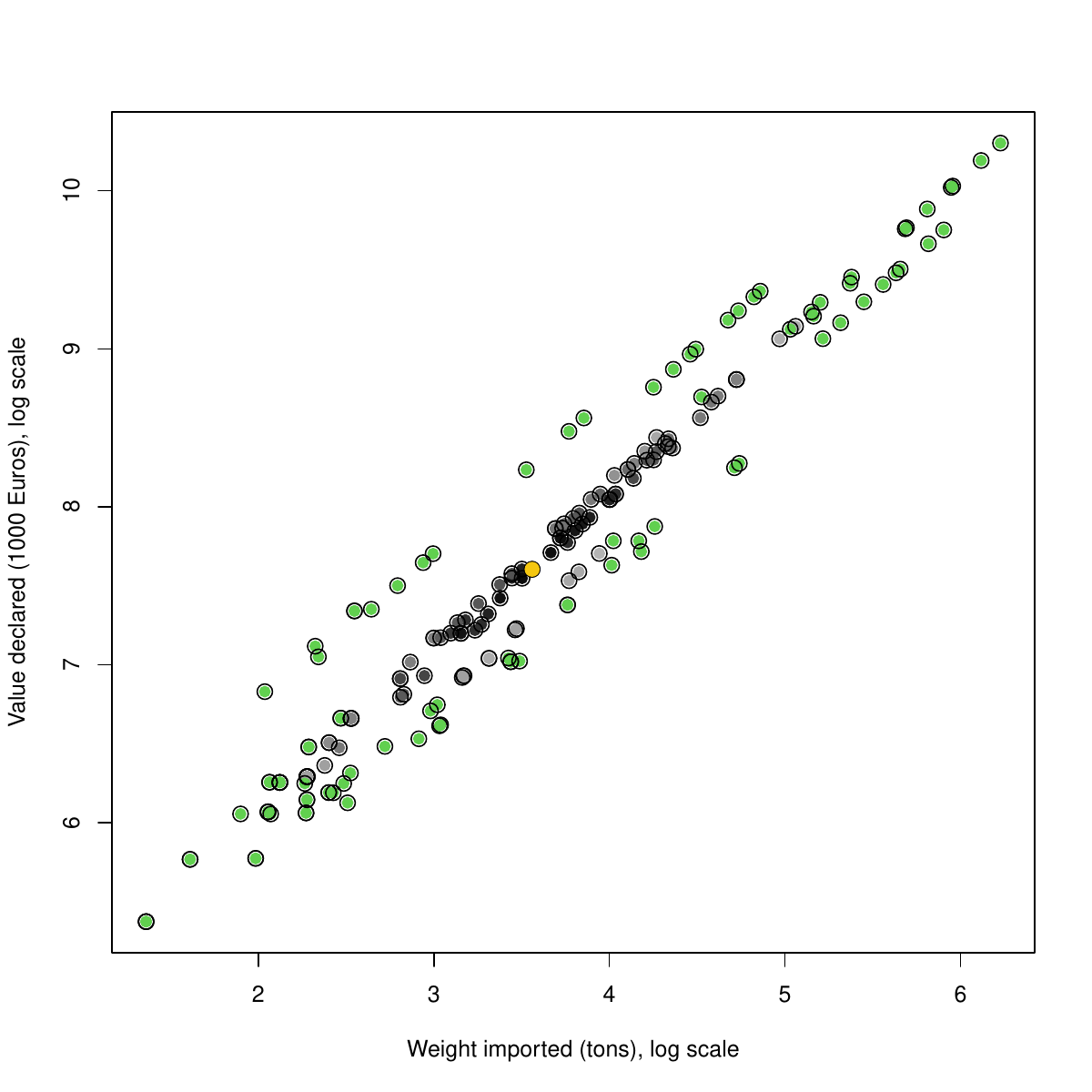}
  \includegraphics[width=0.32\textwidth]{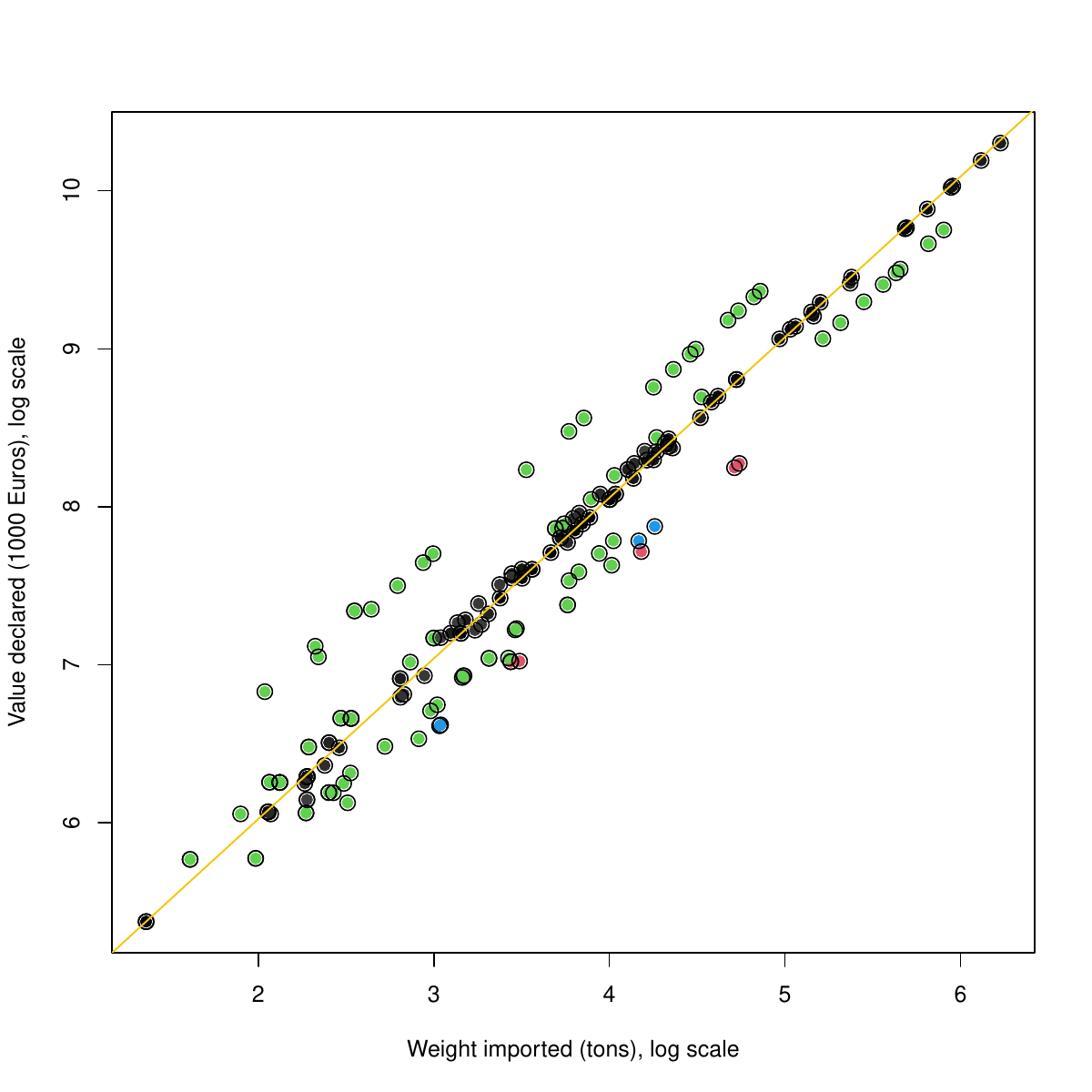}
  \includegraphics[width=0.32\textwidth]{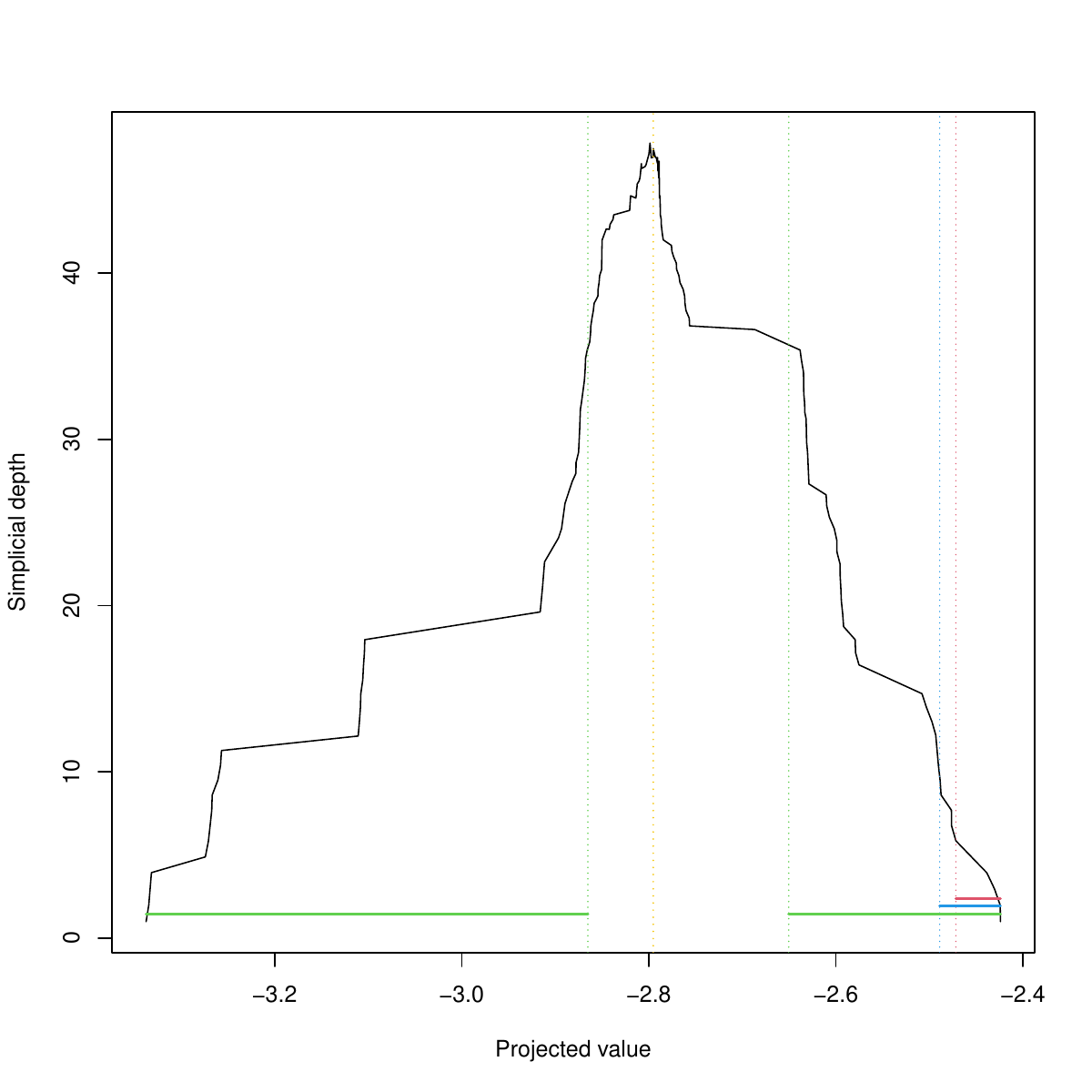}
  \includegraphics[width=0.32\textwidth]{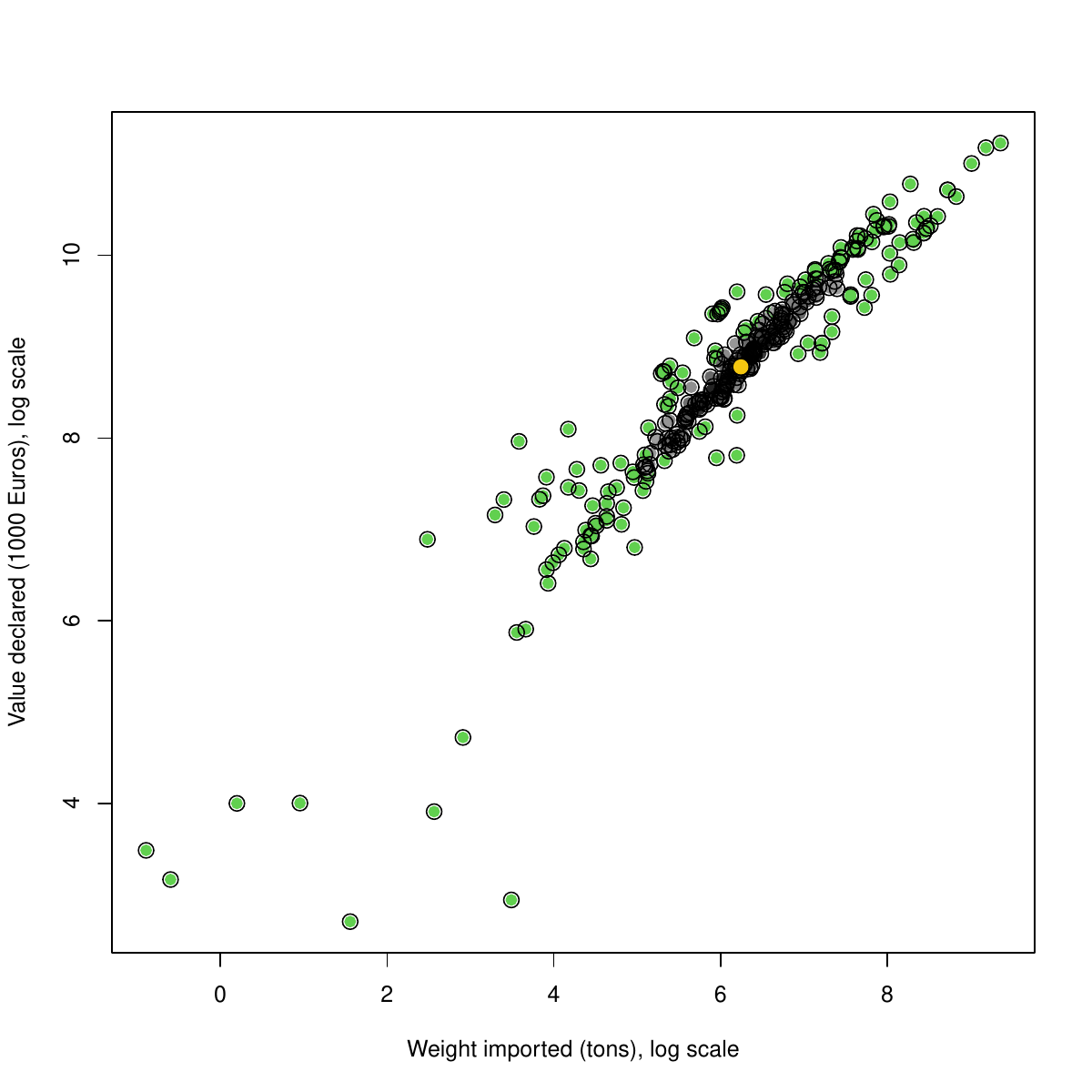}
  \includegraphics[width=0.32\textwidth]{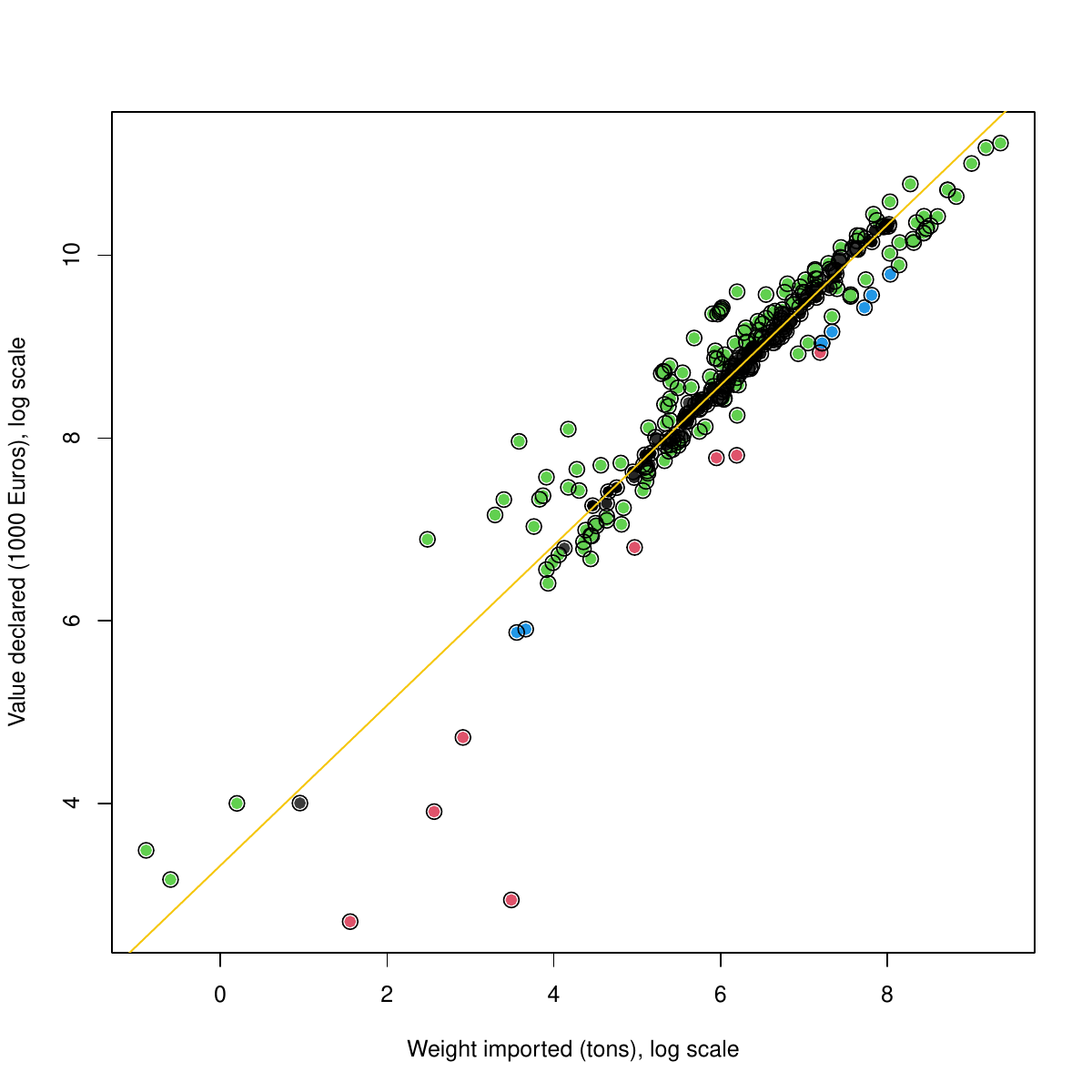}
  \includegraphics[width=0.32\textwidth]{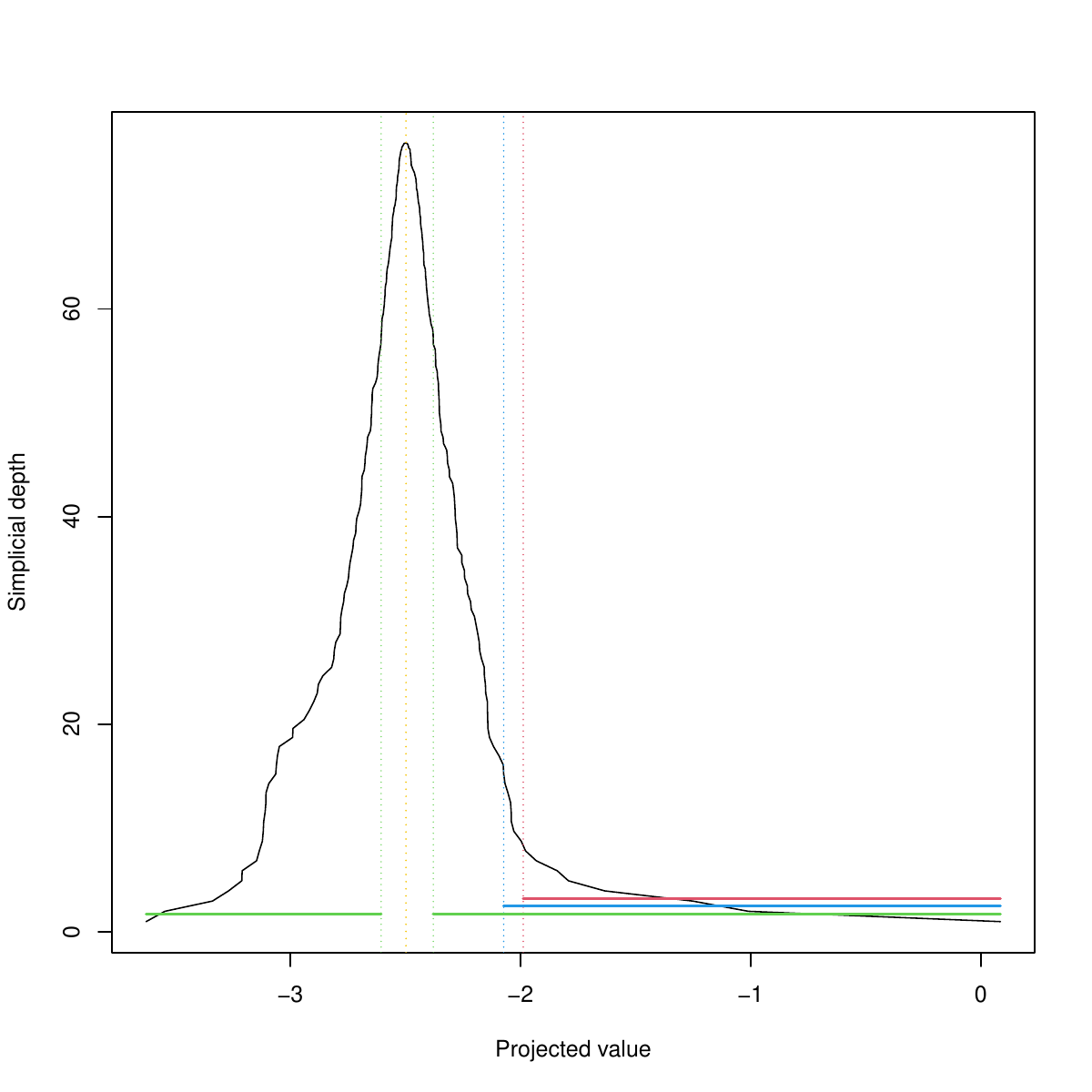}
}
\caption{Product, origin and destination (POD) data. Weights and prices in log scale. POD 19, first row, POD 30, second row, POD 33 third row and POD 54 last row. Data depth (on the left), central subspace data depth (on the center) and depth of projected values (on the right). The analysis is performed using simplicial depth.}
\label{sm:figure_pod_simplicial}
\end{figure}

\begin{figure}
\centering{  
  \includegraphics[width=0.49\textwidth]{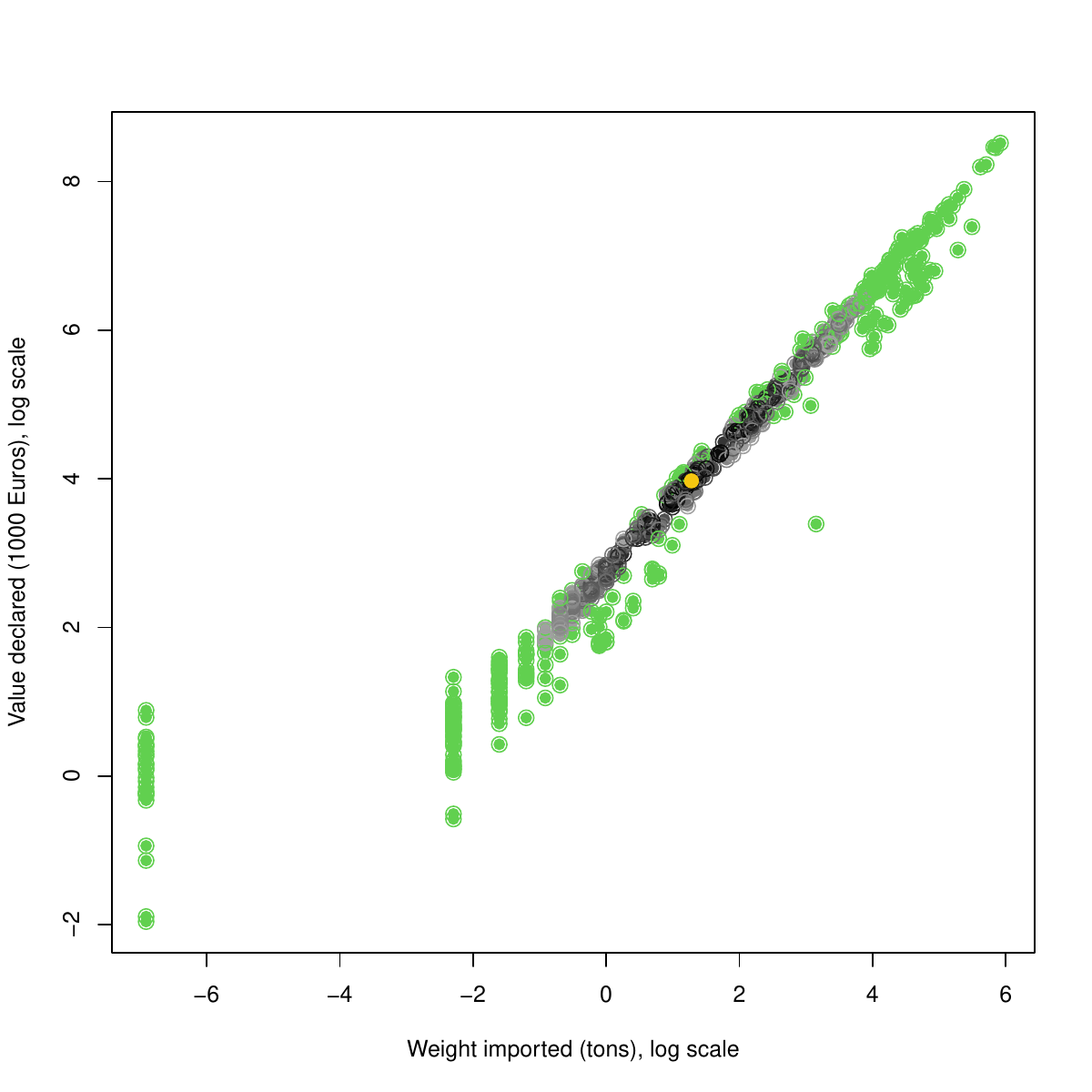}
  \includegraphics[width=0.49\textwidth]{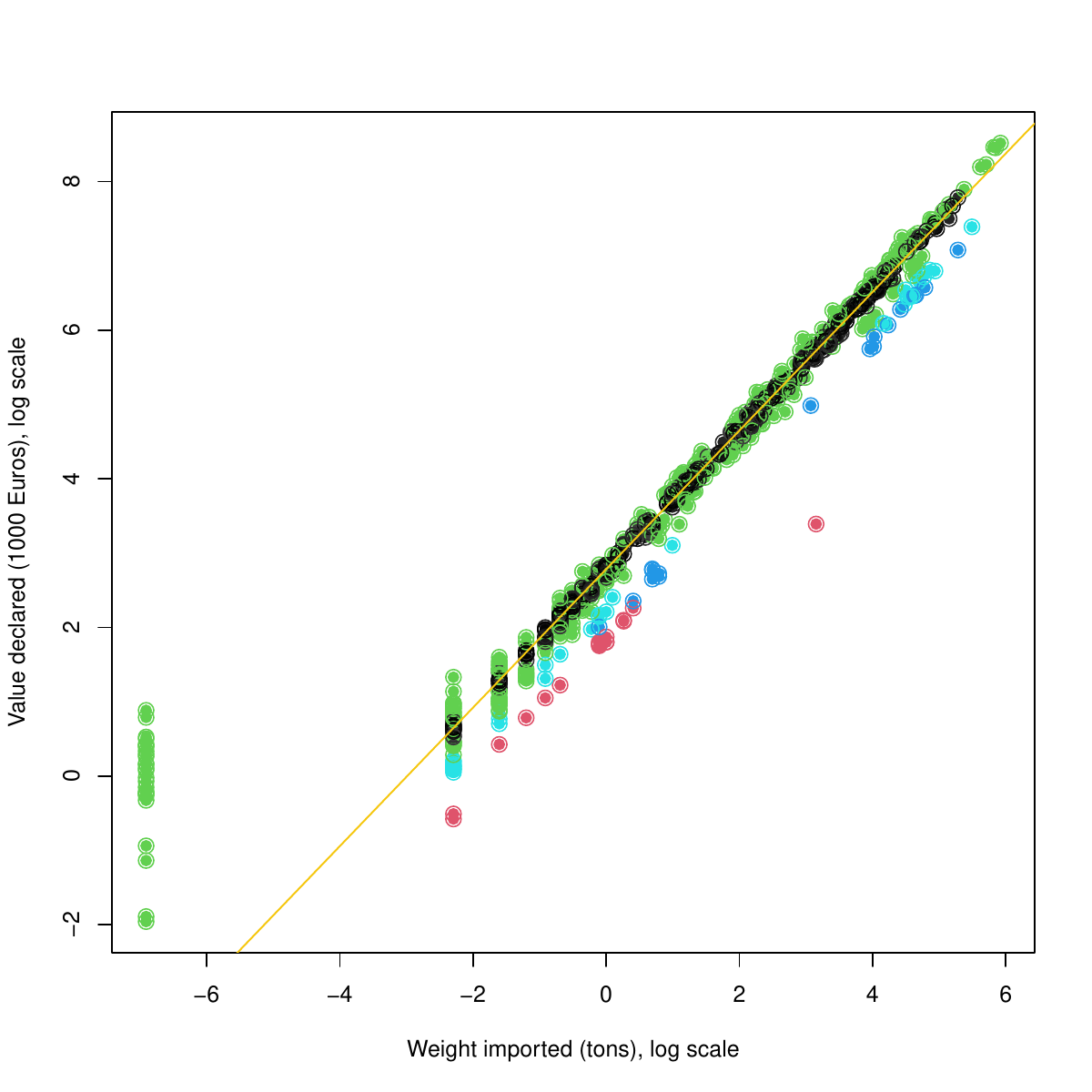}
  \includegraphics[width=0.49\textwidth]{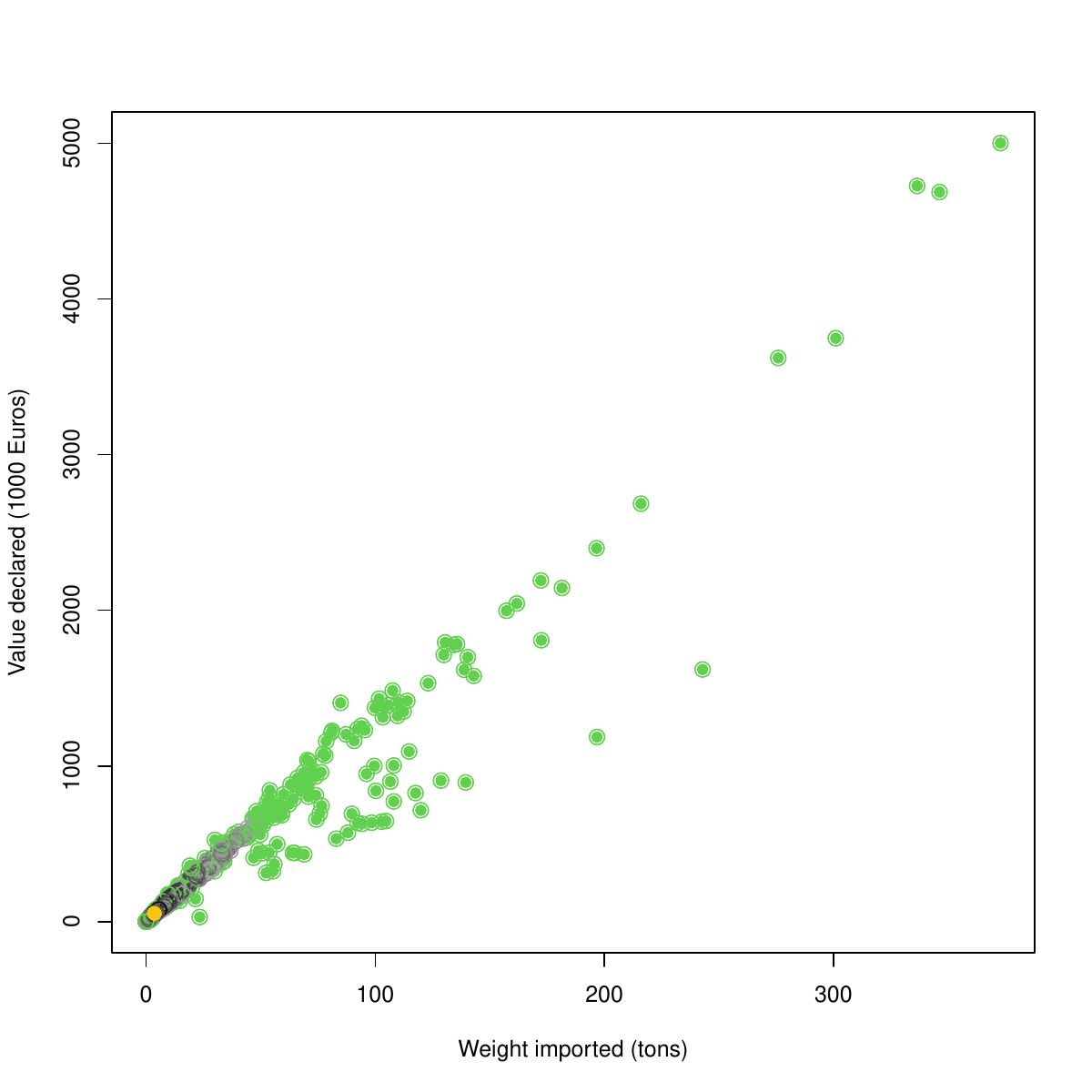}
  \includegraphics[width=0.49\textwidth]{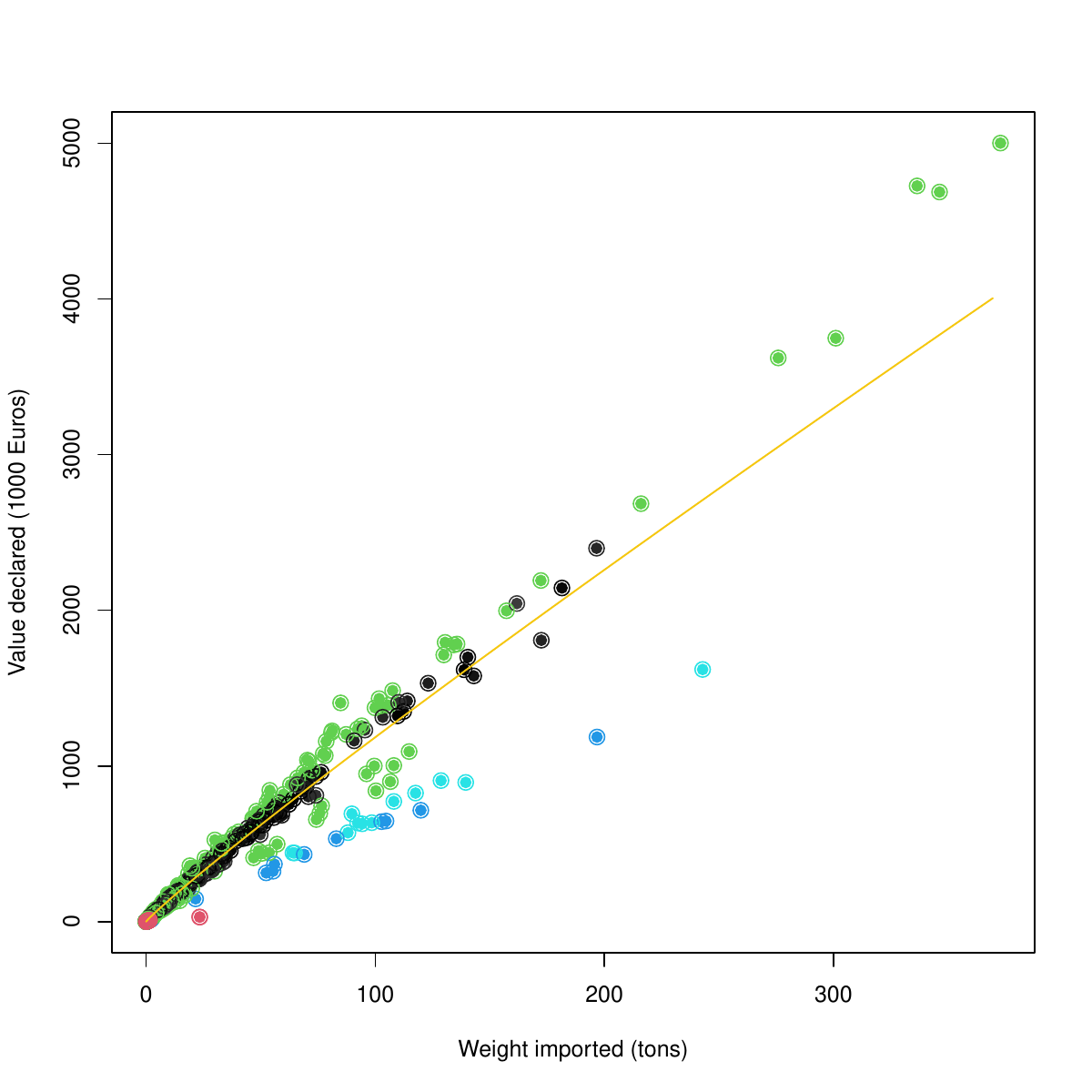}
}
\caption{Fishery data set. In the first row weights and prices are reported in log scale. In the second row the results are back-transformed in the original scale. Data depth (on the left) and central subspace data depth (on the right). The analysis is performed using simplicial depth.}
\label{sm:figure_fishery_simplicial}
\end{figure}

\end{document}